\documentclass[12pt,a4paper]{article}
\usepackage[T1]{fontenc}
\usepackage[utf8]{inputenc}
\usepackage{amsmath}
\usepackage{amsfonts}
\usepackage{amssymb}
\usepackage{setspace}
\usepackage{amsthm}
\usepackage{rotating}
\usepackage[a4paper, left=2.5cm, right=2.5cm, top=2.8cm, bottom=2.8cm]{geometry}
\usepackage{float}
\usepackage{arydshln}
\usepackage{graphicx}
\usepackage[usenames, dvipsnames]{color}
 \usepackage{tabu}
\usepackage{subfigure,bm}
\usepackage[square,sort,comma,numbers]{natbib}
\usepackage{listings} 
\usepackage{color}
\usepackage{hyperref}
\usepackage{algorithm,algpseudocode}
\usepackage{caption}

\usepackage{marginnote} 

\usepackage[usenames, dvipsnames]{xcolor}
\definecolor{darkblue}{HTML}{2C326E}
\usepackage{marginnote} 
\usepackage{sectionbox} 
\definecolor{sectboxrulecol}{rgb}{0,0,0.5}
\definecolor{sectboxfillcol}{rgb}{0.95,0.95,1}

\definecolor{subsectboxrulecol}{rgb}{0,0.5,0}
\definecolor{subsectboxfillcol}{rgb}{0.95,1,0.95}

\definecolor{subsubsectboxrulecol}{rgb}{0.5,0.5,0}
\definecolor{subsubsectboxfillcol}{rgb}{1,1,0.9}

\shadowsectionbox 

\newtheorem{proposition}{Proposition}
\newtheorem{lemma}{Lemma}
\newtheorem{corollary}{Corollary}
\newtheorem{definition}{Definition}
\theoremstyle{remark}
\newtheorem{example}{Example}
\newtheorem{remark}{Remark}

\lstdefinestyle{mystyle}{
backgroundcolor=\color{backcolour},
commentstyle=\color{codegreen},
keywordstyle=\color{magenta},
basicstyle=\tiny,
breaklines=true,
keepspaces=true,
numbers=left,
numbersep=5pt,
}
\definecolor{codegreen}{rgb}{0,0.6,0}
\definecolor{backcolour}{rgb}{0.95,0.95,0.92}

\usepackage{amsopn}

\graphicspath{{figures/}} 

\lstdefinestyle{mystyle}{
backgroundcolor=\color{backcolour},
commentstyle=\color{codegreen},
keywordstyle=\color{magenta},
basicstyle=\small,
breaklines=true,
keepspaces=true,
numbers=left,
numbersep=5pt,
}
\definecolor{codegreen}{rgb}{0,0.6,0}
\definecolor{backcolour}{rgb}{0.95,0.95,0.92}

\usepackage{booktabs} 

\begin{document}
\begin{center}
\vspace*{2cm} \noindent {\bf \large The $OB$-splines -- efficient orthonormalization of the $B$-splines}\\
\vspace{1cm} \noindent {\sc Xijia Liu$^\dagger$, Hiba Nassar$^\ddagger$, Krzysztof Podg\'{o}rski}$^\ddagger$\\
\vspace{1cm}
{\it \footnotesize 
$^\dagger$Department of Mathematics and Mathematical Statistics, Ume{\aa}   University,\\
$^\ddagger$Department of Statistics, Lund University}
\end{center}


\begin{abstract}
 A new efficient orthogonalization of the $B$-spline basis is proposed and contrasted with some previous orthogonalized methods. 
The resulting orthogonal basis of splines is best visualized as a {\it net} of functions rather than a sequence of them. 
For this reason, the basis is referred to as a {\em splinet}.
The splinets feature clear advantages over other spline bases.
They efficiently exploit `near-orthogonalization' featured by the $B$-splines and gains are achieved at two levels:  {\em locality} that is exhibited through small size of the total support of a splinet and {\em computational efficiency} that follows from a small number of orthogonalization procedures needed to be performed on the $B$-splines to achieve orthogonality. 
These efficiencies are formally proven by showing the asymptotic rates with respect to the number of elements in a splinet. 
The natural symmetry of the $B$-splines in the case of the equally spaced knots is  preserved in the splinets, while quasi-symmetrical features are also seen for the case of arbitrarily spaced knots. 
\end{abstract}

{\small
{\bf keywords:} Basis functions, $B$-splines, orthogonalization, splinets.\\

\vspace{-.09cm}
{\bf AMS:} 65D07, 65F25
}

\section{Introduction}
\label{sec:intro}
The $B$-splines are the most popular linear bases of splines \cite{Boor1978APG, schumaker2007spline}. One of their main advantages lies in their locality, i.e. their supports are local and controlled by distribution of the knots in the interval over which the splines are constructed. 
If the knots are equally spaced, then the splines are distributed in a `uniform' fashion over the entire range.
Unfortunately, the $B$-splines bases are not orthogonal, which adds a computational burden when they are used to decompose a function. 
Since any basis can be orthogonalized, it is to the point to consider orthonormalization of the $B$-splines. 
However, orthogonalization can be performed in many different ways leading to essentially different orthogonal spline (O-splines) bases, see \cite{nguyen2015construction,goodman2003class,cho2005class}.
In this work, we discuss orthogonalization methods that are natural for the $B$-splines. 
As our main contribution, we propose orthogonalization which is argued to be the most appropriate since it, firstly, preserves most from the original structure of the $B$-splines and, secondly, obtains computational efficiency both in the basis element evaluations and in the spectral decomposition of a functional signal.

Although fundamentally different, the proposed method was inspired by one-sided and two-sided orthogonalization discussed in \cite{mason1993orthogonal}.
Since the constructed basis spreads a net of splines rather  than a sequence of them, we coin the term {\it splinet} when referring to such a base.  
It is formally shown that the splinets, similarly to the $B$-splines, feature locality with a small size of the total support. 
If the number of knots over which the splines of a given order are considered is $n$, then
the total support size of the $B$-splines is on the order $O(1)$ with respect to $n$, while the corresponding splinet has the total support size on the order of $\log n$ which is only slightly bigger.
On the other hand, the previously discussed orthogonalized bases have the total support size of the order $O(n)$, where $n$ stands for the number of knots which is also the number of basis functions (up to a constant).
Moreover, if one allows for negligible errors in the orthonormalization, then the total support size no longer will depend on $n$, i.e. becomes constant and thus achieving the rate of the original $B$-splines.

The main goal of this work is to deliver efficient orthonormalization of the $B$-splines $B_j$, $j=1,\dots, d$, i.e. we propose a $d\times d$ matrix $\mathbf P$ such that the splines
$$
OB_i=
\sum_{j=1}^d P_{ji}B_j
, ~~ i=1,\dots, d
$$
constitute a convenient orthogonal and normalized basis.
However, the construction can be viewed as an effective diagonalization method for an arbitrary positive definite matrix $\mathbf H$ with non-zero terms only on the $2k-1$-diagonals symmetrically placed above and below the main diagonal, a so-called band matrix. 
The construction uses a generic algorithm that is independent of the spline context.  
Since this problem is of interest on its own terms, we also present the obtained orthogonalizing transformations. 
For a band matrix $\mathbf H$, the algorithm defines $\mathbf P$ such that
\begin{equation}
\label{eq:diag}
\mathbf I=\mathbf P^\top\mathbf H \mathbf P.
\end{equation}
An example of $\mathbf P$ obtained from the orthonormalization algorithm is shown in \ref{fig:matrices}.
In the matrix interpretation of our approach, the equispaced case corresponds to a symmetric Toeplitz (diagonal-constant) matrix, which  in  \ref{fig:matrices} has been presented for the bandwidth equal to seven.
One observes sparsity of $\mathbf P$, which is one of several benefits of the proposed method. 
Additionally to the sparsity, it is shown that most non-zero terms are negligible as it is also seen in \ref{fig:matrices}.   
\begin{figure}[t!]
\includegraphics[width=0.49\textwidth]{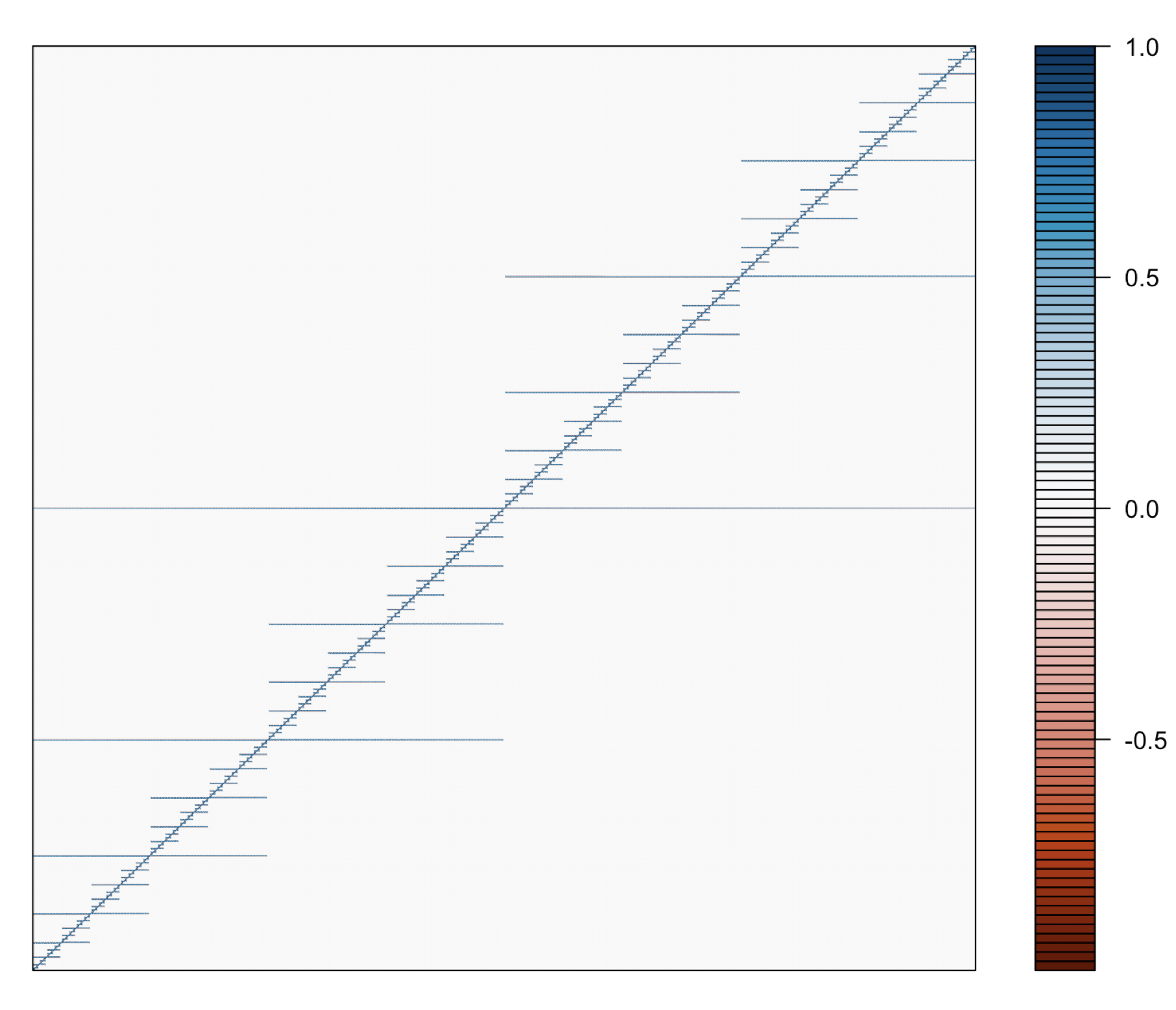}\includegraphics[width=0.49\textwidth]{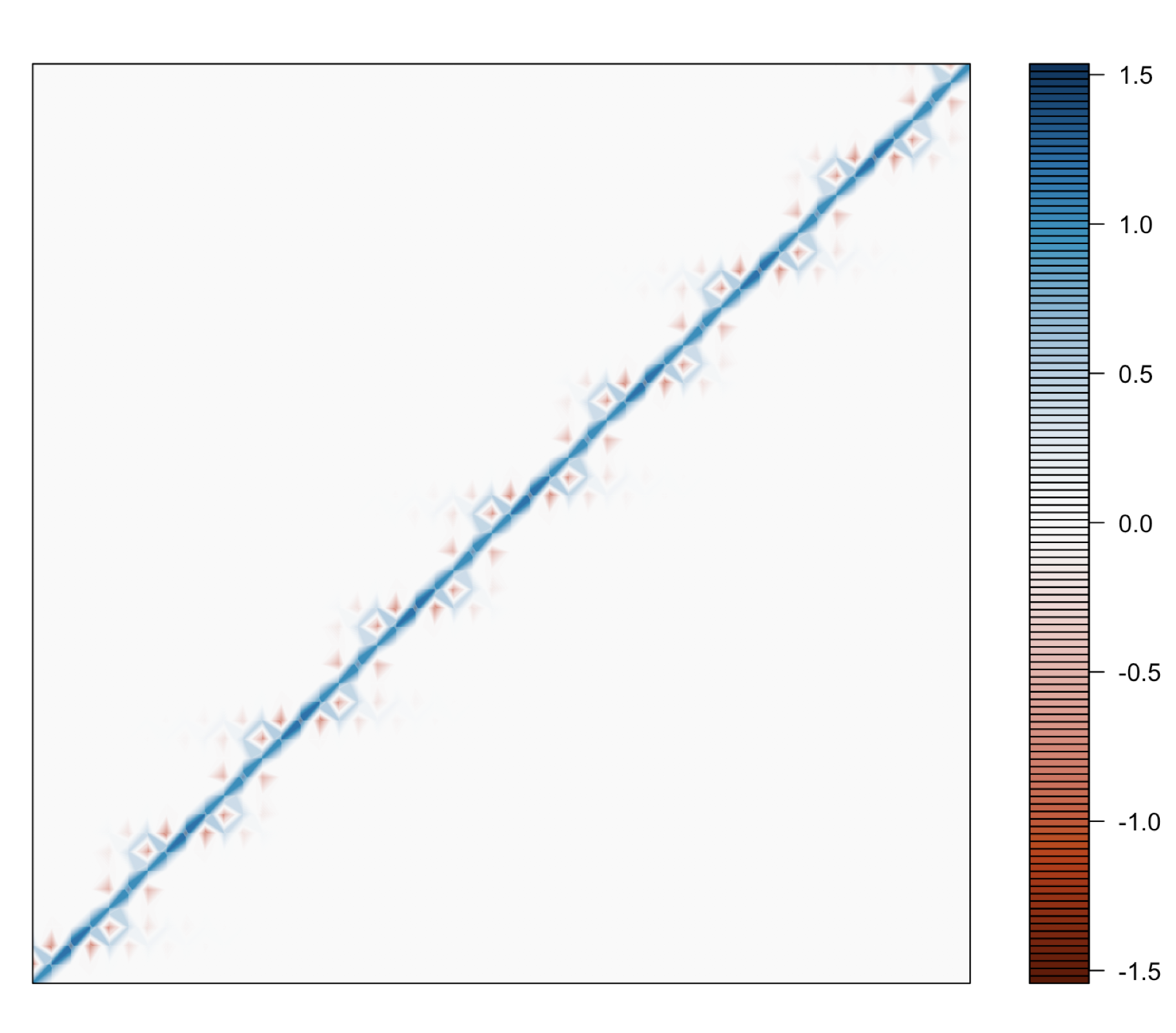}\\
\raisebox{4mm}{\includegraphics[width=0.49\textwidth]{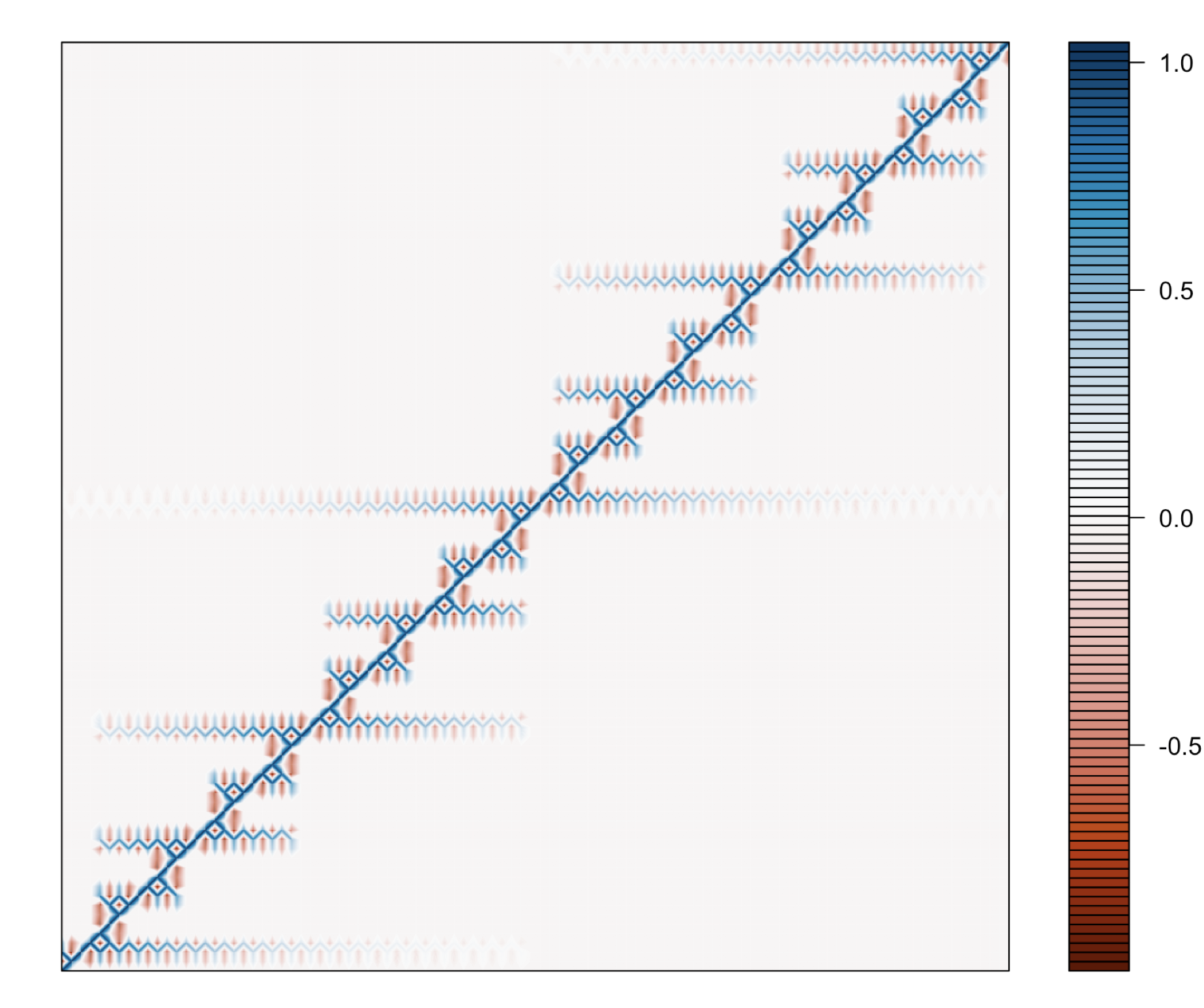}} \includegraphics[width=0.485\textwidth, height=0.44\textwidth]{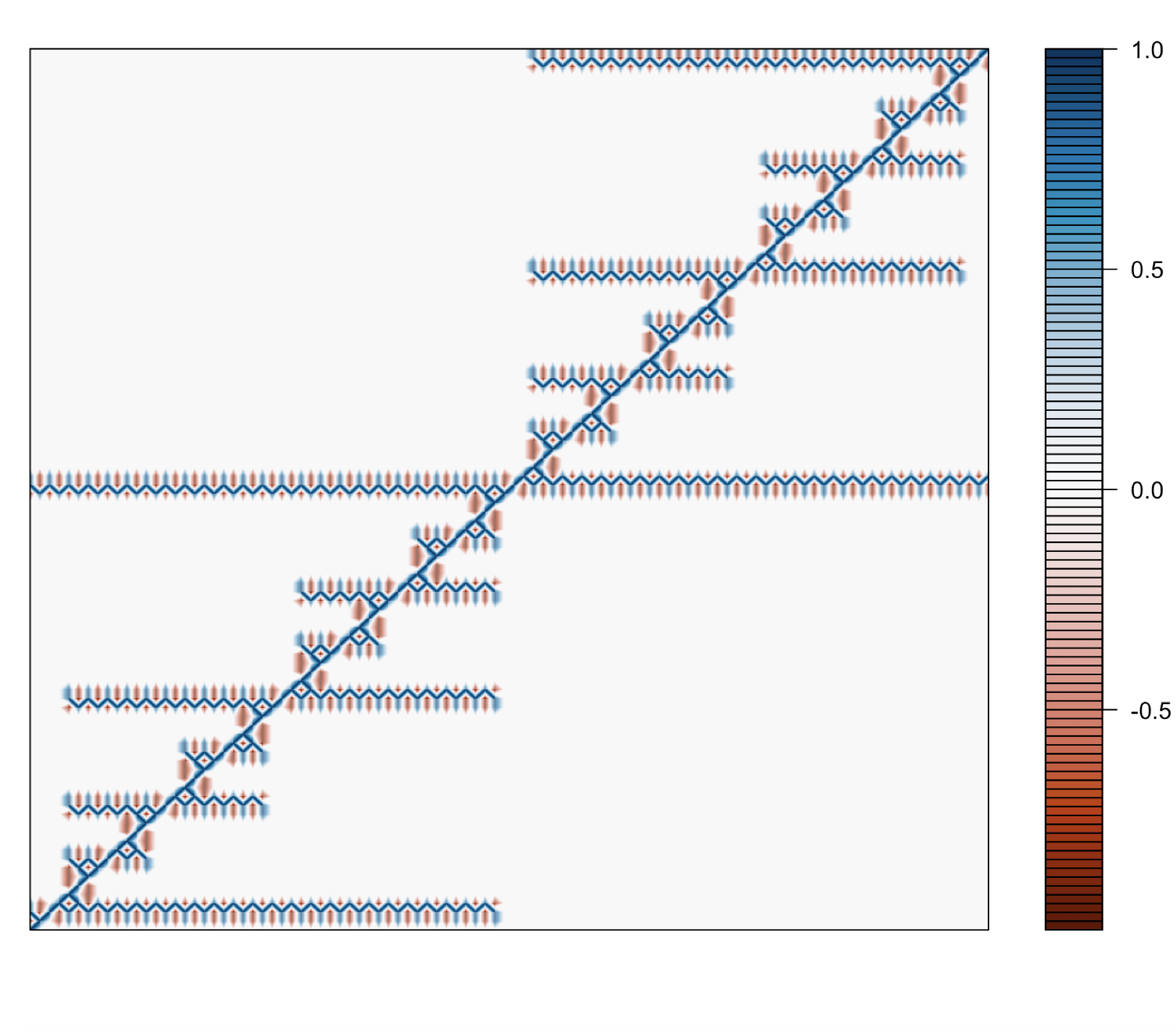}
\caption{A diagonalization matrix $\mathbf P^\top$  (the lower-left corner corresponds to the first entry in  $\mathbf P^\top$) for the case of a symmetric Toeplitz $ 1533 \times 1533$ matrix with $7$ non-zero diagonals. 
The full matrix $\mathbf P$ is graphically presented at the top-left graph, where to show all non-zero entries, they have been raised to power $0.001$. 
That most of the non-zero terms are negligible is seen in the next three pictures where, first, in the upper-right corner, the central $50\times 50$ submatrix of $\mathbf P$ is shown, then in the lower row, we see the central $100\times 100$ sub-matrix in which the original terms have been raised to $0.1$ (left) and to $0.01$ (right).  }
\label{fig:matrices}
\end{figure}
%
%

The organization of the material is as follows. We start in \ref{sec:basics} with a brief account of the $B$-spline bases that establishes notation and recall basic facts. 
Then in \ref{sec:simpleorth} two $B$-spline orthogonalization methods that appeared in the literature are reviewed.
This is followed by \ref{sec:splinets} devoted to the construction of the {\it splinets}, which we deem the central contribution of this work.  
In this section, we also discuss the generic Hilbert space algorithms and relate them to the problem of efficient sparse diagonalization of the band matrices. 
In  \ref{sec:eff}, we present the efficiency of the new method and show an upper bound for non-zero entries away of the diagonal. 
Mathematical proofs are placed in the supplementary material, \ref{sec:pfsal}. 
Additionally, the work is supplemented in \ref{sec:ni} by an introduction to computationally convenient representation of splines that allows to benefit from the properties of the splinets. 
Implemented R-package {\tt splinets} is utilizing these benefits and its numerical fundamentals are also discussed in the supplement. 

\section{Basics on the $B$-splines}
\label{sec:basics}
In functional analysis, it is often desired that functions considered are continuous or even differentiable up to a certain order. For that spline functions are often used. 
For a given set of knots and an order, a spline between two subsequent knots is equal to a polynomial of the given order that is smoothly connected at the knots to the polynomials over the two neighboring intervals. 
The order of smoothness at the knots is equal to the number of derivatives that are continuous at these knots, including the zero order derivative, i.e. the function itself.
A selection of an order that is higher than zero makes splines a smoother alternative to the piecewise constant functions.
The splines of a given order and over a given set knots form a finite dimensional space and one can consider a suitable basis of functions that spans it.  
There are different possible choices but the most popular are $B$-splines. 
This section unifies the notation and provides the most fundamental facts about the $B$-splines  used later throughout the paper. 
Since all results are either simple or well-known (although spread throughout vast literature), the arguments provided are omitted or sketchy.

\subsection{Knots and boundary conditions}
\label{subsec:knbc}
The default domain for splines is $(0,1]$, although it is not truly a restriction since for any interval $(a,b]$, we have a natural transformation 
\begin{equation}
\label{eq:abtransform}
(T_{a}^{b}x)(t)=x\left((t-a)/(b-a)\right)
\end{equation}
between splines given on $(0,1]$  to the ones given on $(a,b]$. 
The set of knots is represented as a vector $\boldsymbol \xi$ of ordered values. 

In the literature, there are considered two alternative but in a certain sense equivalent requirements on the behavior of a spline at the endpoints of its range.
In the first one, no boundary conditions are imposed. 
In which the case, for proper handling recurrent formulas defining the $B$-splines, the knots $\boldsymbol \xi$ need to be extended by adding some initial knots located at zero (the initial endpoint) and the same number of knots located at one (the terminal endpoint). 
Those knots are called superfluous and distinguished from the internal knots that are assumed to be different one from another. 
The number of superfluous knots is equal to the $B$-spline order that we aim at.
The number  of internal knots is denoted by $n$. 
If the order  is $K$, then the total number of the knots is $n+2K+2$ and thus 
\begin{equation}
\label{eq:knots}
\boldsymbol \xi =(\xi_{0},\dots, \xi_{K}, \xi_{K+1},\dots, \xi_{K+n},\xi_{K+n+1},\dots,\xi_{2K+n+1}),
\end{equation}
where $\xi_{0}=\dots =\xi_{K}=0$ and $\xi_{K+n+1}=\dots =\xi_{2K+n+1}=1$ are superfluous knots and endpoints and the internal knots are $\xi_{K+1}<\dots < \xi_{K+n}$. 
For a given $k\le K$, the $l$th $B$-spline of order $k$ built for knots $\boldsymbol \xi$ is denoted as $B_{l,k}^{\boldsymbol \xi}$, where $l=0,\dots, n+2K$, see the next subsection and  \ref{prop:equivalence} for details of the definition. 
The first $K-k$ and the last $K-k$ of all $B$-splines are called superfluous since they are essentially defined over $(\xi_{i},\xi_{i+1}]$
 and $(\xi_{K+k+i+n+1},\xi_{K+k+i+n+2}]$, $i=0,\dots, K-k-1$ and these intervals are empty sets.
These superfluous $B$ splines are introduced only for the sake of a convenient formulation of the recurrent formulas. 
Whenever $k$ is omitted in the notation it is assumed to be equal to $K$, i.e. the highest order of considered splines. 
We note that the dimension of the linear space of splines built upon $\boldsymbol \xi$ is $(n+1)(K+1)-Kn=K+n+1$ and the count for the dimension of splines goes as follows: $(n+1)(K+1)$ because of the number of coefficient of $n+1$ polynomials and $Kn$ stands for the number of the conditions making $K-1$ values derivatives and the value of the function (the $0$-order derivative) equal at the $n$ internal knots. 
\begin{figure}[t!]
\begin{center}
\includegraphics[width=0.5\textwidth]{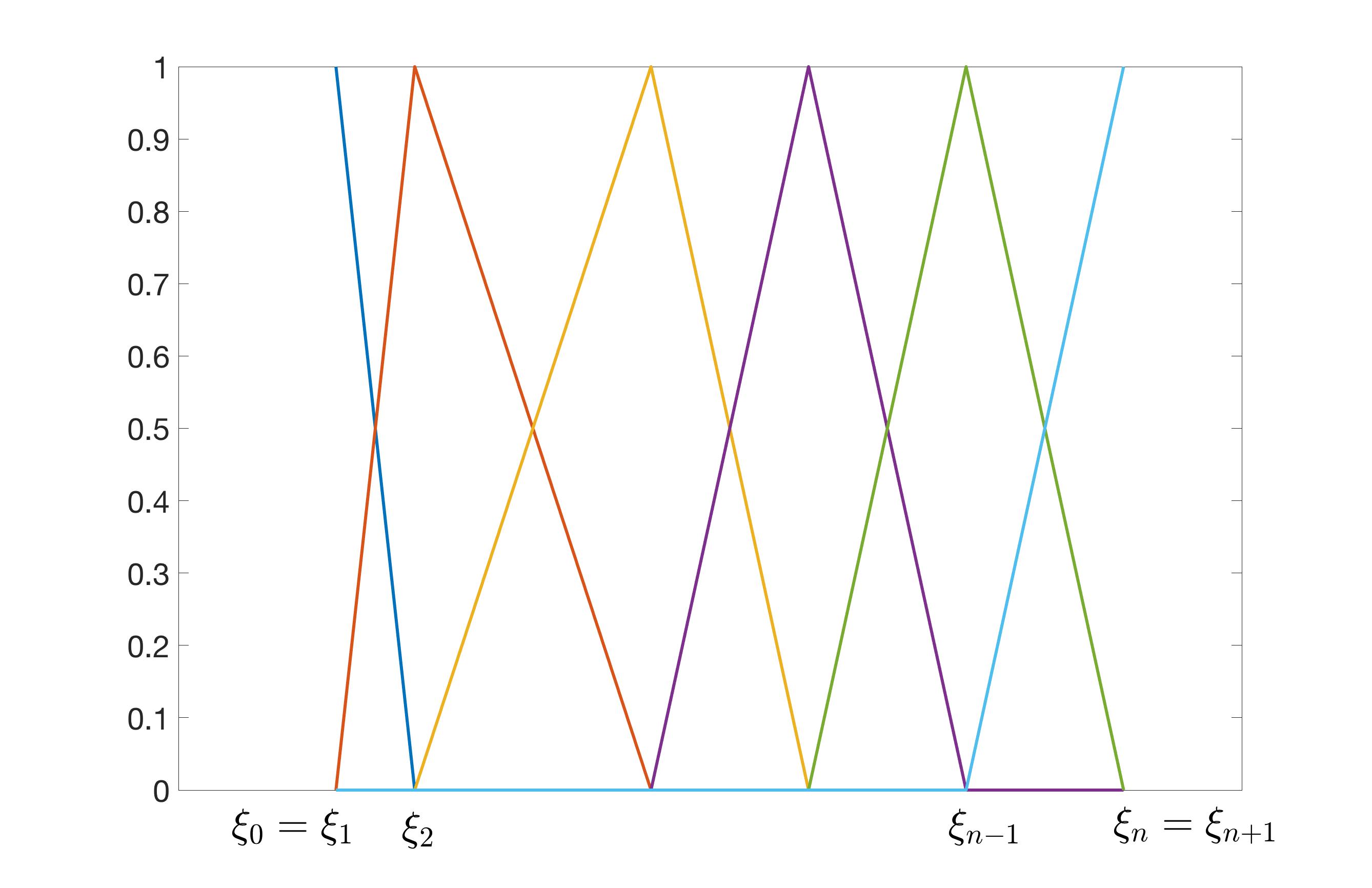}\vspace{-3mm}\hspace{-3mm}
\includegraphics[width=0.5\textwidth]{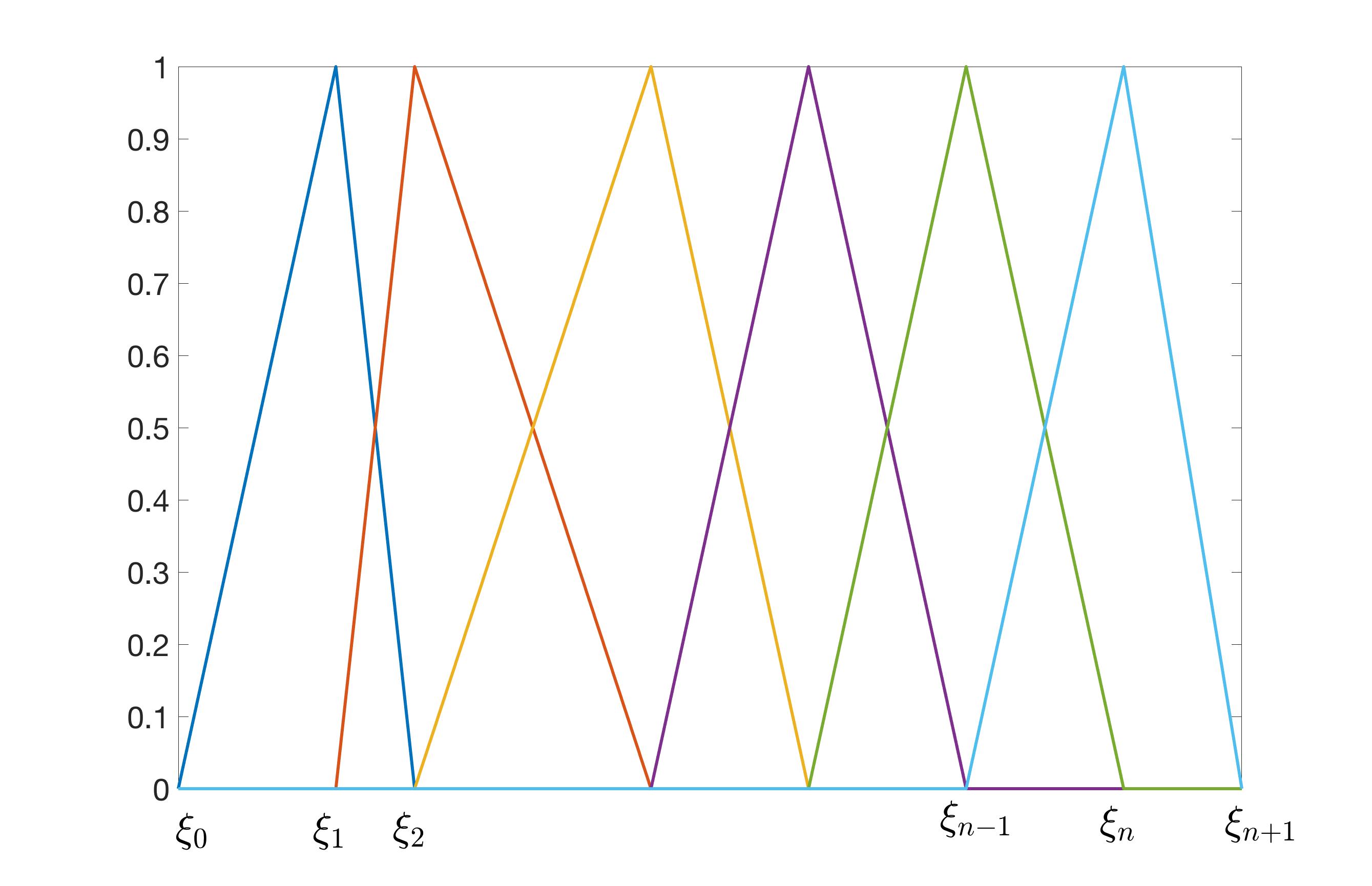}\\
\includegraphics[width=0.5\textwidth]{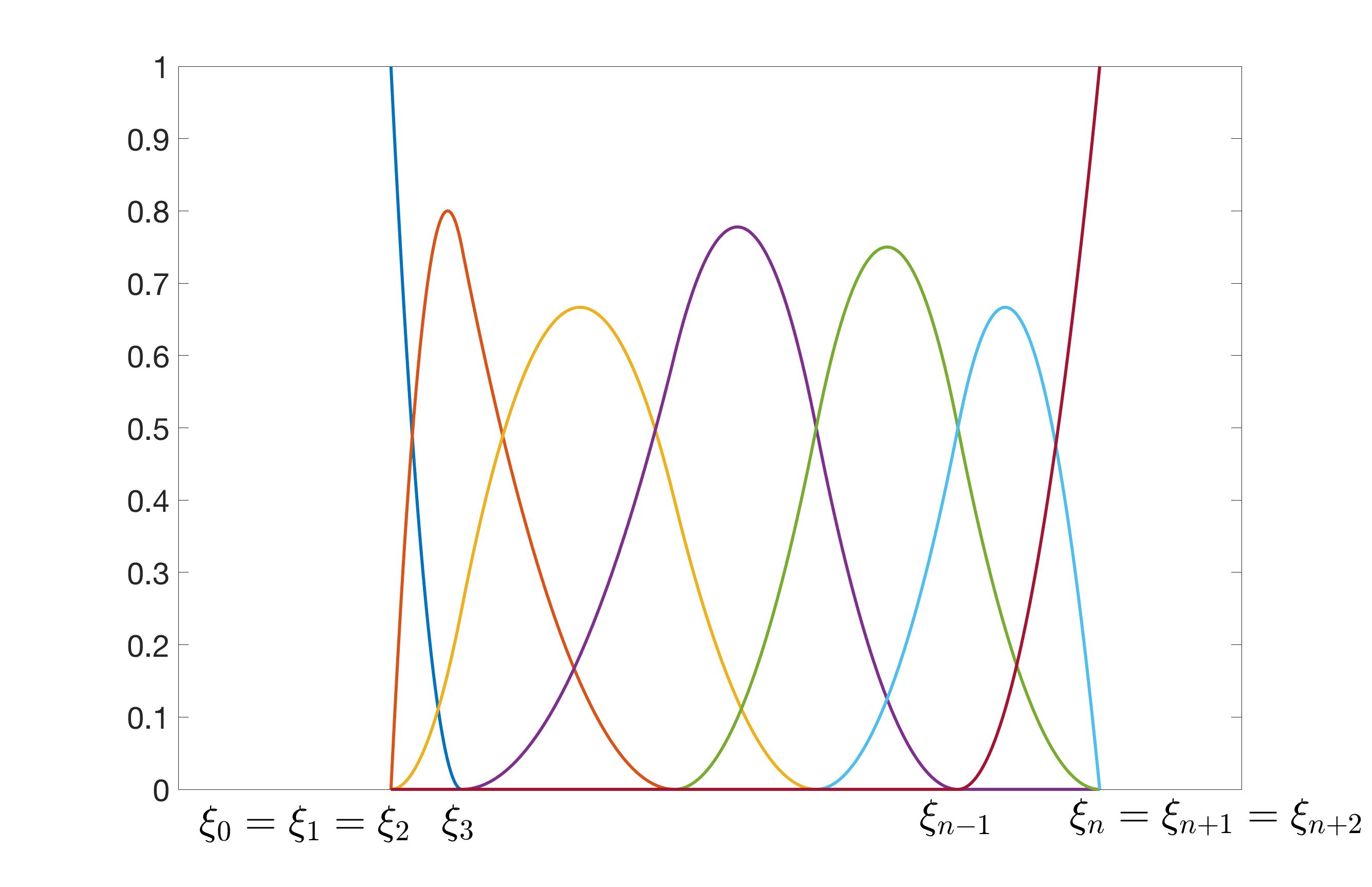}\vspace{-3mm}\hspace{-3mm}
\includegraphics[width=0.5\textwidth]{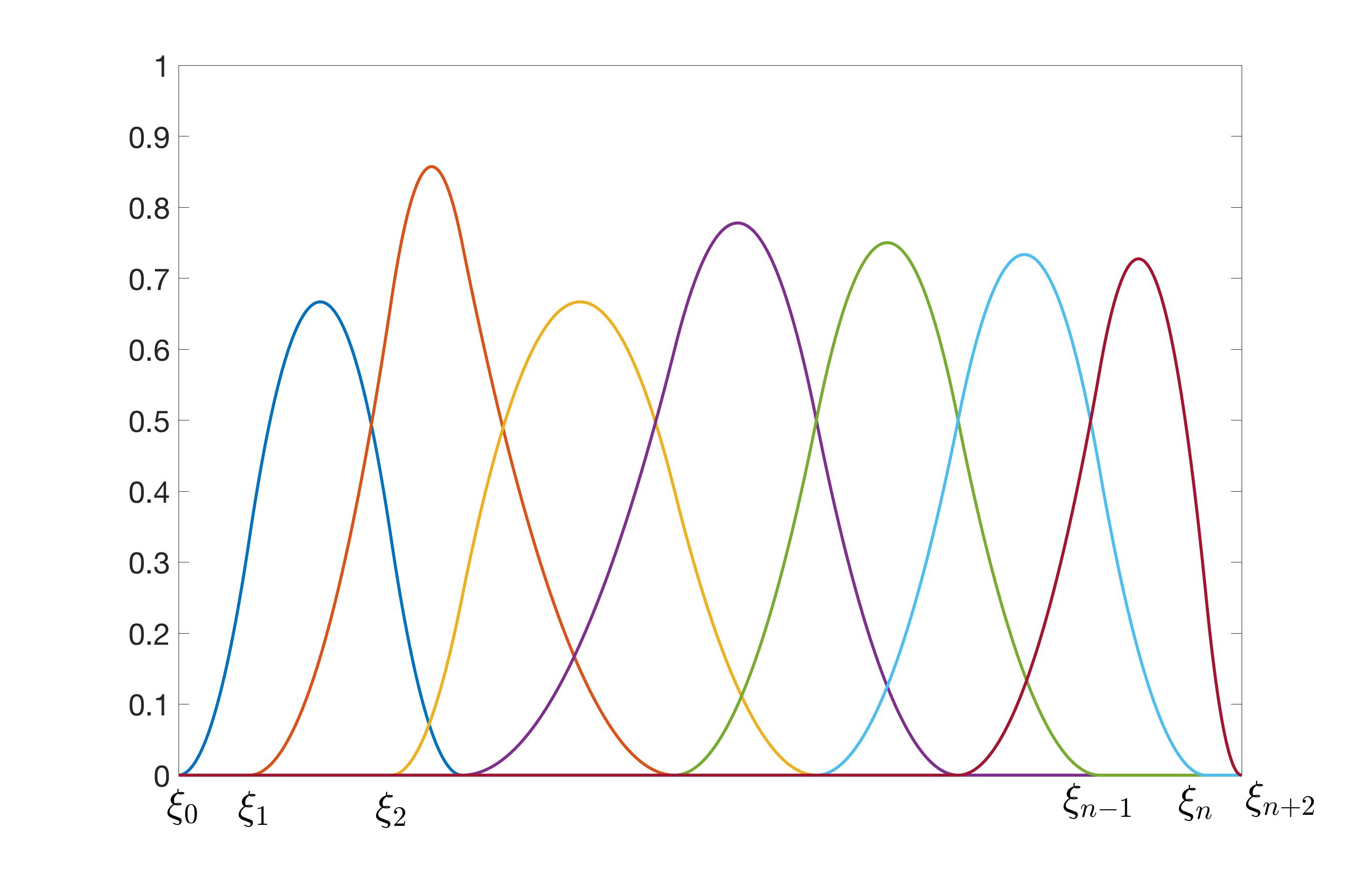}\\
\includegraphics[width=0.5\textwidth]{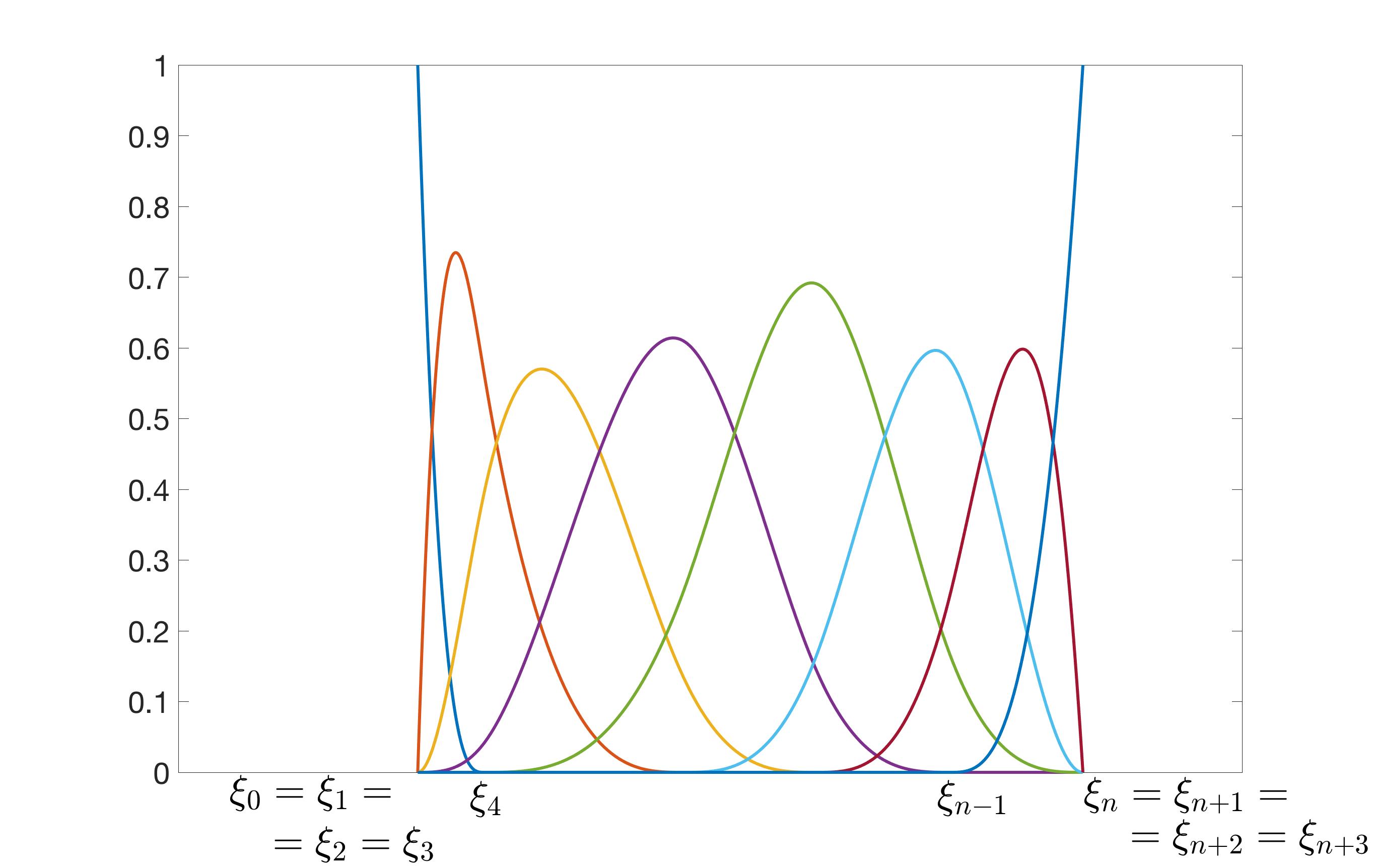}\hspace{-3mm}
\includegraphics[width=0.5\textwidth]{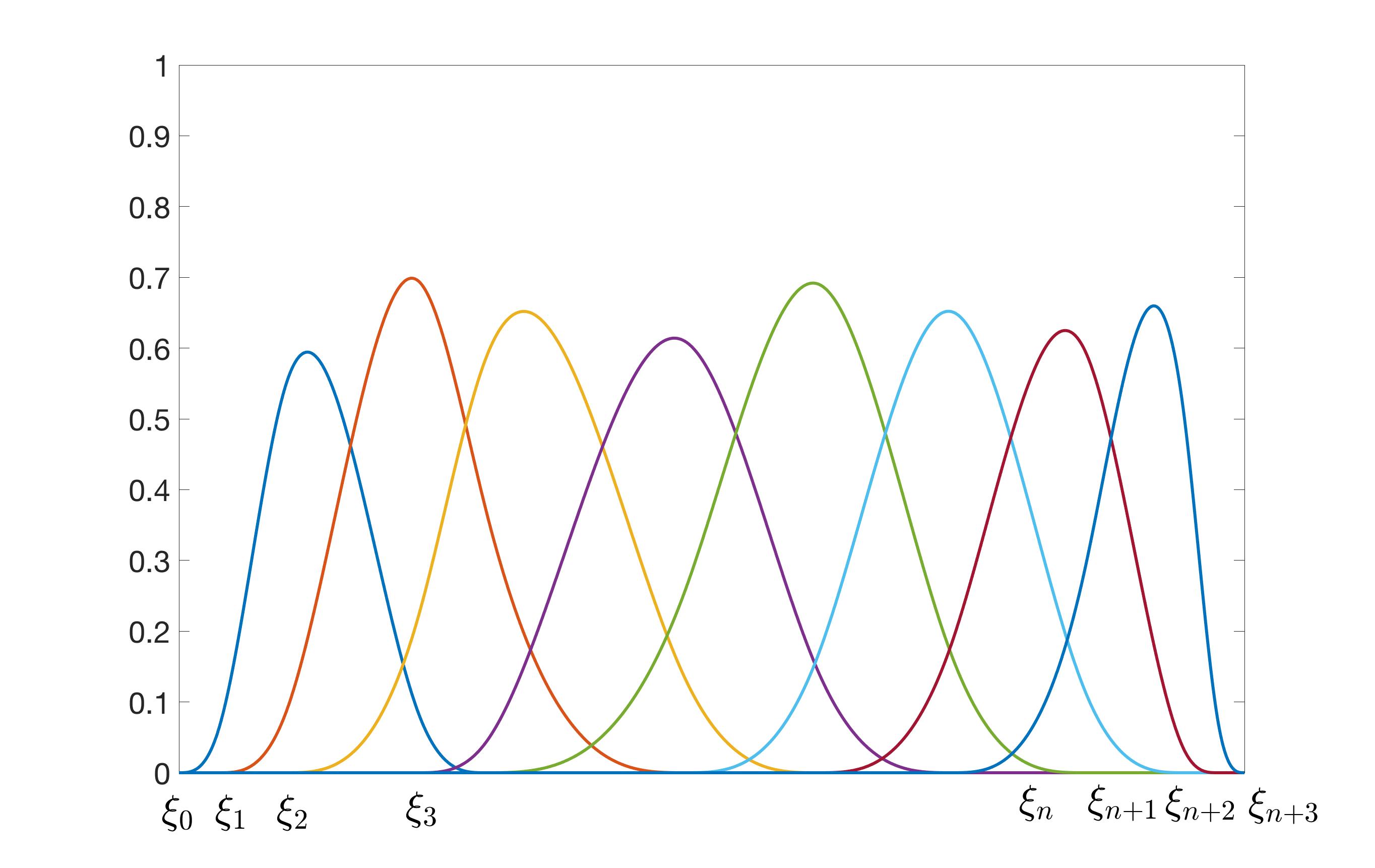}\vspace{-2mm}
\end{center} 
\caption{\small  \small The two approaches to the endpoints for the first {\it (top)}, second {\it (middle)}, and third {\it (bottom)} order $B$-splines: superfluous knots at endpoints {\it (left)},
imposed zeros as the initial conditions at the endpoints {\it (right)}.} 
\label{fig:TwoApp}
\end{figure}

The second approach is slightly more elegant as it does not introduce any superficial knots.
It rather imposes on a spline and all its derivatives of the order smaller than the spline order  the value of zero  at both the  endpoints of the domain. 
In this case, if we introduces knots through \ref{eq:knots}, we no longer assume that $\xi_{0},\dots, \xi_{K}$ and $\xi_{K+n+1},\dots,\xi_{2K+n+1}$ are taking values zero and one, respectively, but consider them to be different and ordered. 
Since in this case, there is no need to distinguish between knots it will be more common to write them simply as
\begin{equation*}
\label{eq:knots2}
\boldsymbol \xi =(\xi_{0},\dots, \xi_{n+1}),
\end{equation*}
where  $n\ge K$, in order to have enough knots to define at least one non-trivial spline with the $2K$ boundary conditions.
Indeed, if $n=K-1$, then we have $K+1$-knots yielding $K$ between knot intervals. On each such interval  a spline is equal to a polynomial of order $K$.
The dimension of the space of such piecewise polynomial functions is $K(K+1)$. However, at each internal knot there are $K$ equations to make derivative up to order $K$ (excluding) continuous. 
This reduces the space by $(K-1)K$ dimensions to $2K$, however there are $2K$ equations for the derivatives to be zero at the endpoints and the dimension of the spline space is eventually reduced to zero.

We note the dimension of the splines with the imposed boundary conditions is $(n+1)(K+1)-2K-nK=n+1-K$ and the counts is made as follows: there is $n+1$ intervals with polynomials having $K+1$ coefficients, from which one subtracts $2K$ initial conditions and $nK$  continuity conditions at the $n$ internal knots.

The two linear spaces of splines, the one with unrestricted splines at the endpoints and the one with the boundary conditions at the endpoints and $2K+n+1$ internal points, have the same dimension. 
In fact, the first one can be obtained from the second by passing to the limit 
$$
\xi_{0},\dots, \xi_{K-1} \rightarrow \xi_K, \,\,\,\xi_{K+n+1},\dots,\xi_{2K+n}\rightarrow \xi_{2K+n+1}
$$
and transforming through $T_{a}^{b}$ given in \ref{eq:abtransform} with $a=\xi_K$ and $b=\xi_{2K+n+1}$.
\ref{fig:TwoApp} illustrates these relations for the first, second, and third order $B$-splines. 
A formal formulation of this fact is given in \ref{prop:equivalence} at the end of the next subsection. 

\subsection{Recurrent definition of the $B$-splines}
\label{subsec:recdef}
The most convenient way to define the $B$-splines on the knots $\boldsymbol \xi =\left(\xi_0,\dots, \xi_{n+1} \right)$, $n=0,1,\dots$ is through the splines with the boundary conditions and using the recurrence on the order. 
Namely, once the $B$-splines of a certain order are defined, then the $B$-splines of the next order are easily expressed by their `less one' order counterparts. 
In the process, the number of the splines decreases by one and the number of the initial conditions (derivatives equal to zero) increases by one at each endpoint. 
We keep the notation $B^{\boldsymbol \xi}_{l,k}$, for the $l$th $B$-spline of the order $k$, $l=0,\dots ,n-k$. 
For the zero order splines, the $B$-spline basis is made of indicator functions 
\begin{equation}
\label{eq:indi}
B^{\boldsymbol \xi}_{l,0}=\mathbb I_{(\xi_{l},\xi_{l+1}]}, ~~~l=0,\dots n,
\end{equation}
 for the total of $n+1$-elements and zero initial conditions. 
Clearly, the space of zero order splines (piecewise constant functions) is $n+1$ dimensional so the so-defined zero order $B$-splines constitute the basis.

%

The following recursion relation leads to the definition of the splines of arbitrary order $k \le n$.
Suppose now that we have defined $B^{\boldsymbol \xi}_{l,k-1}$, $l=0,\dots,n-k+1$. 
The $B$-splines of order $k$ are defined, for $l=0,\dots,n-k$, by
\begin{equation}
\label{eq:recspline}
B_{l,k}^{\boldsymbol \xi }(x)
=
 \frac{ x- {\xi_{l}}
  }{
  {\xi_{l+k}}-{\xi_{l}} 
  } 
 B_{l,k-1}^{\boldsymbol \xi}(x)+
  \frac{{\xi_{l+1+k}}-x}{
 {\xi_{l+1+k}}-{\xi_{l+1}} 
 } 
 B_{l+1,k-1}^{\boldsymbol \xi}(x).
\end{equation}

It is also important to notice that the above evaluations need to be performed only over the joint support of the splines involved in the recurrence relation.
The recurrent structure of the support is as follows. 
For zero order splines, the support of $B_{l,0}^{\boldsymbol \xi}$ is clearly $[\xi_l,\xi_{l+1}]$, $l=0,\dots, n$. 
If the supports of $B_{l,k-1}^{\boldsymbol \xi}$'s are $[\xi_l,\xi_{l+k}]$, $l=0,\dots,n-k-1$, then the support of $B_{l,k}^{\boldsymbol \xi}$ is the joint support of $B_{l,k-1}^{\boldsymbol \xi}$ and $B_{l+1,k-1}^{\boldsymbol \xi}$, which is $[\xi_l,\xi_{l+1+k}]$,  $l=0,\dots,n-k$.

We conclude this section with a recursion formula for the derivatives of the $B$-splines that follows from \ref{eq:recspline}. 

\begin{proposition}
\label{prop:dersp}
For $i= 0,\dots, k$ and $l=0,\dots,n-k+1$:
\begin{multline}
\label{eq:recder}
\frac{d^iB_{l,k}^{\boldsymbol \xi}}{dx^i}(x)
=
\frac{i}{\xi_{l+k}-\xi_l}\frac{d^{i-1}B_{l,k-1}^{\boldsymbol \xi}}{dx^{i-1}}(x)+\frac{ x- {\xi_{l}}
  }{
  {\xi_{l+k}}-{\xi_{l}} 
  } 
\frac{d^{i}B_{l,k-1}^{\boldsymbol \xi}}{dx^{i}}(x)
+\\
+
\frac{i}{\xi_{l+1}-\xi_{l+k+1}}\frac{d^{i-1}B_{l+1,k-1}^{\boldsymbol \xi}}{dx^{i-1}}(x)+\frac{ {\xi_{l+k+1}-x}
  }{
  {\xi_{l+k+1}}-{\xi_{l+1}} 
  } 
\frac{d^{i}B_{l+1,k-1}^{\boldsymbol \xi}}{dx^{i}}(x).
\end{multline}
The support of ${d^iB_{l,k}^{\boldsymbol \xi}}/{dx^i}$ is $[\xi_{l},\xi_{l+k+1}]$ and
if $i=k$, then $d^{i}B_{l+1,k-1}^{\boldsymbol \xi}/dx^{i}\equiv 0$.
\end{proposition}

This section concludes with a result on the equivalence of the two approaches to the $B$-splines. Its proof can be found in  \ref{sec:pfsal}.
\begin{proposition}
\label{prop:equivalence}
Consider the $B$-splines $\widetilde{B}_{l,K}$, $l=0, \dots, n+K$, defined through recurrence \ref{eq:recspline} for the knots 
\begin{equation}
\label{eq:Nknots}
\widetilde{\boldsymbol \xi} =(\widetilde{\xi_{0}},\dots, \widetilde{\xi}_{K}, \widetilde{\xi}_{K+1},\dots, \widetilde{\xi}_{K+n},\widetilde{\xi}_{K+n+1},\dots,\widetilde{\xi}_{2K+n+1}),
\end{equation}
in place of $\boldsymbol \xi$, 
where $\widetilde{\xi}_{0}\le\dots \le \widetilde{\xi}_{K} = \xi_0$ and $\xi_{n}=\widetilde{\xi}_{K+n+1}\le \dots \le \widetilde{\xi}_{2K+n+1}$ are arbitrary superfluous external knots and the internal knots are, $\widetilde{\xi}_{K+i}=\xi_i$, $i=0,\dots,n$. 
Here, in the recurrence given in \ref{eq:recspline}, we assume a convention that whenever the denominator in any of the two terms is zero the entire term is assumed to be zero.  

Then  $\widetilde{B}_{K+i,K}$, $i=0,\dots, n-K$,  vanish outside of $[\xi_0,\xi_n]$ and on this interval $\widetilde{B}_{K+i,K}=B_{i,K}$,  $i=0,\dots, n-K$, where $B_{i,K}$'s are defined through \ref{eq:recspline} for $\boldsymbol \xi$.

Conversely, if $\widetilde{B}_{i,K}$, $i=0,\dots, n+K$, are the $B$-splines defined recursively for $\widetilde{\boldsymbol \xi}$ with $\widetilde{\xi}_{0}=\dots =\widetilde{\xi}_{K} = \xi_0$ and $\xi_{n}=\widetilde{\xi}_{K+n+1}= \dots = \widetilde{\xi}_{2K+n+1}$ and  $\widetilde{B}^h_{i,K}$ are the analogous ones but for  $\widetilde{\boldsymbol \xi}^h$ with $\widetilde{\xi}_{i}^h=-(K-i)h+\xi_0$ and $\xi_{n}=\widetilde{\xi}_{K+i+n+1}=\xi_{n}+ih$, $i=0,\dots, K$. Then for $i=0,\dots, n-K+1$:
\begin{align*}
\lim_{h\rightarrow 0} \widetilde{B}^h_{i,K}(x)&=\widetilde{B}_{i,K}(x),~~x\in (\xi_0,\xi_n),\\
\lim_{h\rightarrow 0} \|\widetilde{B}^h_{i,K}-\widetilde{B}_{i,K}\|&=0.
\end{align*}
where $\|\cdot\|$ is the $L_2$ norm of the square integrable functions restricted to $[\xi_0-\epsilon,\xi_n+\epsilon]$, for some $\epsilon>0$ while functions outside their support are vanishing.  
 \end{proposition}

 \section{Simple  $B$-spline orthogonalizations}
 \label{sec:simpleorth}
The popularity of the $B$-splines is, to a great extent, due to small sizes of the basis element supports, which are also disjoint for different elements except the `nearest' neighbors. 
The zero order $B$-splines (piecewise constant functions) have mutually disjoint supports which makes them orthogonal.
Any first order $B$-spline is built over two neighboring intervals defined by knots and its support is not disjoint only to one neighbor on the right and one on the left.
Similarly, the second order $B$-splines are built around a between knot interval and having overlapping support only to the two predecessors and the two successors. 
This extends to any order, i.e. the $k$th order splines  are built around a between knot interval and their support overlaps with the $k$ predecessors and the $k$  successors. 
This property is clearly seen in \ref{fig:TwoApp}~{\it (right)}.

In analysis of functional data, it is convenient to work with the orthonormal bases such as Fourier bases and similar. 
However, the $B$-splines are not orthogonal and obtaining an orthonormal basis of splines sharing to some extent the favorable properties of the $B$-splines is of interest.  
The most direct approach to obtaining such a base is through orthonormalization of the $B$-splines that has been discussed in the literature, see \cite{Mason} and \cite{Redd} and an $R$-package for the purpose has been developed, \href{https://CRAN.R-project.org/package=orthogonalsplinebasis}{\it Orthogonal B-Spline Basis Functions}. 
In \cite{BogdanL}, the orthogonalization of periodic splines has been solved in  an analytical form, while the case of numerically efficient orthonormalization was discussed in \cite{Redd} and utilized in R-package {\tt orthogonalsplinebasis}.
A method of constructing orthogonal splines that has small total support has been proposed through the so-called symmetric $O$-splines in \cite{Mason}.
Below we discuss these previous approaches. 
\subsection{One-sided orthogonalization -- the Gram-Schmidt method}
\label{subsec:GSCH}
This is simply the Gram-Schmidt  (GS) orthogonalization of the $B$ splines if one starts with either the furthest left or the furthest right $B$-spline  (the order is dictated by the order of the support intervals of the $B$-splines) and then progressively orthogonalize subsequent $B$-splines toward the other end.
The two versions of this orthogonalization can be referred to as the right-to-left GS and the left-to-right GS and they are implemented in the software packages such as {\tt orthogonalsplinebasis} as well as in our proposed package {\tt splinets} that accompanies this work (see the function {\tt grscho}). 
A computationally efficient approach to one-sided normalization is given in \cite{Qin} and was applied to splines in \cite{Redd}.
Our computational approach implemented in {\tt splinets} is utilized to obtain the graphs in \ref{fig:simporth}.

For the sake of completeness we present generic algorithm for the GS method defined for an arbitrary sequence of linearly independent elements $h_n\in \mathcal H$, $n\in \mathbb N_0$, $\mathcal H$ is some Hilbert space.
Let $h_i$, $i=0,\dots,n$, be represented as a sequence of numerical vectors $\mathbf a_i$, $i=0,\dots,n$,  in a certain basis (not necessarily orthonormal and its form is irrelevant for the procedure). 
Consider Gram matrix  $\mathbf H=\left [\langle h_i,h_j\rangle\right]_{i,j=0}^n$.
Define $\mathbf b_0=\mathbf a_0/\|h_0\|=\mathbf a_0/\sqrt{h_{00}}$, which is the representation of the first vector of the GS orthonormal basis. 
The second vector can be obtained by defining first
\begin{subequations}
\begin{align*}
\tilde{\mathbf a}_i&=\mathbf a_i-\frac{h_{i0}}{h_{00}}\mathbf a_0,~ i=1,\dots, n\\
\tilde{\mathbf H}&=\left[ h_{ij} -\frac{h_{i0}h_{0j}}{h_{00}}\right ]_{i,j=1}^n
\end{align*}
\end{subequations}
and then taking $\mathbf b_1=\tilde{\mathbf a}_1/\sqrt{\tilde{h}_{11}}$.
For $n=1$:
\begin{subequations}
\begin{align*}
\mathbf b_1&=\frac{\tilde{\mathbf a}_1}{\sqrt{{h}_{11}-\frac{h_{10}h_{01}}{h_{00}}}}=\frac{\mathbf a_1-\frac{h_{10}}{h_{00}}\mathbf a_0}{\sqrt{{h}_{11}-\frac{h_{10}h_{01}}{h_{00}}}}
\end{align*}
\end{subequations}
Moreover, to find the next vector $\mathbf b_2$, one applies the same procedure (but by one dimension smaller) to $\tilde{\mathbf a}$ and $\tilde{\mathbf H}$ in place of $\mathbf a$ and $\mathbf H$. 
This approach is implemented in our package {\tt splinets}, function {\tt grscho}, see \ref{alg:igso}.
Alternatively, one can use implemented programs for the Cholesky decomposition such as {\tt } or  {\tt gramSchmidt} in the package {\tt pracma v1.9.9}. 
\begin{algorithm}[tbh]
\caption{Gram-Schmidt orthonormalization (in R-language)}
\begin{lstlisting}[language=R]
#INPUT: `A' - columnwise matrix representation of the input vectors; `H' - Gram matrix of the vectors; 
#OUTPUT: `B' - columnwise matrix representation of the GS orthonormalization of the vectors represented by `A';

grscho=function(A,H)
{
  nb=dim(H)[1]
  B=A   # to be output with the orthonormalized columns of A
  B[,1]=A[,1]/sqrt(H[1,1]) #1st normalized output vector
  for(i in 1:(nb-1)){
    A[,(i+1):nb]=A[,(i+1):nb]-
    		(A[,i]%*%H[i,(i+1):nb,drop=F])/H[i,i] 
    #`drop=F' keeps matrix form of a row vector
    H[(i+1):nb,(i+1):nb]=H[(i+1):nb,(i+1):nb]-
    	t(H[i,(i+1):nb,drop=F])%*%H[i,(i+1):nb,drop=F]/H[i,i]
    B[,i+1]=A[,i+1]/sqrt(H[i+1,i+1])
  }
  return(B)
}
\end{lstlisting}
\label{alg:igso}
\end{algorithm}


\subsection{Two-sided orthogonalization -- a symetrized GS method}
\label{subsec:symGS}
One of the disadvantages of the one-sided orthogonalization is that the supports of the elements of the basis are growing large at each step of the orthogonalization.  
Moreover, the obtained elements are asymmetric even for the equally spaced knots. 
For example, the resulting $O$-splines on the two opposite sides of the interval have different sizes of their supports, the ones from the side where the orthogonalization started have small support while the ones on the side where the orthogonalization concludes are reaching with their support the entire interval.
In \cite{Redd}, it was suggested how one can modify the one-sided method to obtain the two-sided $O$-splines to improve the method in this respect.
We observe that the two-sided $O$-splines utilize the two one-sided orthogonalizations when properly modified at the center.
The result has two advantages compared to the one-sided one, firstly, it produces smaller total support of the obtained $O$-splines, secondly, it has the natural symmetry around the center. 
Since the actual details of the approach have not been presented, we give a short account of it. 
Moreover, the presented symmetrization is also used in our construction of the splinets. 
In a sense, the two-sided orthogonalization is a crude prototype of our approach to the orthogonalization. 
In the splinet, the size of the support of the $O$-splines is further reduced and the symmetry property becomes also a localization property within a net of splines.

\begin{figure}[t!]
  \centering
 \includegraphics[width=0.51\textwidth]{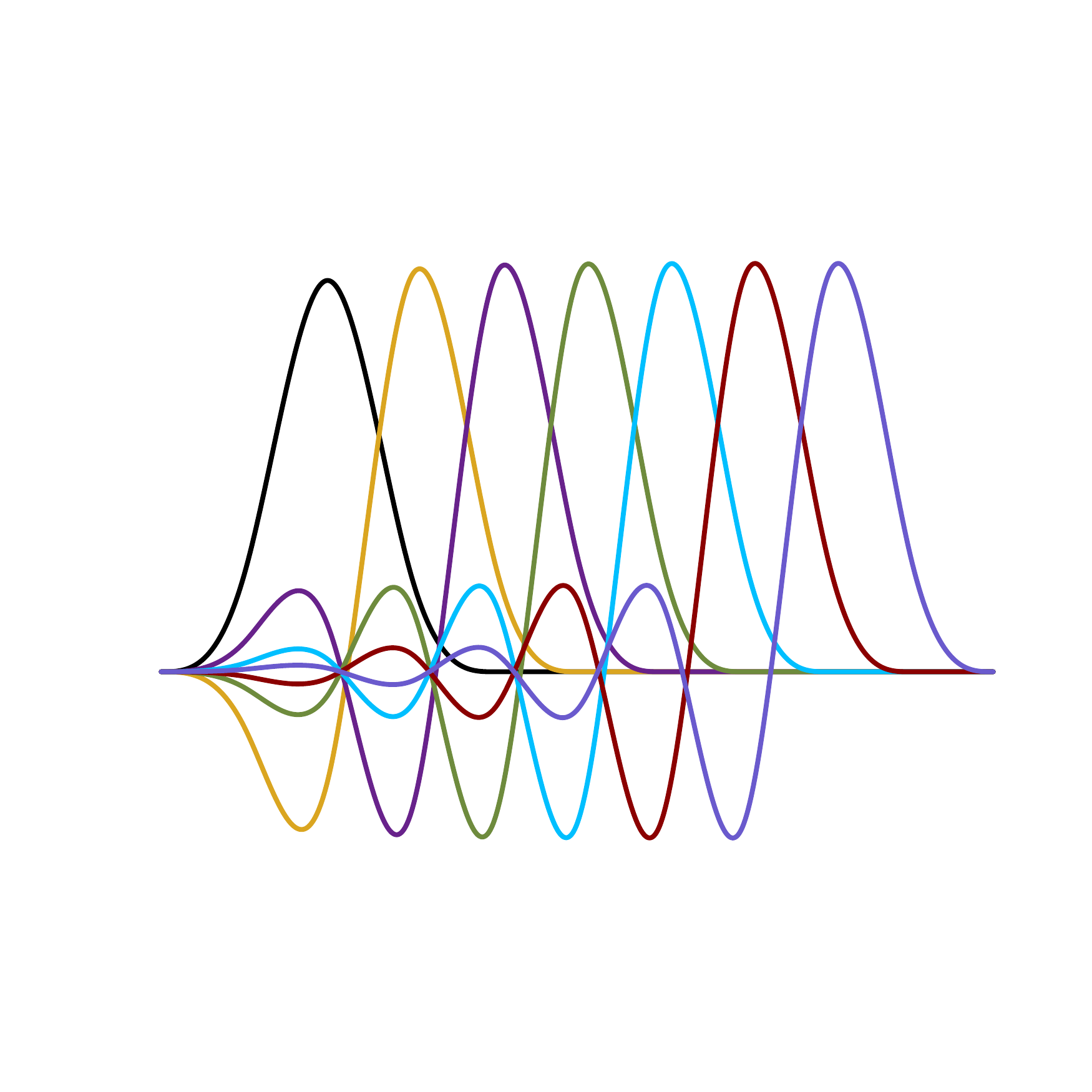}\includegraphics[width=0.49\textwidth]{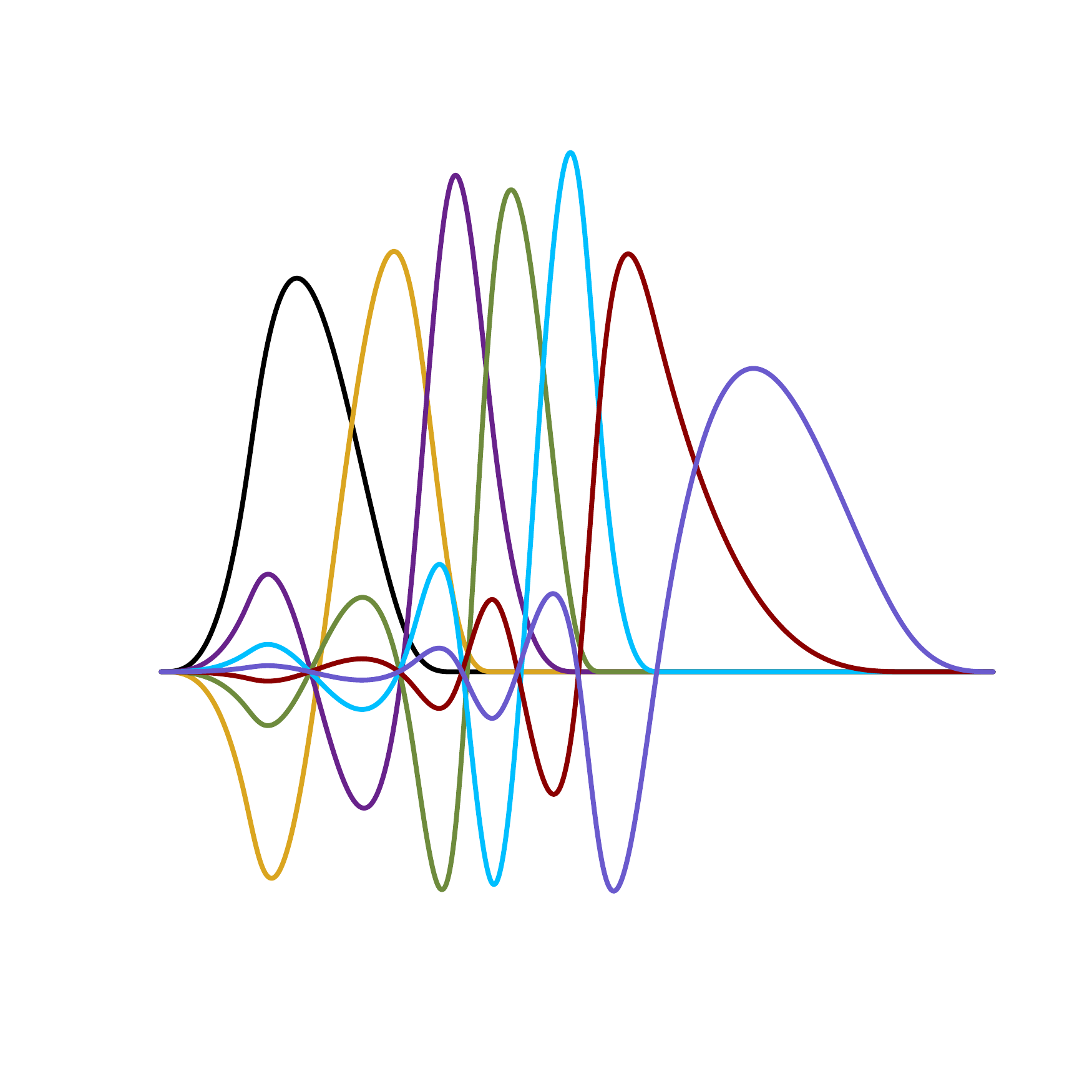}\vspace{-2.3cm}\\
\includegraphics[width=0.51\textwidth]{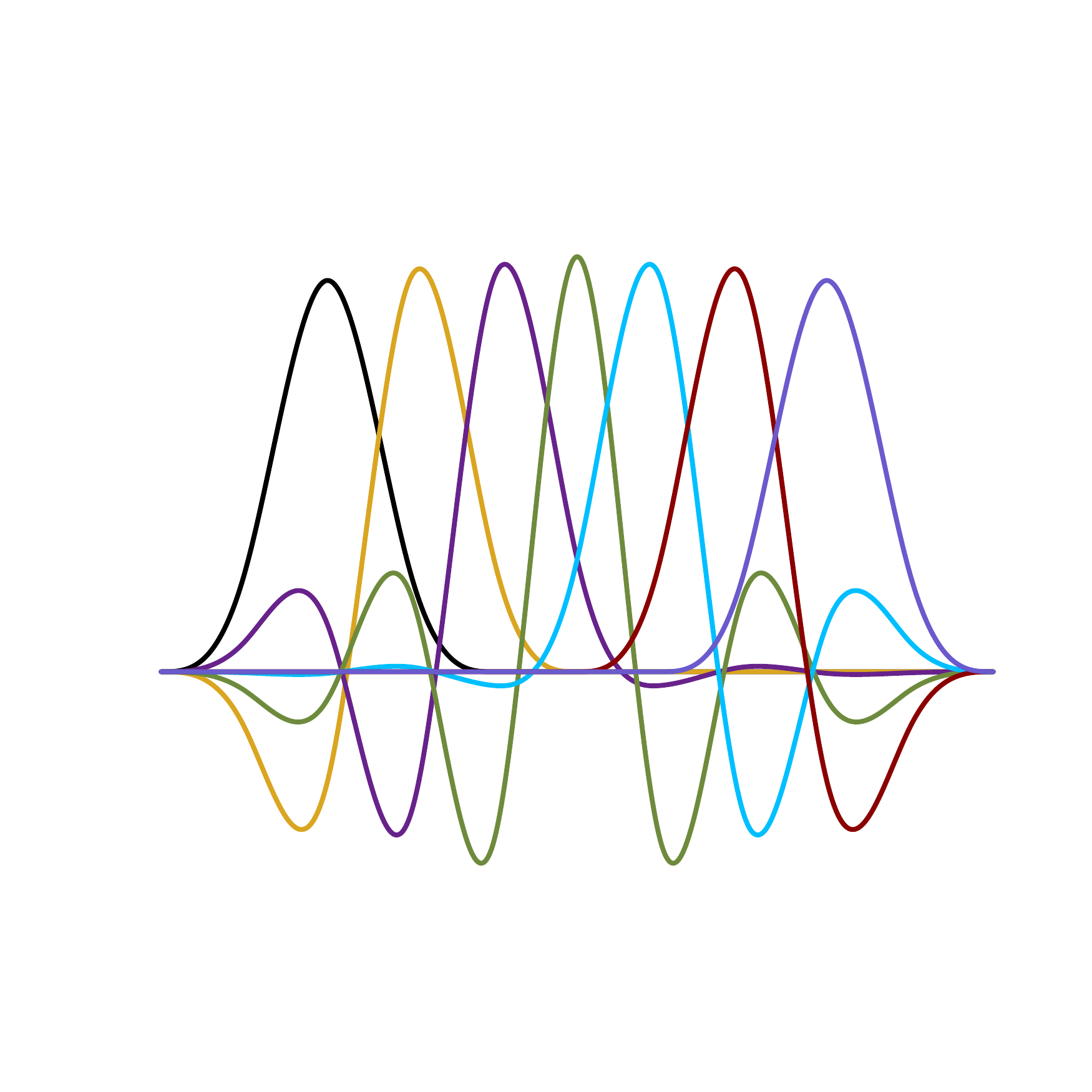}\includegraphics[width=0.49\textwidth]{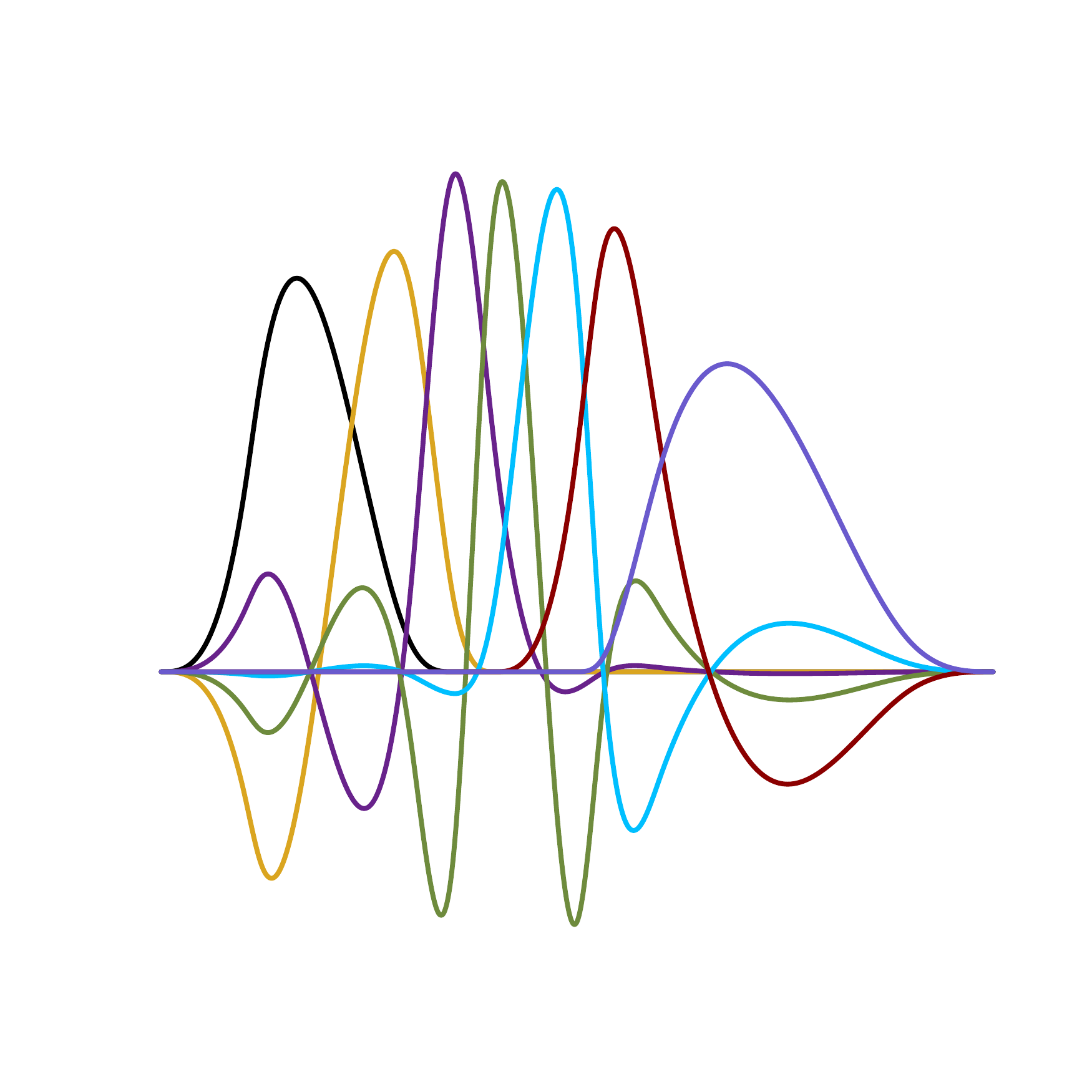}\vspace{-1cm}
  \caption{The third order $O$-splines. {\it Top}: one-sided left-to-right, ten knots;  {\it Bottom}: two-sided, eleven knots;  in the left column equally spaced knots, in the right column irregularly spaced knots. The graphs obtained using {\tt splinets} package. }
  \label{fig:simporth}
\end{figure}
 
The two-sided method we describe here is similar to the ones suggested in the past, although the central splines may slightly differ. 
It is based on a symmetrized version of the GS method formulated in the generic Hilbert space setup.
In the supplementary material, we present derivations, while here the rationale behind the approach is discussed in a less formal manner. 
The whole idea is based on the following. \vspace{2mm}
\begin{description}
\item{{\sc Step 1:}} Choose a central point inside the interval with respect to which the sides of the orthonormalization will be performed.
\item{{\sc Step 2:}} Perform the {\it left-to-right} GS orthogonalization over the knots located on the left-hand side (LHS) of the chosen center and the {\it right-to-left} GS orthogonalization over the knots located on the right-hand side (RHS) of it as long as the LHS ones are residing on the supports that are disjoint with the RHS ones.
This orthogonalization will lead to the jointly orthogonal splines.  
\item{{\sc Step 3:}} The above step leaves a central group of the $B$-splines not orthogonalized yet. 
The group is made of splines for which supports contain the central point and their number depends on the considered order of the splines.
These will be orthogonalized using a modification of GS method that will preserve natural symmetry of the splines around the central point.  
\end{description}
 
\subsubsection*{Selection of the central point}
It is quite clear that the choice of the central point affects the total support size of the resulting two-sided orthogonalization. 
For equally spaced knots, the mid-point of the entire range should be considered to minimize the total support. 
However, if the knots are not evenly spread it is less clear where it would be optimal to place the central point. 

Let us consider a vector of knots $\boldsymbol \xi$ with an internal  knot $\xi_{k_0}$.
When deciding for the choice of  the central point near $\xi_{k_0}$ or just at it, one can aim for the minimal total support of the resulting $O$-splines. 
For simplicity, let us assume that we deal with the first order splines although the argument extends to an arbitrary order. 
We also take  $\xi_0=0$ and $\xi_{m+1}=1$. 
The total support of the $B$-splines of the first order is 
\begin{subequations}
\begin{align*}
\xi_2+\xi_3-\xi_1+\dots + \xi_{m+1}-\xi_{m-1} &=\sum_{k=2}^{m+1} \xi_k - \sum_{k=0}^{m-1} \xi_k \\
&=\xi_{m+1}+\xi_m - \xi_1=2-\left[(1-\xi_{m}) + \xi_1 \right] \approx 2, 
\end{align*}
\end{subequations}
where approximation holds for a sufficiently fine grid of knots. 
The one-sided orthogonalization of the $B$-splines yields the following support
$$
\xi_2+ \xi_3+\dots + \xi_{m+1}=\sum_{k=2}^{m+1}\xi_k=2-\left[(1-\xi_{m}) + \xi_1 \right]+\sum_{k=1}^{m-1~} \xi_k,
$$
which, in general, is larger than the total support for the $B$ splines and generally increase without bound when the knot grid is infinitesimally refined. 
For the two-sided orthogonalization around the central point at $\xi_{k_0}$ we have for the total support size
\begin{subequations}
\begin{align*}
1+\sum_{k=2}^{k_0}\xi_{k}+ \sum_{k=k_0}^{m-1} (1-\xi_{k})
&=2+\sum_{k=2}^{k_0-1}\xi_{k} + \sum_{k=k_0+1}^{m-1} (1-\xi_{k})\\
&=2-\left[(1-\xi_{m}) + \xi_1 \right]+\sum_{k=1}^{k_0-1}\xi_{k} + \sum_{k=k_0+1}^{m} (1-\xi_{k}),
\end{align*}
\end{subequations}
where we take into the account that the central $O$-spline with the central point $\xi_{k_0}$ spreads over the entire range and thus the size of its support is equal to one. 
We observe that the minimal $k_0$ such that $\xi_{k_0}$ is larger than $1-\xi_{k_0+1}$ typically results in the optimal or nearly optimal choice. 
Thus one can simply take $k_0$ such that $\xi_{k_0}$ is the closest to 0.5 among all $\xi_k$'s. 
We conclude that for the two-sided orthogonalization the best is to take the central point around the midpoint of the total range. 
\subsubsection*{Symmetry with respect to the central point}
Once a choice of the central point, say $\xi$, is determined one can consider two one-sided orthogonalization around this point, the left-to-right orthogonalization for all $B$-splines with the support contained in $[\xi_0,\xi]$ and the right-to-left orthogonalization for those $B$-splines that have the support contained in $[\xi,\xi_{m+1}]$. The resulting two sets of $O$-splines are jointly orthogonal due to disjoint supports.
In \ref{fig:simporth} {\it (bottom)}, the first two $O$-splines from each side correspond to this portion of the orthogonalization.  
In the case of evenly spaced knots this leads to the symmetric $O$-splines around the mid-point as seen in \ref{fig:simporth} {\it (bottom-left)}.
\subsubsection*{Symmetric orthogonalization}
There are several $B$-splines having the support containing the central point $\xi$. 
We aim at orthogonalization that for the equispaced knots preserve the symmetry around $\xi$. 
In general, there may be even or odd number of such $B$-splines and there can be several approaches to their orthogonalization that preserves their symmetry. 
Our approach is based on the following generic result, for the proof see \ref{sec:pfsal}, and for geometric interpretation see \ref{fig:SymOrth}.
\begin{figure}[thb]
  \centering
 \includegraphics[width=0.99\textwidth]{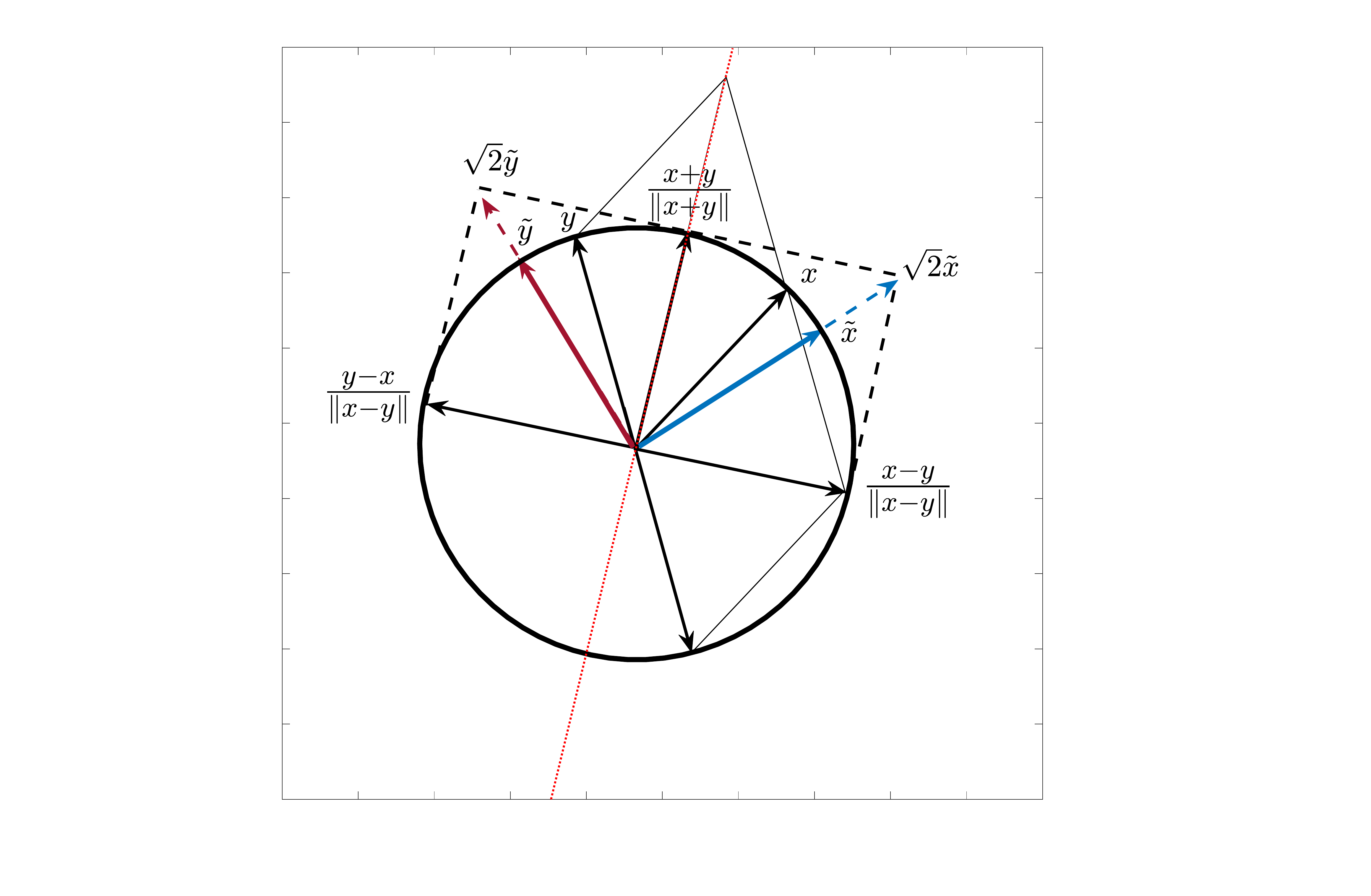}\vspace{-3mm}
  \caption{Geometric interpretation of \ref{prop:reflect}. The symmetric orthogonalization with respect to the reflection to the axis marked by the dotted line.  }
  \label{fig:SymOrth}
\end{figure}

\begin{proposition}
\label{prop:reflect}
Let $S$ be a symmetry operator on a certain linear vector space, i.e. a linear operator for which $S^2=I$. Let $x$ and $y$ be two linearly independent and normalized vectors  such that $y=Sx$. 
Define
\begin{subequations}
\begin{align*}
 \tilde x& = \frac{1}{\sqrt{2}} \left(\frac{x+y}{\|x+y\|}+  \frac{x-y}{\|x-y\|}\right)\\
 \tilde y& = \frac{1}{\sqrt{2}} \left (\frac{x+y}{\|x+y\|}-  \frac{x-y}{\|x-y\|}\right)
\end{align*}
\end{subequations}
Then vectors $\tilde x$ and $\tilde y$ are orthonormal and symmetric, i.e. $\langle \tilde x, \tilde y \rangle=0$ and $S\tilde x = \tilde y$.
\end{proposition}

Implementation of this proposition is shown in \ref{alg:iso}, where we use that for the normalized input  $\|x\pm y\|^2=2(1\pm\langle x,y\rangle)$. 
\begin{algorithm}[thb]
\caption{Symmetrized orthonormalization (in R-language)}
\label{alg:iso}
\begin{lstlisting}[language=R]
#INPUT: vectors as two columns in `x', the vectors are assumed to be NORMALIZED with respect to a certain inner product, their inner product is in `h';
#OUTPUT: symmetrically orthonormalized two vector as two columns in `res'; 
#The inner product and norm DO NOT NEED to be Euclidean, i.e. not 
#necessarily `sum(x[,1])^2=sum(x[,2])^2=1' and `sum(x[,1]*x[,2])=h'
 symo=function(x,h){
# x: 2 columns matrix contains coefficients w.r.t. the basis
# h: the inner product of two vectors w.r.t. the basis
  a1=(1/sqrt(1+h)+1/sqrt(1-h))/2
  a2=(1/sqrt(1+h)-1/sqrt(1-h))/2
  res=x
  res[,1] = a1*x[,1]+a2*x[,2]
  res[,2] = a2*x[,1]+a1*x[,2]
  return(res)
}
\end{lstlisting}
\end{algorithm}
\begin{remark}
\label{rem:lowdin}
The orthogonalization in \ref{prop:reflect} is  equivalent to the L\"owdin symmetric orthogonalization for the case of two vectors,  see \cite{lowdin}, since \begin{subequations}
\begin{align*}
 \tilde x& =  \left(\frac{1}{\sqrt {1+ \langle x, y \rangle}} + \frac{1}{\sqrt {1- \langle x, y \rangle}}\right) \frac{x}2 +
 \left(\frac{1}{\sqrt {1+ \langle x, y \rangle}} - \frac{1}{\sqrt {1- \langle x, y \rangle}}\right)\frac y 2, \\
 \tilde y& =\left (\frac{1}{\sqrt {1+ \langle x, y \rangle}} - \frac{1}{\sqrt {1- \langle x, y \rangle}}\right) \frac x 2+ \left(\frac{1}{\sqrt {1+ \langle x, y \rangle}} + \frac{1}{\sqrt {1- \langle x, y \rangle}}\right) \frac y 2.
\end{align*}
\end{subequations}
\end{remark}

\sloppy Once the symmetric orthogonalization of two vectors is defined, the central splines, say $s_i$, $i=1,\dots, r$, are handled as follows. 
The first pair $(s_1,s_r)$ is first orthogonalized with respect to the LHS and RHS ones obtained in Step 2, leading to, say, $(x_1,x_r)$.
This pair, in turn, is orthogonalized using the above symmetric orthogonalization and then added to the LHS and RHS $O$-splines. 
This is repeated with every $(s_i,s_{r-i+1})$ until either no $B$-spline is left ($r$ is even) or there is one central $B$-spline left. 
In the latter case, we simply orthogonalize it with respect to all the previous $O$-splines. 
In \ref{fig:simporth}~{\it (bottom)}, we see the resulting central $O$-splines with their support extending over the entire $[\xi_0,\xi_{m+1}]$. 
In the LHS graph, the $O$-splines are symmetric around the central point. 
\begin{remark}
\label{rem:SymOp}
A natural symmetry for splines on the equally spaced knots inside the interval is the mirror reflection, say $S$, with respect  to the midpoint of the interval. We note that $S^2=I$ and $S^*=S$. 
This also extends to the case of non-equally spaced knots when they are placed symmetrically around the midpoint. 
\end{remark}

\subsubsection*{Symmetrized Gram-Schmidt orthogonalization}
The goal here is to present an analog to the Gram-Schmidt orthogonalization that yields a symmetric set of orthogonal vectors while starting from a set of linearly independent vectors that forms a set of pairs   $\{(x_i,y_i)\}_{i=1}^n$ that are symmetric with respect to a symmetry operator $S$, i.e. $Sx_{i}=y_{i}$.
For the construction to preserve the symmetry, additionally to $S^2=I$,  it is required that $S$ is self-adjoint, i.e. $S^*=S$.

This additional condition, self-adjointness, is only satisfied for the case of equally spaced knots thus the following orthogonalization results in symmetric pairs of splines only in that case. 
However, the orthonormalization procedure itself can be performed in any case leading to visually balanced $O$-splines as seen in  \ref{fig:simporth} {\it (bottom-right)}.
The symmetry property does not seem directly interpretable when the knots are not equally spaced unless they are spaced symmetrically around the midpoint. 

Similarly to the regular GS method and the symmetric orthogonalization presented above, the construction is generic, in the sense that it does not require the vectors to be splines. 
More precisely, the outcome of the procedure is a set of orthogonal vectors in the form of pairs $(\tilde x_i,\tilde y_i)$, $i=1,\dots ,n$ for which, in the case of self-adjoint $S$,  $\tilde y_i=S\tilde x_i$.
The process formalizes symmetrization and orthogonalization of vectors as described above. 
In its description, for $k=1,\dots,n$, $P_k$ stands for the orthogonal projection to the linear span of $\{(x_i,y_i)\}_{i=1}^k$.
 
For a set of pairs $\mathcal S=\{(x_i,y_i)\}_{i=1}^n$,
take the first pair $(x_1,y_1)$ and construct the orthogonal and symmetric pair $(\tilde x_1,\tilde y_1)$ as it follows from \ref{prop:reflect}.
Replace all $\{(x_i,y_i)\}_{i=2}^n$ by $ \{(x_i-P_1x_i,y_i-P_1y_i)\}_{i=2}^n$ and set $\mathcal S$ to this new set.
We observe that the symmetry is preserved 
\begin{align*}
S(x_i-P_1x_i)&=y_i-SP_1x_i=y_i-\frac{\langle x_i,\tilde x_1\rangle}{\|\tilde x_1\|^2} S\tilde x_1-\frac{\langle x_i,\tilde y_1\rangle}{\|\tilde y_1\|^2} S\tilde y_1\\
&=y_i-\frac{\langle x_i,S\tilde y_1\rangle}{\|S\tilde y_1\|^2} \tilde y_1-\frac{\langle x_i,S\tilde x_1\rangle}{\|S\tilde x_1\|^2} \tilde x_1\\
&=y_i-\frac{\langle Sx_i,\tilde y_1\rangle}{\|\tilde y_1\|^2} \tilde y_1-\frac{\langle Sx_i,\tilde x_1\rangle}{\|\tilde x_1\|^2} \tilde x_1\\
&=y_i-\frac{\langle y_i,\tilde y_1\rangle}{\|\tilde y_1\|^2} \tilde y_1-\frac{\langle y_i,\tilde x_1\rangle}{\|\tilde x_1\|^2} \tilde x_1\\
&=y_i-P_1y_i.
\end{align*}
Apply recursively the procedure to $\mathcal S$ until there are no more pairs to orthogonalize. 

To implement this orthogonalization we utilize the GS method.
The next result combines GS orthogonalization with the pairwise symmetrization discussed above to produce the symmetrized GS orthogonalization. 
It is formulated so that \ref{alg:igso} and \ref{alg:iso} can be directly utilized in numerical implementations as shown in \ref{alg:isgso}.
 The proof of the result is given in \ref{sec:pfsal}. 
\begin{proposition}
\label{prop:gssym}
For any sequence $\mathbf z=(z_i)_{i=1}^n$ of linearly independent vectors let $\mathbf y=gs(\mathbf z)$ be a sequence of the orthonormal vectors obtained by application of the GS method to $\mathbf z$. 
Consider $\mathbf x=(x_i)_{i=1}^{2k}$ and define $ \mathbf x^R$ and $ \mathbf x^L$ through
\begin{subequations}
\begin{align*}
(x^L_1,x^L_{2k},x^L_2,x^L_{2k-1}, \dots,x^L_k,x^L_{k+1})&=gs(x_1,x_{2k},x_2,x_{2k-1}, \dots,x_k,x_{k+1} ),\\
(x^R_{2k},x^R_1,x^R_{2k-1},x^R_2, \dots,x^R_{k+1}, x^R_k )&=gs(x_{2k},x_1,x_{2k-1},x_2, \dots,x_{k+1}, x_k ).
\end{align*}
\end{subequations}
Define $\mathbf y=(y_i)_{i=1}^{2k}$ so that $(y_i,y_{2k - i+1})$ is obtained from $(x_i^L, x_{2k-i+1}^R)$ by the orthonormalization described in \ref{prop:reflect}.
Then 
\begin{itemize}
\item[1)] ${ \mathbf y}$ has orthogonal terms, 
\item[2)] $(y_i, y_{2k-i+1})$ are orthogonal to $\{ x_j,  x_{2k-j+1}, j<i\}$,
\item[3)] $\{ x_j,  x_{n-j+1}, j\le i\}$ is spanned by $\{ y_j,  y_{2k-j+1}, j \le i\}$, 
\item[4)]  if $Sx_i=  x_{2k-i+1}$, $i=1,\dots, k$, for a certain self-adjoint symmetry operator $S$, then also $Sy_i=  y_{2k-i+1}$.
\end{itemize}
\end{proposition}

\begin{algorithm}[thb]
\caption{Symmetrized Gram-Schmidt orthonormalization (in R-language)}
\label{alg:isgso}
\begin{lstlisting}[language=R]
#INPUT: `A' - columnwise matrix representation of the input vectors; `H' - Gram matrix of the vectors; 
#OUTPUT: `B' - columnwise matrix representation of the symmetric GS orthonormalization of the vectors represented by `A';

sgrscho=function(A,H){
  nb=dim(H)[1]; np=floor(nb/2) #  number of vectors and pairs
  B=A
  BR=matrix(0,ncol=2*np,nrow=nb)
  HR=matrix(0,ncol=2*np,nrow=2*np)
  J=c(2*(1:np)-1,2*(1:np)); K=c(nb:(nb-np+1),1:np)
  BR[,J]=A[,K]; HR[J,J]=H[K,K]

  BL=matrix(0,ncol=nb,nrow=nb); HL=matrix(0,ncol=nb,nrow=nb)
  J=c(J,nb); K=c(1:np,nb:(nb-np+1),np+1)
  BL[,J]=A[,K]; HL[J,J]=H[K,K]

  BL=grscho(BL,HL); BR=grscho(BR,HR) # Gram-Schmidt method

  B[,np+1]=BL[,nb] # center for the odd case

  for(i in 1:np){ # symmetrization
    X=cbind(BL[,2*i-1], BR[,2*i-1])
    h=(X[,1]%*% H %*%X[,2,drop=F])[1,1]
    B[,c(i,nb-i+1)]=symo(X,h)
  }
  return(B)
}
\end{lstlisting}
\end{algorithm}

The above defines symmetrized GS orthogonalization for an even number of vectors to be orthogonalized. 
The following augments it to the case of an odd number. 
\begin{corollary}
\label{cor:c2}
Consider $\mathbf x=(x_i)_{i=1}^{2k+1}$. Define $ \mathbf x^R$ and $ \mathbf x^L$ through
\begin{subequations}
\begin{align*}
(x^L_1,x^L_{2k+1},x^L_2,x^L_{2k-1}, \dots,x^L_k,x^L_{k+2},x^L_{k+1})&=gs(x_1,x_{2k+1},x_2,x_{2k}, \dots,x_k,x_{k+2},x_{k+1} ),\\
(x^R_{2k+1},x^R_1,x^R_{2k},x^R_2, \dots,x^R_{k+2}, x^R_k )&=gs(x_{2k+1},x_1,x_{2k},x_2, \dots,x_{k+2}, x_k ).
\end{align*}
\end{subequations}
Further, let $(y_i)_{i=1,i\ne k+1}^{2k+1}$ be defined using \ref{prop:gssym} for $(x_i)_{i=1,i\ne k+1}^{2k+1}$ and let $y_{i+1}=x^L_{k+1}$ . Then
$\mathbf y=(y_i)_{i=1}^{2k+1}$ satisfies
\begin{itemize}
\item[1)] orthogonal terms $(y_i, y_{2k-i+2})$ are also orthogonal to $\{ x_j,  x_{2k-j+2}, j<i\}$, 
\item[2)] $\{ x_j,  x_{n-j+2}, j\le i\}$ are spanned by $\{ y_j,  y_{2k-j+2}, j \le i\}$,
\item[3)]  if $Sx_i=  x_{2k-i+2}$, $i=1,\dots, k$,  and $Sx_{k+1}=x_{k+1}$, for a certain self-adjoint symmetry operator $S$, then also $Sy_i=  y_{2k-i+2}$, $i=1,\dots, k$, and $Sy_{k+1}=y_{k+1}$. 
\end{itemize}
\end{corollary}
The results are used in \ref{alg:isgso}, which is the basis for orthogonalization procedures used in numerical implementation of the methods of this and the next section in the {\tt splinets} package. 
Although it was not emphasized in our discussion, all the orthogonalized outputs have been also normalized.  
This is spelled out explicitly in \ref{alg:igso} -\ref{alg:isgso} by using the term {\it`orthonormalization'} in their descriptions. 
In our descriptions of orthonormalization methods, the following definition proves useful in reducing the notational burden. 
\begin{definition}
\label{def:soop}
Let $\mathbf x=(x_1,\dots,x_m)$ have linearly independent components belonging to some Hilbert space and $\mathbf y=(y_1,\dots,y_m)$ have orthonormalized components obtained from $\mathbf x$ by orthogonalization described in  \ref{prop:gssym} and \ref{cor:c2}, augmented by the normalization of the orthogonal outcome. 
This transformation 
\begin{equation}
\label{eq:chb} 
\mathbf y=\mathcal G\left( \mathbf x\right)
\end{equation}
can be also interpreted as a mapping $\mathbf P:\sum_{i=1}^m \alpha_ix_i \mapsto \sum_{i=1}^m \alpha_iy_i$, i.e. a change of basis transformation.  
\end{definition}
\begin{remark}
Let $\mathbf A=\left[ a_{ij}\right]$ be such that $x_i=\sum_{j=1}^p a_{ij} e_j$, for a certain basis $\mathcal E=\{e_i,i=1,\dots p\}$.
Then the transformation $\mathbf P$ connects with the output $\mathbf B$ in \ref{alg:isgso} through 
\begin{align*}
y_i&=\sum_{j=1}^p b_{ij} e_j=\mathbf P x_i =\sum_{j=1}^p a_{ij} \mathbf P e_j =\sum_{j=1}^p\sum_{r=1}^p a_{ir}  p_{rj} e_j .
\end{align*}
i.e. the matrix representation of $\mathbf P$ in the basis $\mathcal E$ satisfies $\mathbf B=\mathbf A \mathbf P$. 
\end{remark}

\section{The splinets - structured orthogonalization}
\label{sec:splinets}
We propose a  novel approach to $B$-spline orthogonalization.
It has important advantages over the previous methods by giving overall small support of the basis and  sporting an elegant symmetry. 
The term {\it  $O$-spline} has been used loosely in the past to any orthogonal family of splines.
In what follows we will use the term {\em splinet} specifically to the set of $O$-splines that is defined through the presented method. 
In the approach, we orthogonalize the $B$-splines so that their natural structure is followed as closely as possible, leading to some optimality of the resulting orthogonal spline base.
To emphasize that our starting point are the $B$-splines, the orthogonal splines of a splinet are referred to as the {\it $OB$-splines}. 

In some aspects, splinets resemble wavelets. 
In particular, we have a parallel concept to resolution scales that are determined by the number of in-between knots intervals holding the support of an individual spline. 
However, we refer to them as the {\it support levels} since in the case of the splines they do not truly represent different resolution scales.  

The symmetric GS orthogonalization of the previous section aimed at reduction of the size of the total support for the elements of the basis while at the same time preserving symmetry features.
Both properties are improved even further by the proposed orthogonalization leading to a splinet. 
The construction is largely benefitting from  restructuring  the knots, the in-between knots intervals, and then the $B$-splines from a sequential structure into a dyadic one as described next. 

\subsection{Dyadic structures}
\label{subsec:dyadicstr}
For the $B$-splines of order $k$, the case of the number of the internal knots equal to $n=k2^N-1$ is referred to as the {\it full dyadic case}.
We first consider the case of $k=1$ to set the basic notation and then we turn to the case of general $k$ that can be then easily laid out. 
\subsubsection*{The case $k=1$}
We first note that the case of $N=1$ is trivial  so in what follows we assume $N>1$. 
By having the dyadic structure of the in-between knots intervals, we denote the following hierarchical structure of support sets for the to-be constructed $OB$-splines. 

The smallest support range is at the level $l=0$ (referred to as the $0$-support level) and is made of neighboring pairs of individual intervals, i.e. we pair the first one with the second, the third with the fourth and so on until the second last is paired with the last.  
These smallest support intervals are denoted by $I_{r ,0}$, $r=1,\dots, 2^{N-1}$.
We note that the Lebesgue measure of the total support made of all the intervals at this support range is just equal to the length of the entire domain, i.e. to $\xi_{2^N}-\xi_{0}$.
Assuming the knots $\xi_{0}<\xi_{1}<\dots <\xi_{2^N-1}< \xi_{2^N}$, we have $I_{r,0}=(\xi_{2r-2},\xi_{2r}]$, with the central knot $\xi_{2r-1}$ in each $I_{r,0}$.

\begin{figure}[t!]
  \centering
 \begin{minipage}{0.47\textwidth}
\includegraphics[width=\textwidth]{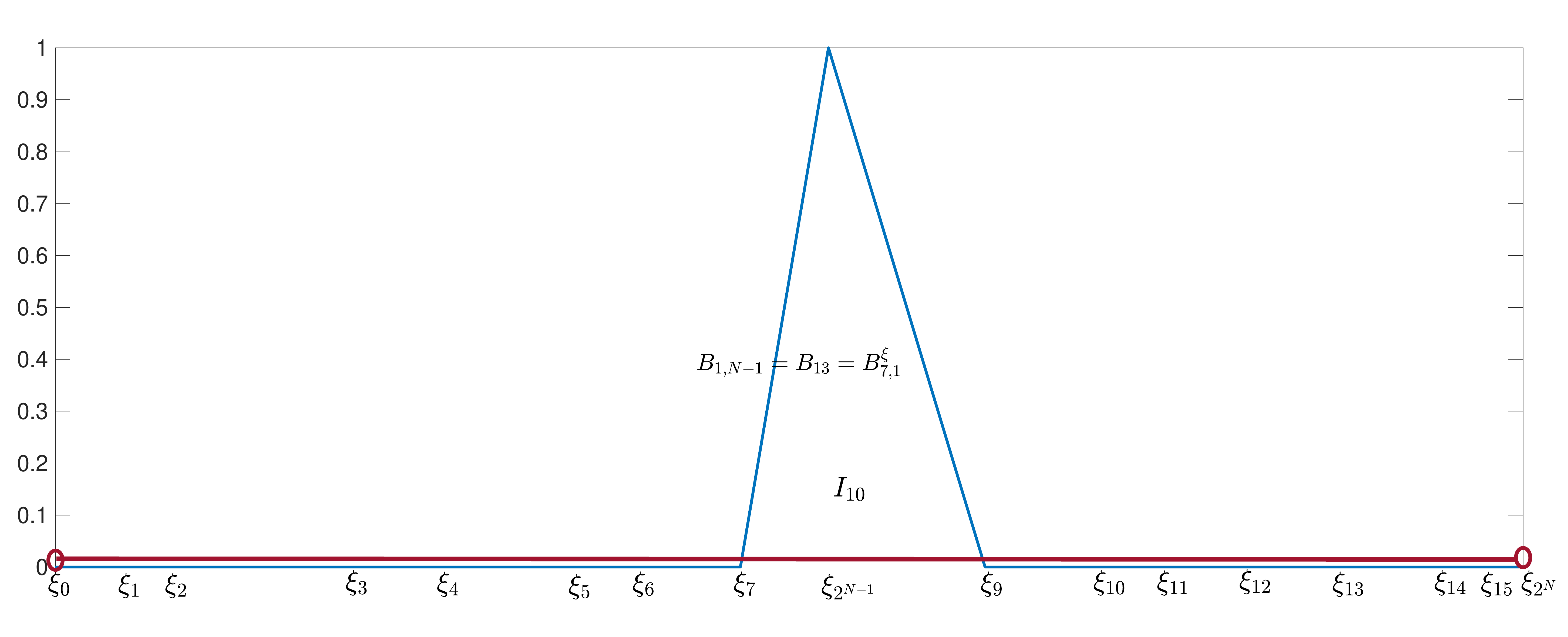}
\includegraphics[width=\textwidth]{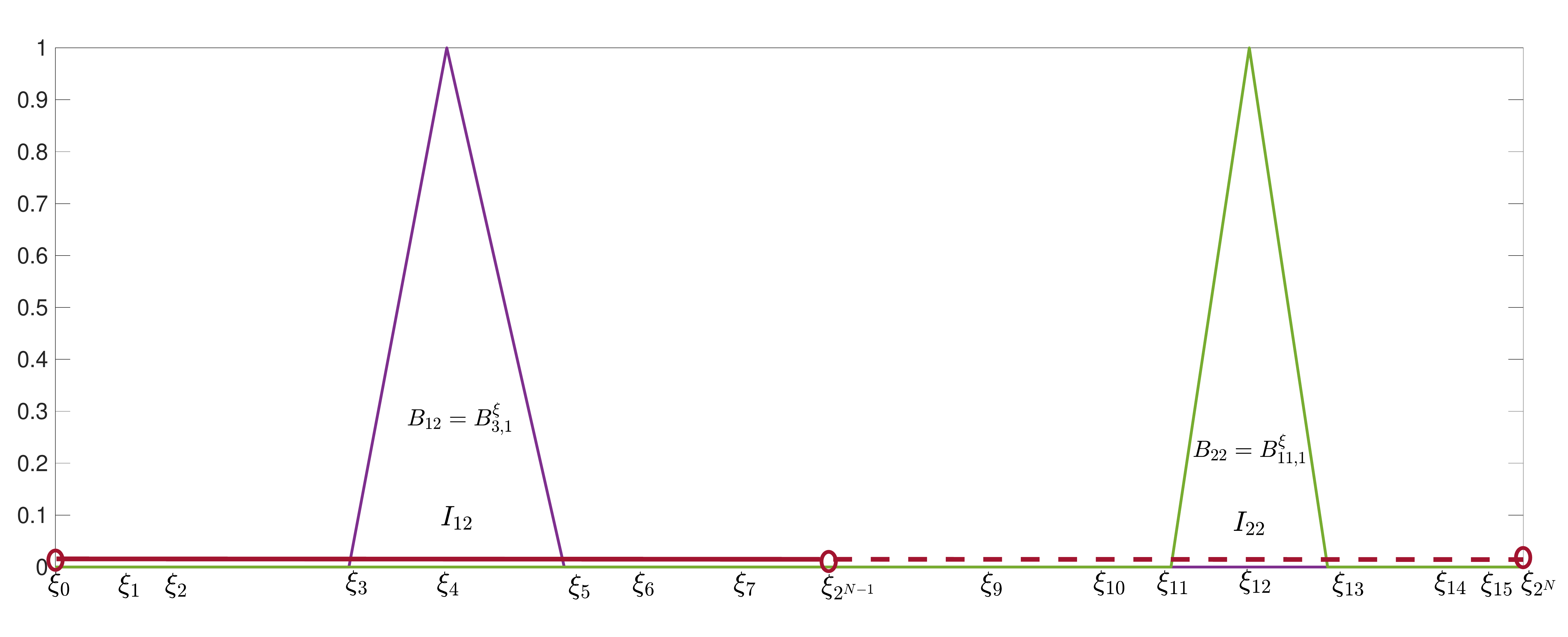}
\includegraphics[width=\textwidth]{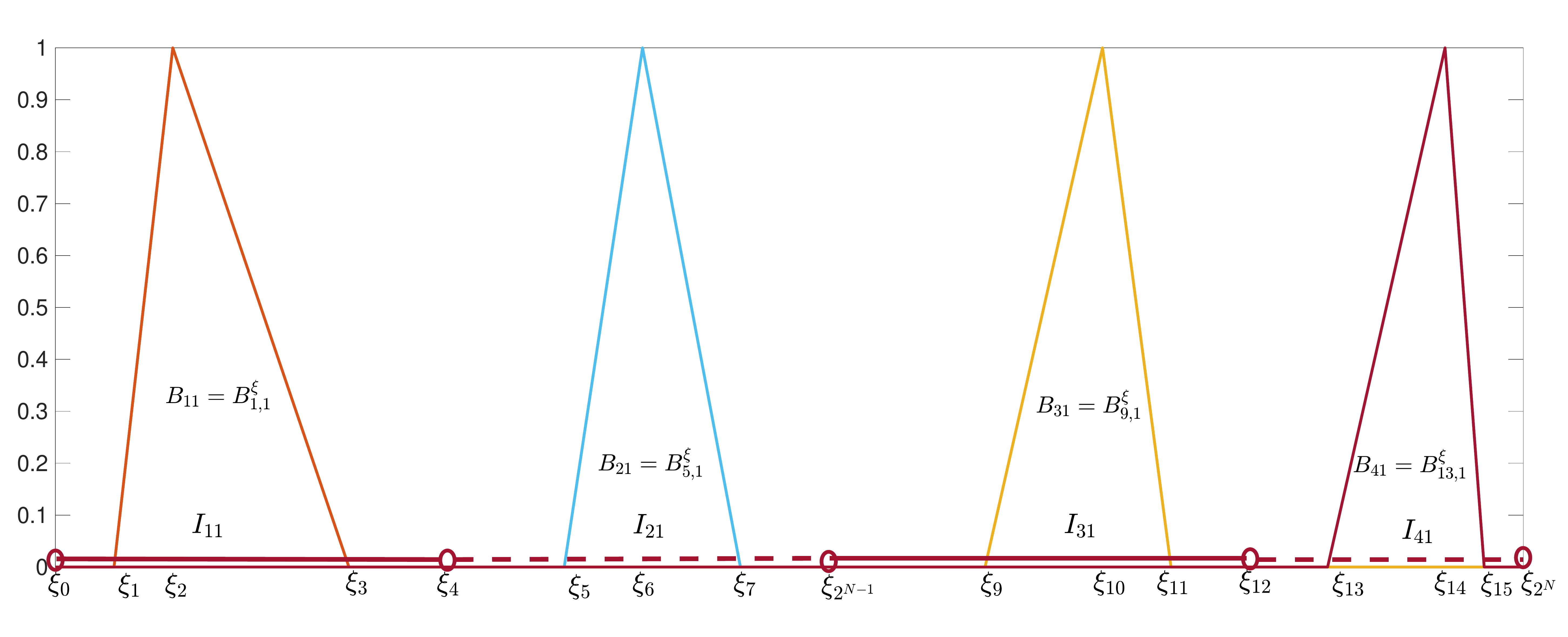}
\includegraphics[width=\textwidth]{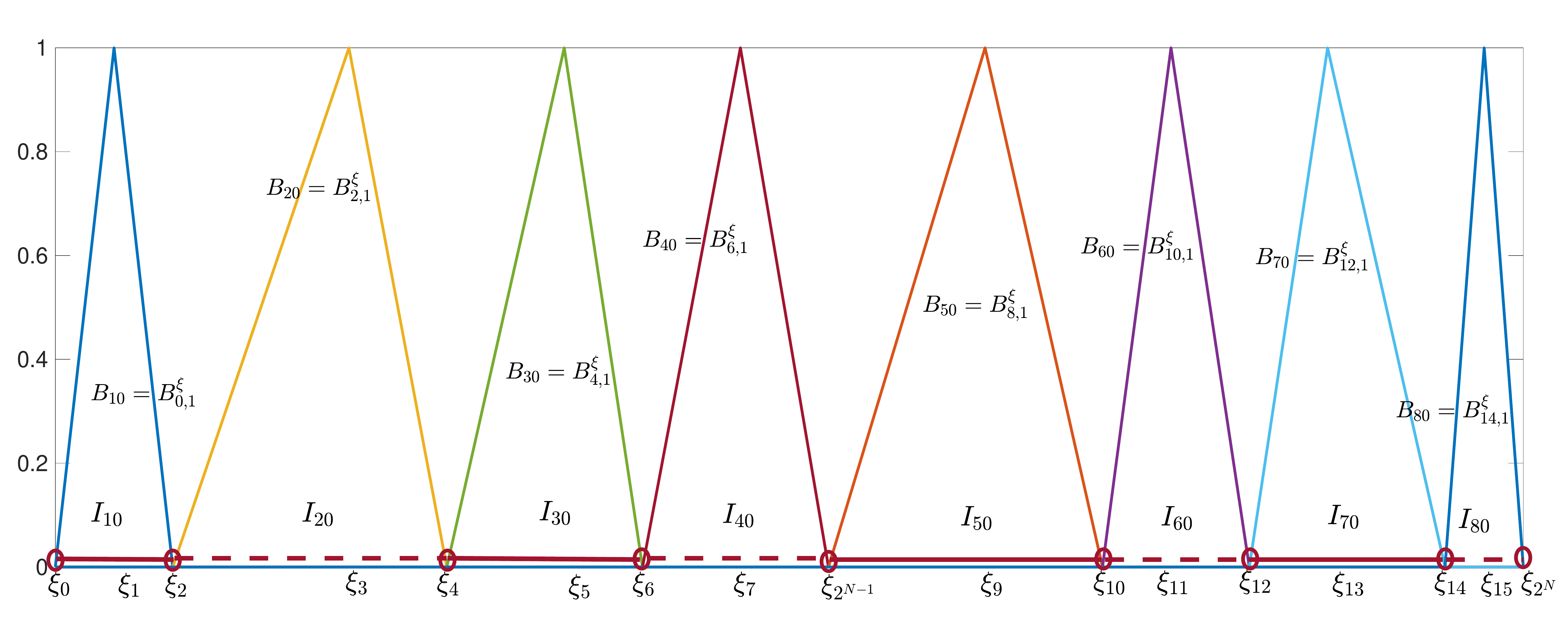}
\end{minipage}
\begin{minipage}{0.52\textwidth}
\includegraphics[width=1\textwidth, height=12.9cm]{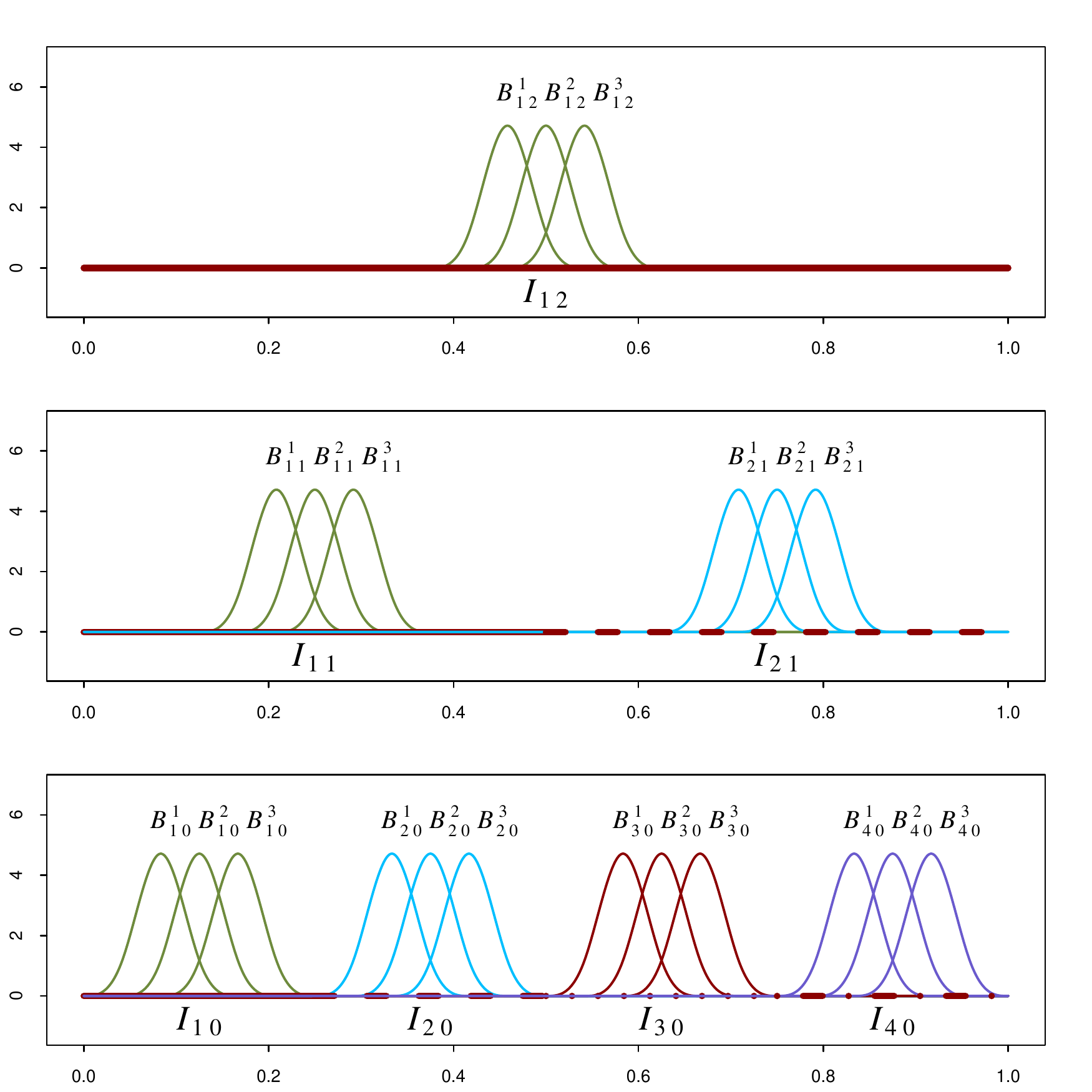}
\end{minipage}
  \caption{{\it Left: }The dyadic structure of the support sets and corresponding $B$-splines for the case of $N=4$, i.e. $n=15$.  {\it Right:} 
The dyadic structure of the support sets and cubic B-splines  for the case of $N = 3$, i.e. $n = 23$. On the horizontal axis, we denoted by alternating solid and dashed lines the support sets structure for the four (left) and  three (right) support levels }
  \label{fig:DydStr}
\end{figure}
For larger support ranges we proceed recursively. 
For  the support level $l\in \{ 0,\dots, N\}$, let us recursively define   the support range interval  $I_{r,l}= (\xi_{r,l}^L, \xi_{r,l}^R]$, $r=1,\dots, 2^{l} $, where 
$\xi_{r,l}^L$ is left knot of the $r^{\rm{th}}$ support interval at level $l$ and $\xi_{r,l}^R$ is right knot of this interval, with the central point $\xi_{r,l}^C$ so that
$$
\xi_{r,0}^L=\xi_{2r-2} ,~~\xi_{r,0}^C=\xi_{2r-1},~~\xi_{r,0}^R=\xi_{2r}.
$$
Given the definition for $l\ge 0$ one defines the support of the order $l+1$ by grouping the two neighboring intervals $I_{r,l}$, i.e. $I_{r,l+1}=(\xi_{2r-1,l}^L, \xi_{2r,l}^R]$, $r=1,\dots, 2^{N-l-1}$,  in the same manner as in the first step, with the central point $\xi_{r,l+1}^C= \xi_{2r-1,l}^R=  \xi_{2r,l}^L $  being the common knot of the two combined intervals of the lower support range that make $I_{r,l+1}$. 
 The right,  the left and the central knots of the $r^{\rm th}$ support interval at level $l$ define in terms of the original knots $\xi_{0}<\xi_{1}<\dots <\xi_{2^N-1}< \xi_{2^N}$  as follows
 \begin{align}
 \label{eq:knots1}
\xi_{r,l}^L = \xi_{(2r-2) 2^{l}} ;  \quad \quad  \xi_{r,l}^C =\xi_{(2r-1) 2^{l}}  ; \quad \quad   \xi_{r,l}^R = \xi_{(2r) 2^{l}}.
\end{align}

We proceed like this until the single interval $I_{1,N-1}=(\xi_{0},\xi_{2^N}]$ of the zero support range with $\xi_{1,N-1}^C=\xi_{2^{N-1}}$. 
We point that the total size of the interval at any given support level is still equal to $\xi_{2^N}-\xi_{0}$. 
For illustration see \ref{fig:DydStr}~{(Left)}.  

Now, we explain how the sequential structure of the $B$-splines is transformed into the dyadic structure. 
As argued before, the dimension of the linear space spanned by the $B$-splines is  $2^{N}-1$.
 Once we have the support interval structure we can relate the first order $B$-splines $B_{r,l}$ with the corresponding support interval $I_{r,l}$ by taking the central point $\xi_{r,l}^C$ of $I_{r,l}$ and taking for $B_{r,l}$ this unique $B$-spline that has $\xi_{r,l}^C$ as its central point. See \ref{fig:DydStr}~{\it(Left)} for illustration how the so-rearranged $B$-splines relate to the original sequentially ordered $B$-splines.  
 
 \subsubsection*{The case of an arbitrary $k$} 
 First, we  bind every $k$ adjacent $B$-Splines into a group called a $k$-tuplet
 \begin{equation}\label{eq:tuplet}
\mathbf B_{i}=(B_{ik,k}^{\boldsymbol \xi},\dots, B_{(i+1)k-1,k}^{\boldsymbol \xi}),~~i=0,\dots, 2^N-2.
\end{equation}
 By this construction, the total number of $k$-tuplets is  $2^{N}-1$.
\ref{fig:DydStr}~{(Right)} is an example of a dyadic structure for splines of order three. 

The `smallest support range' is at the level $0$ and consists of $2^{N-1}$ neighboring tuplets that are not intersecting.  
These smallest support intervals are denoted by $I_{r ,0}$, $r=1,\dots, 2^{N-1}$.
Assuming that we have the knots $\xi_{0}<\xi_{1}<\dots <\xi_{k2^N-1}< \xi_{k2^N}$, we consider $I_{r,0}=(\xi_{2k(r-1)},\xi_{2kr}]$, with the central knot $\xi_{2kr-k}$ in each $I_{r,0}$.

For larger support intervals denote the support interval $l\le N-1$  by $I_{r,l}= (\xi_{r,l}^L, \xi_{r,l}^R]$, $r=1,\dots, 2^{N-l-1} $, where 
$\xi_{r,l}^L$ is left hand side knot of the $r^{\rm{th}}$ support interval at the level $l$ and $\xi_{r,l}^R$ is the right hand side knot of this interval, with the central point $\xi_{r,l}^C$ so that
\begin{equation*}
\xi_{r,0}^L=\xi_{2k(r-1)} ,~~\xi_{r,0}^C=\xi_{2kr-k},~~\xi_{r,0}^R=\xi_{2kr},
\end{equation*}
and similarily as before the support interval of the order $l+1$ is obtained by merging the two neighboring intervals $I_{r,l}$, i.e. $I_{r,l+1}=(\xi_{2r-1,l}^L, \xi_{2r,l}^R]$, $r=1,\dots, 2^{N-l-2}$ with the central point $\xi_{r,l+1}^C= \xi_{2r-1,l}^R=  \xi_{2r,l}^L $. 
Similarly as before
 \begin{align*}
\xi_{r,l}^L = \xi_{2k(r-1) 2^{l}} ;  \quad \quad  \xi_{r,l}^C =\xi_{(2kr-k) 2^{l}} ; \quad \quad   \xi_{r,l}^R = \xi_{(2kr) 2^{l}} .
\end{align*}

We proceed like this until the single interval $I_{1,N-1}=(\xi_{0},\xi_{k2^N}]$ of the zero support interval with $\xi_{1,N-1}^C=\xi_{k2^{N-1}}$. 
On the so defined dyadic structure of supports, we build from the original sequence of $B$-splines the dyadic structure of the $k$-tuplets in the same fashion as one can build it for the case of $k=1$ except we replace individual splines by $k$-tuplets. 
The dyadic net (pyramid, tree) of $k$-tuplets having $N$ rows is denoted as $\mathcal B=\left\{\mathbf B_{i,l}, i=1,\dots, 2^{N-l-1},~~ l=0,\dots, N-1 \right\}$.
  
\begin{algorithm}[t!]
\caption{Recursive step $\bar{\mathcal B} \mapsto (\mathcal D(\bar{\mathcal B}\,),\mathcal R(\bar{\mathcal B}\,))=(\widetilde{\mathcal B},\mathcal{OB})$.}
\label{alg:DySp}
\begin{algorithmic}[1]\vspace{2mm}
\State{ 
\parbox{0.11\textwidth}{ \hspace{-3mm} \bf Input:} 
\parbox[t]{0.85\textwidth}{ \small
 \hspace{-5mm} Dyadic structure of the $k$-tuplets, 
 $$
 \bar{\mathcal B}=\left\{ \bar{\mathbf B}_{i,l}, l=0,\dots, \bar N-1, i=1,\dots,2^{\bar N- l-1}\right\}. 
 $$ }
}
\State{
\parbox{0.11\textwidth}{ \hspace{-3mm} \bf Output:\vspace{-5mm}}\parbox[t]{0.85\textwidth}{
\begin{description}   \small
\item[\hspace{-1mm}  \sc 1.]   Dyadic structure of the $k$-tuplets 
{
$$\widetilde{\mathcal B}=\left\{ \widetilde{\mathbf B}_{i,l}, l=0,\dots,\bar N-2, i=1,\dots,2^{\bar N -l-1} \right\}$$
}
\hspace{-12mm} that is orthogonalization of 
{
 $
\left\{ \bar{\mathbf B}_{i,l}, l=1,\dots, \bar N-1, i=1,\dots,2^{\bar N- l-1}\right\}
 $ }\\
\mbox{}\hspace{-12mm}  with respect to $\bar{\mathbf B}_{i,0}$,  $i=1,\dots,2^{\bar N-1}$; \vspace{2mm}
\item[ \hspace{-2mm} \sc 2.] Symmetric orthonormalization 
{\small 
$$\mathcal{OB}=\left(\mathbf{OB}_{i, 0}, i=1,\dots,2^{\bar N-1}\right)$$} 
\hspace{-12mm} of the $k$-tuplets in the bottom row of $\bar{\mathcal B}$
 \end{description}
\vspace{2mm}
}
}
\State{
 \small
{\sc Step 1.} Using symmetric orthonormalization \ref{eq:chb} obtain
 $$
\mathbf{OB}_{i,0}=\mathcal G\left( \bar{\mathbf B}_{i, 0}\right), ~~~ i=1,\dots,2^{\bar N-1};
$$
{\sc Step 2.} For each  $k$-tuplet $\bar{\mathbf B}_{i,l}$, $l=1,\dots,\bar N-1$, $i=1,\dots,2^{\bar N- l-1}$, orthogonalize its entries with respect to  
$\mathbf{OB}_{i,0}$,  $i=1,\dots,2^{\bar N-1}$ using \ref{eq:recst} to obtain for $l=0,\dots,\bar N-2$ and $i=1,\dots,2^{\bar N- l-2}$:
$$
\widetilde{\mathbf B}_{i,l}=\mathbf D(\bar{\mathbf B}_{i,l-1});
$$
}
\end{algorithmic}
\end{algorithm}
\subsection{The dyadic algorithm}
\label{subsec:da}
For the set of B-splines with a dyadic structure, orthogonalization is performed using recursion. 
In the main recursion step implemented in  \ref{alg:DySp}, we consider  an input dyadic structure of splines $\bar{\mathcal{B}}$, with $\bar N$ levels. 
We start with defining the GS symmetric orthonormalization of the $k$-tuplets  in the bottom row in this structure, i.e.  ${\mathbf{OB}}_{i,0}=\mathcal G\left(  \bar{\mathbf{B}}_{i,0} \right) $, $i=1,\dots 2^{\bar N-1}$, where $\mathcal G$ is defined in \ref{eq:chb}.
The bottom row output from the recursion becomes 
\begin{equation*}
\mathcal{R}\left(\overline{\mathcal{B}}\, \right)\stackrel{\tiny def}{=}\left({\mathbf{OB}}_{i,0} \right)_{i=1}^{2^{\bar N-1}}.
\end{equation*}
 
Further,  we perform the  orthogonalization of a $k$-tuplet 
 at level $l$ with respect to the pairs of the orthonormalized $k$-tuplets at the level zero (bottom)  according to
\begin{multline}\label{eq:recst}
\mathbf D \left(\bar{\mathbf  B}_{i,l}\right)\stackrel{\tiny def}{=}\\
{\displaystyle \left(\!\bar{B}^s_{i,l}-\!
\sum\limits_{m=1}^k \! \left(\left \langle \! \bar{B}^s_{i,l}, {OB}^m_{r_{il}, 0} \right \rangle \!{OB}^m_{r_{il}, 0} \!+  \left \langle \! \bar{B}^s_{i,l}, {OB}^m_{r_{il}+1, 0} \right \rangle \! {OB}^m_{r_{il}+1,0}\right)\! \right)_{s=1}^k\!,}
\end{multline}
where $r_{il}=2^{l}(2i-1)$ and $i=1,\dots,2^{\bar N - l-1}$, $l=1,\dots,\bar N-1$. 
Using this operation on the $k$-tuplets we define an operation on a dyadic structure through
\begin{equation}
\label{eq:bigD}
\mathcal D\left(\bar{\mathcal B}\,\right)
=\left(
\mathbf D
\left(
\bar{\mathbf B}_{i,l}
\right), i=1,\dots, 2^{\bar N - l-1}, l=1,\dots, \bar N-1 
\right).
\end{equation}
The operation orthogonalizes with respect to $\mathcal R(\bar{\mathcal B}\,)$ all the rows  of the dyadic structure that are above the bottom row.  
Thus for a given $\bar{\mathcal{B}}$, \ref{alg:DySp} returns the pair 
$$
\left(\mathcal D\left(\bar{\mathcal B}\right) ,\mathcal R\left(\bar{\mathcal B}\right)  \right).
$$
The dyadic structure $\mathcal D(\bar{\mathcal B})$ has one row less than $\bar{\mathcal B}$ while  $\mathcal R(\bar{\mathcal B})$ is made of the orthogonalized splines in the bottom row of $\bar{\mathcal B}$,

This recursion is applied  in \ref{alg:DySp0} to a decreasing sequence of the dyadic structures ${\mathcal{B}}_l$, $l=0,\dots,N-1$ until only one support level is left in it. 
Namely, we start by defining 
\begin{equation}
\label{eq:zero}
\mathcal{PO}_0
\stackrel{def}{=}
\left(
{\mathcal{B}}_0,\mathcal{OB}_0
\right)
\stackrel{def}{=}\left(\mathcal D\left({\mathcal B}\right) ,\mathcal R\left({\mathcal B}\right)  \right).
\end{equation}
We treat $\mathcal{OB}_0$ as the bottom row of a dyadic structure  $\mathcal{PO}_0$ with $N$-rows that is a partially orthogonalized. 
Assume that we have defined $\mathcal{OB}_l$ and ${\mathcal B}_l$, and thus $\mathcal{PO}_j$, for some $j$ such that $0\le l < N-1$.
We  then define $\mathcal{PO}_{l+1}=(\mathcal{B}_{l+1},\mathcal{OB}_{l+1})$ through
\begin{equation}
\label{eq:re}
\begin{split}
 {\mathcal{OB}}_{l+1}&=\left(\mathcal R\left({{\mathcal B}_l}\right), {\mathcal{OB}}_l\right),\\
{\mathcal{B}}_{l+1}&=\mathcal D\left({{\mathcal B}_l}\right),
\end{split}
\end{equation}
so that ${\mathcal{OB}}_{l+1}$ contains $l+2$ bottom rows of an orthogonalized dyadic structure, while ${\mathcal{B}}_{l+1}$ is a dyadic structure with $N-l-2$ rows with splines not yet orthognalized within themself but orthogonal to ${\mathcal{OB}}_{l+1}$.

\begin{algorithm}

\caption{Dyadic algorithm.}

\label{alg:DySp0}

\begin{algorithmic}\vspace{2mm}

\State{\parbox{0.11\textwidth}{ \hspace{-3mm} \bf Input:} \parbox[t]{0.85\textwidth} {
 \hspace{-5mm} $\boldsymbol \xi=(\xi_0,\xi_1,\dots, \xi_n,\xi_{n+1})$ -- knots, $k$ -- the order of $B$-splines. }\vspace{2mm}}

\State{\parbox{0.11\textwidth}{ \hspace{-3mm} \bf Output:} \parbox[t]{0.85\textwidth}
{The splinet with dyadic structure of the $k$-tuplets {\small $\mathcal{OS}=\{ \mathbf{OB}_{i,l}, l = 0,...,N-1, i = 1,...,2^{N-l-1} \}$}, where $l$ is a support level and $i$ is the index of tuplets at level $l$.
 }\vspace{2mm}} 

{\sc Step 1.} Generate the $B$-splines spread over $\boldsymbol \xi$ and group them into the $k$-tuplets
$
\mathbf B_{i},~~i=0,\dots, 2^N-2
$ as in \ref{eq:tuplet};\;

{\sc Step 2.} Rearrange the $k$-tuplets  to a dyadic net $\mathcal B$ made of the $k$-tuples $\mathbf B_{i,l}$,  where $l=0,\dots, N-1$ is a support level and $i=1,\dots, 2^{N-l-1}$;\;\vspace{2mm}

{\sc Step 3.} \\
\For{ $l=0,\dots, N-1$} 
\State $$
  (\mathcal B,\mathcal{OB}):=(\mathcal D(\mathcal B),\mathcal R(\mathcal B))\;
$$
Add $\mathcal{OB}$ as the $l$th row in $\mathcal{OS}$ \;
\EndFor  
\end{algorithmic}
\end{algorithm} 

The recurrence stops at the moment when 
\begin{equation}\label{eq:stop}
\mathcal{OS}\stackrel{def}{=}\mathcal{OB}_{N-1}=\left(\mathcal R\left({{\mathcal B}_{N-2}}\right), {\mathcal{OB}}_{N-2}\right)
\end{equation}
is defined. In this step $\mathcal D\left({{\mathcal B}_{N-2}}\right)$ is not performed since ${\mathcal B}_{N-2}$ is made of one row with one $k$-tuple in it.   
The outcome from the algorithm becomes the  splinet  $\mathcal {OS}$. 
In \ref{fig:OB}, we see the two splinets obtained by the application of the dyadic algorithms to the B-splines with the dyadic structure presented in \ref{fig:DydStr}.

\begin{figure}[htbp]
  \centering
 \begin{minipage}{0.47\textwidth}
\includegraphics[width=\textwidth]{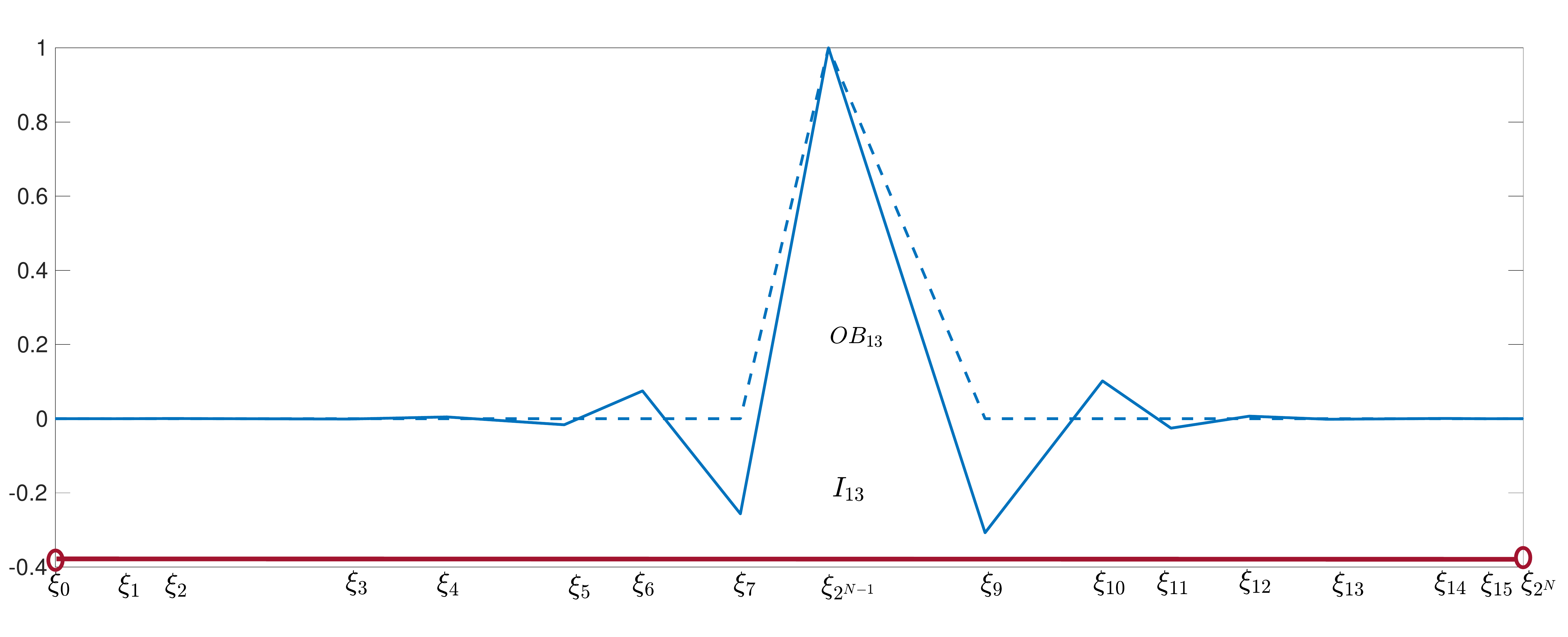}
\includegraphics[width=\textwidth]{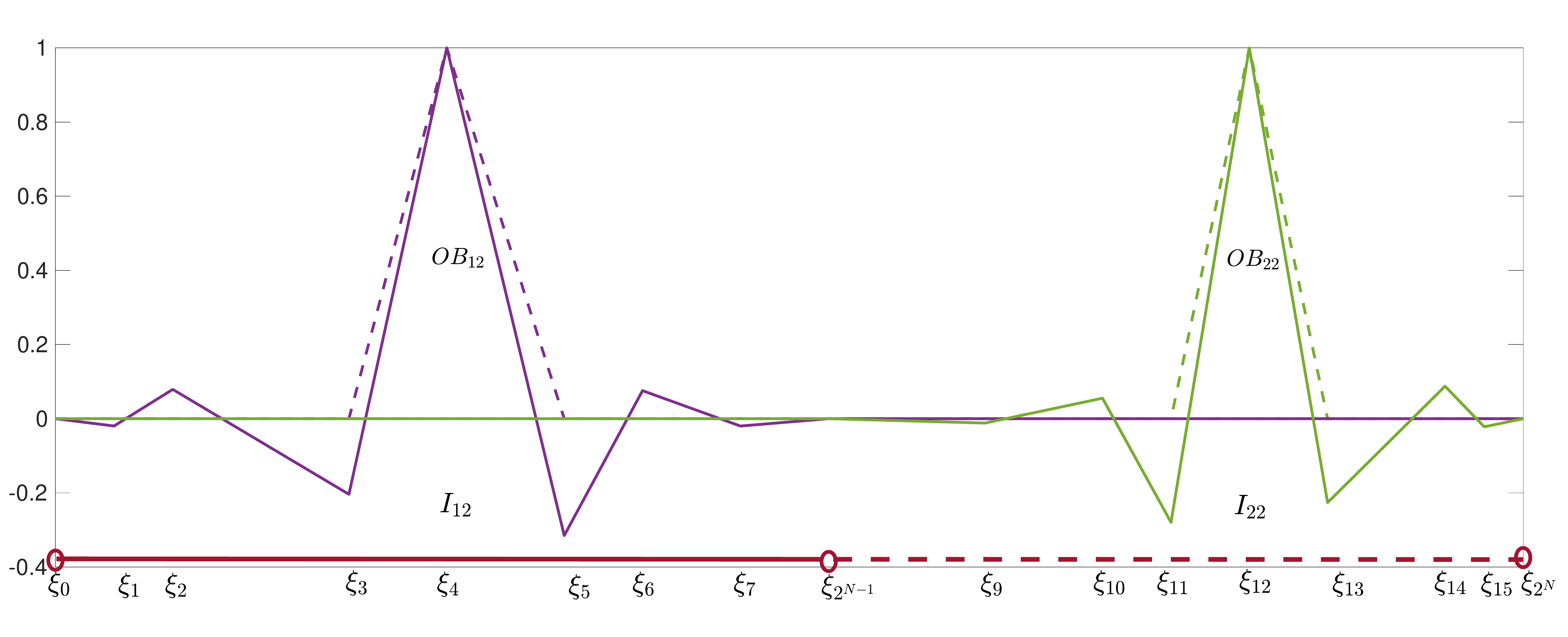}
\includegraphics[width=\textwidth]{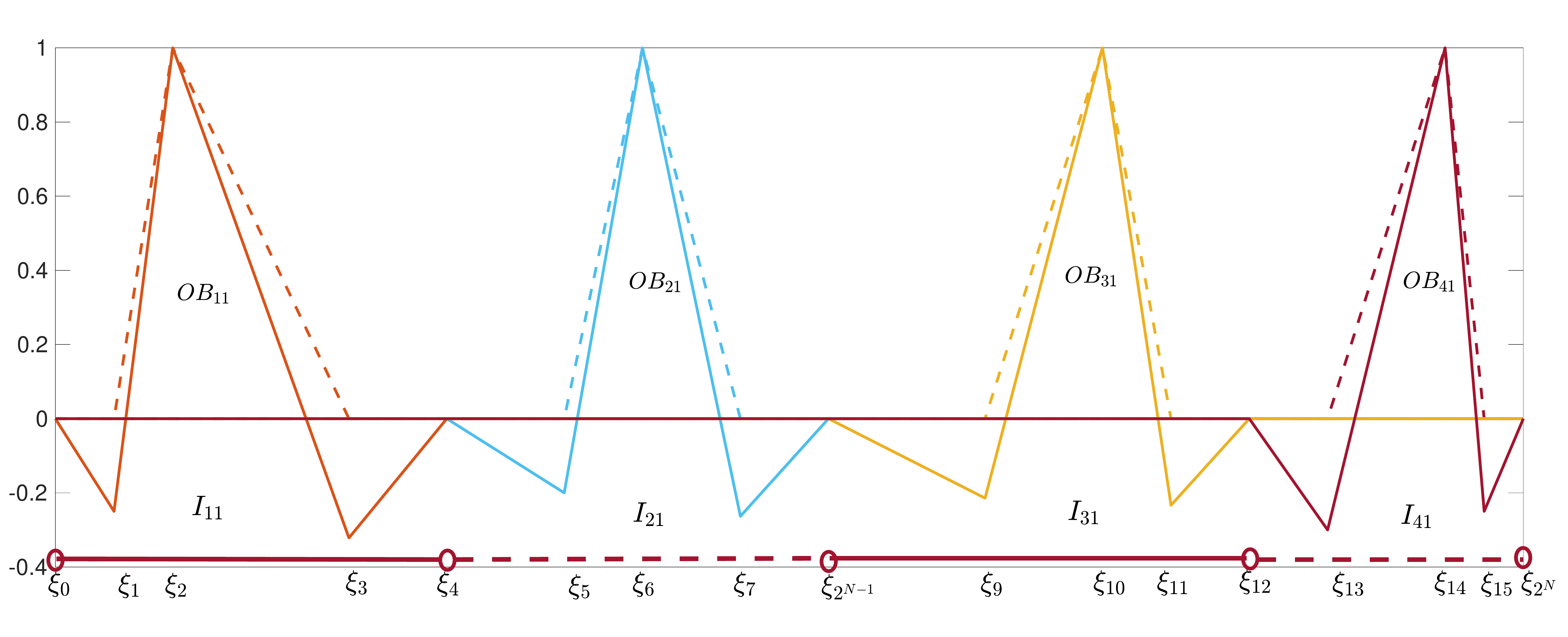}
\includegraphics[width=\textwidth]{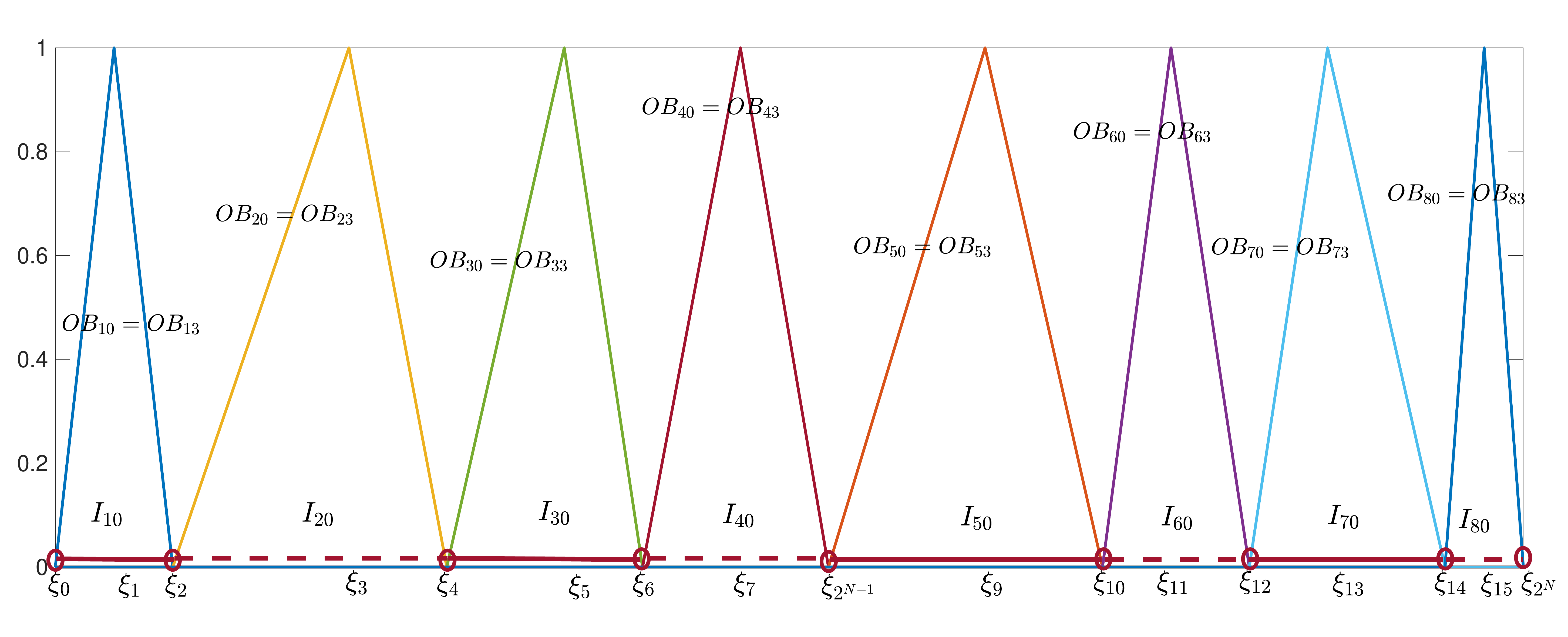}
\end{minipage}
\begin{minipage}{0.52\textwidth}
\includegraphics[width=1.0\textwidth, height=12.9cm]{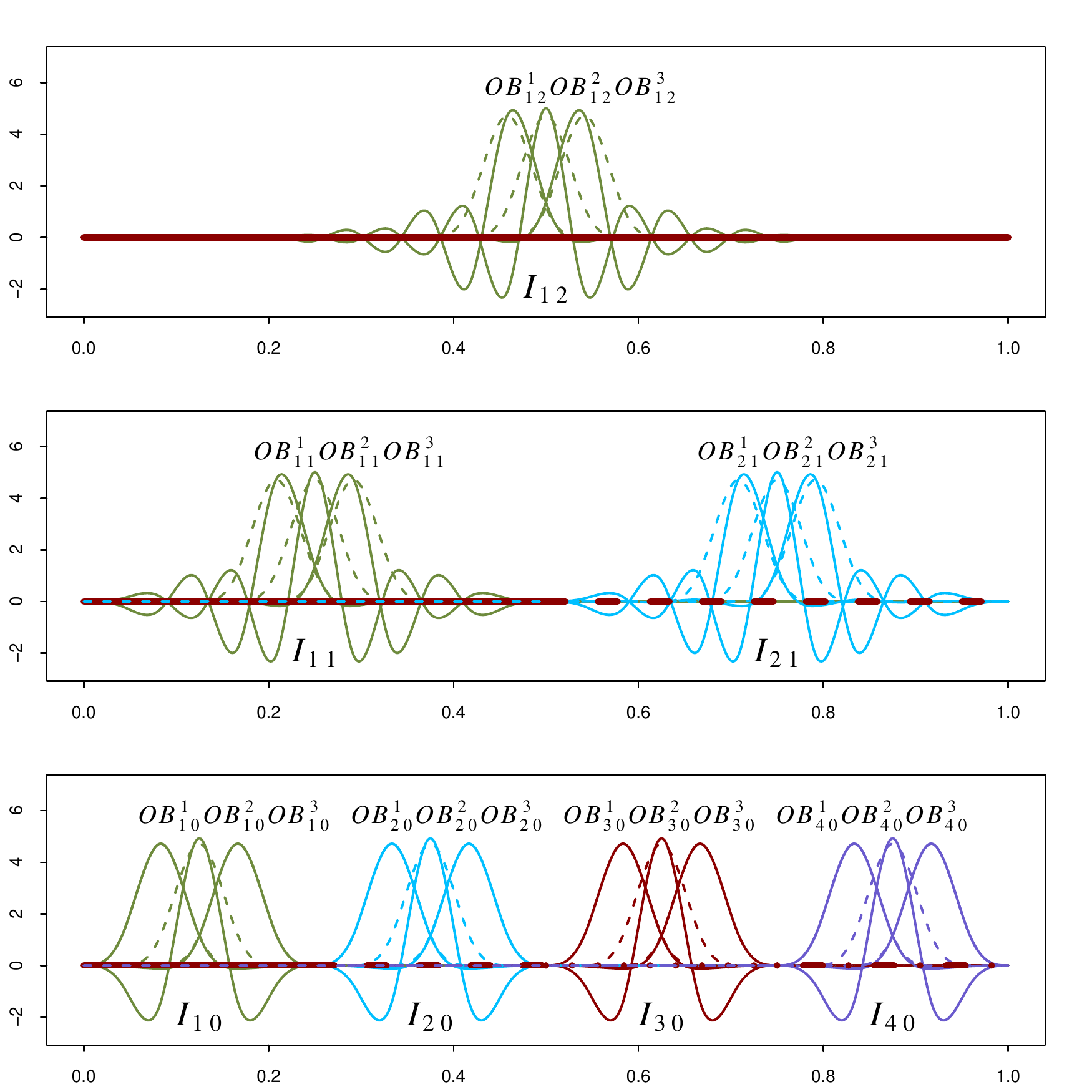}
\end{minipage}
  \caption{{\it (Left): } The construction of a splinet of the first order for the dyadic case $N=4$. The hierarchical with respect to support range orthogonalization was performed on the corresponding $B$-splines shown in \ref{fig:DydStr}~{\it (Left)} and also seen in the dashed lines in the top three graphs. 
{\it (Right):} The construction of a splinet of the 3rd order for the dyadic case $N = 3$ and equispaced knots. 
The hierarchical with respect to support range orthogonalization was performed on the corresponding $B$-splines shown in \ref{fig:DydStr}~{\it (Right)}. }
  \label{fig:OB}
\end{figure}

\subsection{Generic approach - dyadic orthogonalization}
\label{subsec:genapp}
To generalize the dyadic algorithm for a non-dyadic case, we firstly recast the dyadic algorithm into a generic Hilbert space approach. 
One can note that the Gram matrix of a set of B-splines has the special structure of a band matrix, i.e. a sparse matrix in which the nonzero elements are located in a band about the main diagonal. 
The bandwidth $m=2k-1$, where $k$ is the order of B-splines. 
Thus orthogonalization of a set of B-splines with dyadic structure can be rebranded as diagonalization of a band structure Gram matrix as explained in \ref{sec:intro}, in \ref{eq:diag}.
In this setting, we consider a set of linearly independent vectors in a Hilbert space $\mathcal X=\{x_r, r=1,\dots,d\}$.
These vectors represent in a certain basis $\mathcal E$ as columns of a matrix $\mathbf A$.  
The actual nature of this basis is irrelevant but simplify settings one can consider that $\mathcal X=\mathcal E$ and thus $\mathbf A=\mathbf I_d$, where $\mathbf I_d$ is the $d\times d$ identity matrix. 
Further $\mathbf H=[\langle x_r,x_s\rangle]_{r=1,s=1}^{d,d}$ is the Gram matrix for $\mathcal X$ and we assume that is is a band matrix with the bandwidth equal to $2k-1$ and $d=k(2^N-1)$, for some positive integer $N$. 
The goal is to define a transformation $\mathcal X \mapsto \mathcal Y$ such that the Gram matrix for $\mathcal Y$ is an identity matrix and $\mathcal Y$ coincides with the splinet obtained in the previous section.
In the terms of matrix representation of the problem we define an algorithm that transforming an input: $(\mathbf A, \mathbf H)$ into the output $\mathbf B$,
where $B$ is the representation of $\mathcal Y$ in $\mathcal E$. 
Assuming that $\mathbf A=\mathbf I_d$, than the matrix $\mathbf B$ represents the change of basis transformation of $\mathcal X$ that diagonalize $\mathbf H$, i.e.
$$
y_r=\sum_{s=1}^d B_{sr}x_s
$$
and $\mathbf I_d=\mathbf B^\top \mathbf H \mathbf B$. 
In \ref{fig:matrices} in \ref{sec:intro}, we presented two examples of $\mathbf B$ that were obtained for the dyadic $B$-splines. 

Following the recursion presented in the previous section, our algorithm is utilizing the following scheme
\begin{equation}
\label{eq:scheme}
(\mathbf A=\mathbf I_d, \mathbf H)\mapsto (\mathbf A_0, \mathbf H_0)\mapsto \dots (\mathbf A_{N-1}, \mathbf H_{N-1})=(\mathbf B, \mathbf I_d).
\end{equation}

The key recursion step is a function $\mathcal D(\tilde{\mathbf H}, \tilde N)$ of a band $\tilde d\times \tilde d$ matrix $\tilde{\mathbf H}$ with $\tilde d=k(2^{\tilde N}-1)$ and the band width $2k-1$:
\begin{align*}
(\bar{ \mathbf A}, \bar{ \mathbf H})&=\mathcal D(\tilde{\mathbf H}, \tilde N).
\end{align*}
This function which essentially corresponds to \ref{alg:DySp} is explained next in a descriptive manner, while all technical details are presented in \ref{app:dyad}. 

Let $\bar{\mathbf B}_{i}$ be the $k\times k$ output from \ref{alg:isgso} with the input $\mathbf I_k$ and $\widetilde{\mathbf H}_{i}$, which is the Gram matrix 
$$
\widetilde{\mathbf H}_{i}=\widetilde{\mathbf H}_{(2i-2)k,(2i-2)k}^{k,k},
$$ 
where $\widetilde{\mathbf H}_{r,s}^{t.u}\stackrel{def}{=}[h_{r+i,s+j}]_{i=1,j=1}^{t,u}$.
We define the columns of $\bar{\mathbf A}$ with indexes in the vector 
$$
\mathcal J=\left((2i-1)k-k+j\right)_{j=1,\dots.k,i=1,\dots,2^{\tilde N-1}}
$$
through
$$
\bar{\mathbf A}_{(2i-2)k,(2i-2)k}^{k,k}=\bar{\mathbf B}_{i},
$$
and zero everywhere else. 
We also set columns and rows of $\bar{\mathbf H}$ with indexes in $\mathcal J$ through
$$
\bar{\mathbf H}_{(2i-2)k,(2i-2)k}^{k,k}=\mathbf I_k
$$
and zero everywhere else. 
This step corresponds to the symmetric GS orthonormalization of the `lowest' level in the dyadic structure.

\begin{algorithm}[t!]

\caption{Generic Hilbert space dyadic algorithm for band Gram matrices}

\label{alg:genalg}

\begin{algorithmic}\vspace{2mm}

\State{\parbox{0.11\textwidth}{ \hspace{-3mm} \bf Input:} \parbox[t]{0.85\textwidth} {
 \hspace{-5mm} A band Gram matrix { $\mathbf{H}$}, with bandwidth { $2k-1$}}\vspace{2mm}}

\State{\parbox{0.11\textwidth}{ \hspace{-3mm} \bf Output:} \parbox[t]{0.85\textwidth}{
 $\mathbf{B}$, $d \times d$ matrix such that $\mathbf I_d=\mathbf B \mathbf H \mathbf B^\top$
 }\vspace{2mm}}
 
$\tilde{\mathbf A}:=\mathbf I_d$; $\tilde{\mathbf H}:=\mathbf H$; $\mathcal I:=\left(1,\dots, d\right)$;\vspace{2mm}\\
\For{  $l$ from $0$ to $N-1$} \vspace{2mm}
\State 
 $(\bar{\mathbf A},\bar{\mathbf H}):=\mathcal D(\tilde{\mathbf H}_{\mathcal I, \mathcal I}, N-l)$;\vspace{2mm}\\
\State
 $\tilde{\mathbf A}_{\cdot,\mathcal I}:=\tilde{\mathbf A}_{\cdot,\mathcal I}\bar{\mathbf A}$;  , $\tilde{\mathbf H}_{\mathcal I, \mathcal I}:=\bar{\mathbf H}$;\vspace{2mm}\\
\State
 $\mathcal I:=\left( \mathcal I_r \right)_{
 r\notin \left\{2^l(2i-1)k-k+j, j=1,\dots.k, i=1,\dots,2^{ N-l-1}\right\}
 }$;\vspace{2mm}\\
 \EndFor \vspace{2mm}\\
 
$\mathbf{B}=\tilde{\mathbf A}$;
\end{algorithmic}
\end{algorithm}

The remaining columns of $\bar{\mathbf A}$ and corresponding entries of $\bar{\mathbf H}$ are obtained through the orthogonalization with respect to the vectors from the `lowest' level.
The technical details of this step are given in \ref{app:dyad}. 
If we delete from the matrix $\bar {\mathbf A}$ the columns with indices in $\mathcal J$, the resulting matrix $\bar {\mathbf A}_0$ have the dimension $\bar d\times \bar d$, where $\bar d= k(2^{\tilde N-1}-1)$.
The corresponding $\bar d \times \bar d$ submatrix of $\bar{\mathbf H}$ (obtained by removing columns and rows with indices in $\mathcal J$) say $\bar{\mathbf H}_0$, is the Gram matrix for the vectors represented in $\bar {\mathbf A}_0$.
It is again a band matrix. 

We can now clarify \ref{eq:scheme} in the form of \ref{alg:genalg}.
An example of the matrix $\mathbf B$ for the $1533 \times 1533$ Gram matrix for the third-order $B$-splines with equally spaced knots (the case of $N=9$ and $k=3$) is shown in \ref{fig:matrices}.

\subsection{General splinet algorithm}
We close this section by presenting the complete algorithm for an arbitrary order and an arbitrary number of knots. 
Considering an order $k$ and a number of knots $n$, there exists an integer $N$ such that
$$
k2^{N-1}-1 < n \leq k2^N-1,
$$ 
namely,
 \begin{align*}
N&=\Big \lceil \frac{\log\left( \frac{n+1}{k}\right)}{\log 2}\Big\rceil,
\end{align*}
where $\lceil \cdot \rceil$ is the standard ceiling operator.
The main idea of a general algorithm is simply submerging the $m\times m$ Gram matrix $\mathbf H$, $m=n+1-k$ of the $B$-splines as the central part of a  dyadic $d \times d$ band matrix $\widetilde{\mathbf H}$, $d=k(2^N-1)$ and the bandwidth $2k-1$. 
For this, let 
\begin{align*}
n_U=\Big \lfloor \frac{d-m}2 \Big \rfloor=\Big \lfloor \frac{k2^N-n-1}2 \Big \rfloor,~~n_D=k2^N-n-1-n_U,
\end{align*}
$\lfloor \cdot \rfloor$ is the standard floor operator, and define an augmented band  matrix $\widetilde{\mathbf{H}}=\left [ \widetilde H_{ij} \right]_{i,j=1}^{d}$ through
\begin{align}
\label{eq:submerge}
\widetilde{{H}}_{ij}=\begin{cases}
1:& i=j, i\le n_U \mbox{ or } i > d-n_D,\\
H_{i'j'}:& i=n_U+i', j=n_U+j', i',j'=1,\dots, n+1-k,\\
0:& \mbox{ otherwise.}
\end{cases}
\end{align} 

\begin{algorithm}[t!]
\caption{General splinet algorithm.}
\label{alg:splinet}
\begin{algorithmic}\vspace{2mm}

\State{\parbox{0.11\textwidth}{ \hspace{-3mm} \bf Input:} \parbox[t]{0.85\textwidth} {
 \hspace{-5mm} $\boldsymbol \xi=(\xi_0,\xi_1,\dots, \xi_n,\xi_{n+1})$ -- knots, $k$ -- order.}\vspace{2mm}}

\State{\parbox{0.11\textwidth}{ \hspace{-3mm} \bf Output:} \parbox[t]{0.85\textwidth}
{The splinet {\small $\mathcal{OS}$ of the $k$th-order built over $\boldsymbol \xi$}.
 }\vspace{2mm}}
 
{\sc Step 1.}
Given the input knots $\boldsymbol{\xi}$ and order $k$, generate B-splines ${B}_{i}$, $i = 1,2,...,n-k+1$, and calculate the Gram matrix $\mathbf{H}$;\; 

{\sc Step 2.} Extend $\mathbf{H}$ to the augmented Gram matrix, $\widetilde{\mathbf{H}}$ through \ref{eq:submerge}, and by \ref{prop:augmentGrammMatrix} obtain $\mathbf P=\left[ P_{ji}\right]_{j,i=1}^{n-k+1}$; \;

{\sc Step 3.}  Obtain the splinet, $\mathcal{OS}$, by orthogonalizing the B-splines $\{ {B}_i \}_{i=1}^{n-k+1}$ through 
$$
OB_i=\sum_{j=1}^{n-k+1} P_{ji}B_j, ~~ i = 1,\dots, n-k+1.
$$
\;
\end{algorithmic}
\end{algorithm}

The augmented band matrix $\widetilde{\mathbf{H}}$ is a well-defined band matrix with a dyadic structure.
We can now formulate the central result for the general splinet algorithm.

\begin{proposition}
\label{prop:augmentGrammMatrix}
Consider an $n$ dimensional band matrix, $\mathbf{H}$, with bandwidth $k$.  
Submerge this matrix into $\widetilde{\mathbf{H}}$ as defined in \ref{eq:submerge}. 
Let $\widetilde {\mathbf P}$ be the output from  \ref{alg:genalg} applied to $\widetilde{\mathbf{H}}$. 
For $\mathbf P=\widetilde {\mathbf P}_{n_U+1,n_U+1}^{d-n_D,d-n_D}$, we have 
 \begin{align}
\label{eq:bign}
\mathbf I_m=\mathbf P^\top \mathbf H \mathbf P. 
\end{align}
\end{proposition} 

The proof is given in \ref{sec:pfsal} and the main splinet algorithm is formulated in \ref{alg:splinet}.   
To illustrate the non-dyadic case, in \ref{fig:splinet5}, an orthonormal splinet of the order $k=3$ and with $n=100$ equally spaced knots and thus with $88$ basis functions is shown.

\begin{figure}[htbp]
  \centering
\includegraphics[width=0.9\textwidth]{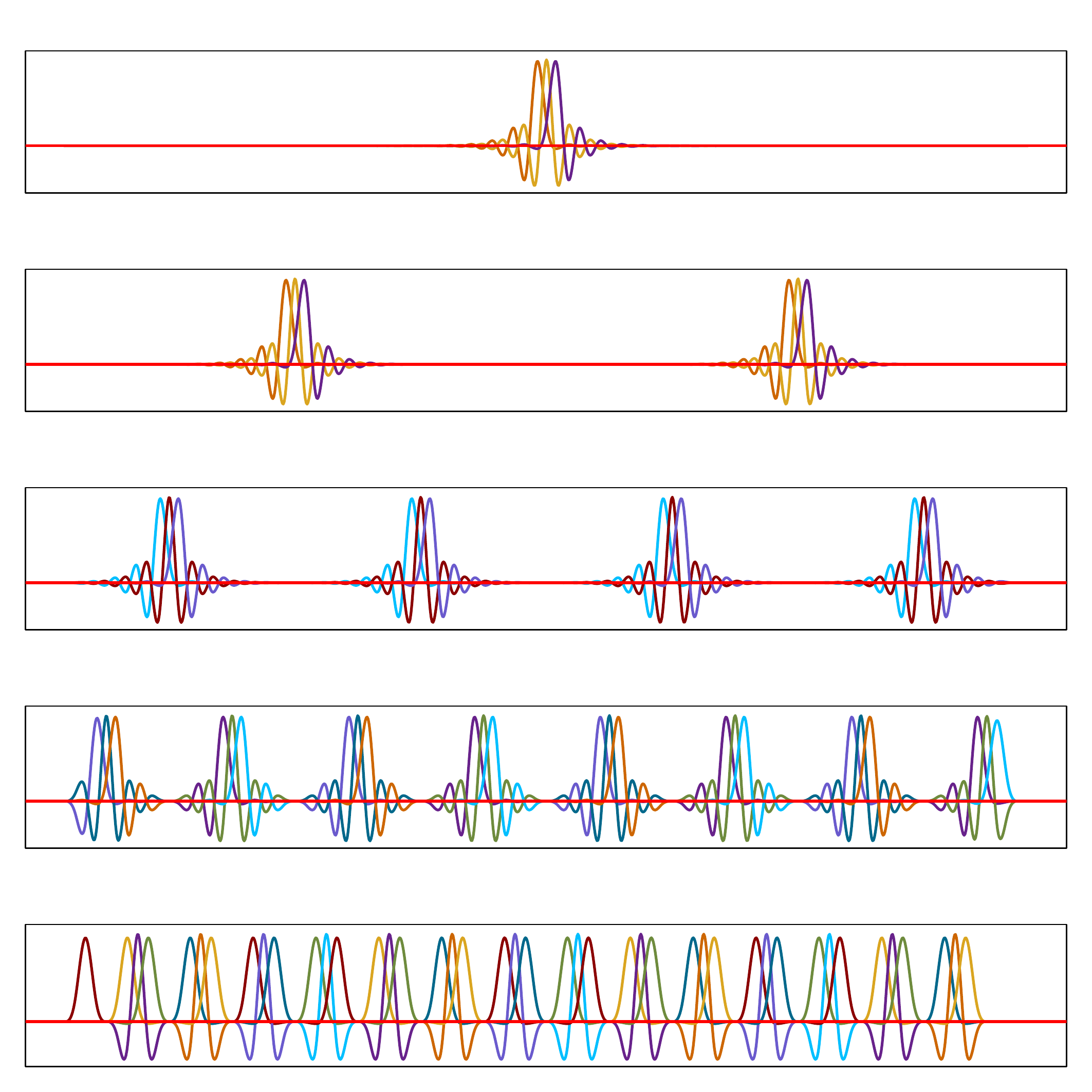}
\caption{  A non-dyadic splinet of the $3^{rd}$ order for 100 knots and with 88 basis functions.  }
 \label{fig:splinet5}
\end{figure}

\section{Efficiency of the splinets}
\label{sec:eff}
The two main features of the splinets are their locality that makes them efficient in the decomposition of a function and their efficiency in their computational evaluations.
In fact, the first feature amplifies the second one since a small support of a spline reduces the computational burden of the inner product evaluations involving such a spline. 
The main reason for the computational efficiencies of the splinets is due to near orthogonality of the $B$-splines. 
If $B$-splines in a pair are not mutually orthogonal it is only due to an, often small, overlap of their supports.  
By the nature of the dyadic net, it allows for containment of the growth of the support in the process of orthogonalization leading to a splinet. 
In the following, we restrict ourselves to discussing the dyadic case since the general splinet algorithm is based on it and its asymptotic efficiency follows the dyadic case.

\subsection{Locality}
It is easy to notice that a basis of the first order $B$-splines has the total support approximately twice the size of the knot range and thus is knot location independent. 
On the other hand, the size of the total support of the constructed first order splinets is always equal to $N-1$ multiplied by the knot range and also does not depend on the location of the knots. 
Indeed, on each support level, the total support of constructed splinets covers the whole range of knots and the conclusion follows from the fact that there are $N-1$ support levels.  
The ratio of the total support of a spline basis over the range of knots is referred to as the {\it relative support} of this basis.

For the one and two-sided symmetric orthogonalizations that constitute alternatives to our approach, the total relative support depends on the location of the knots.
In the equally spaced case is equal, for the one-sided case, to  
$$
\frac{1}2 2^{N}\left(1+1/2^N\left(1-1/2^{N-1}\right)\right).
$$
This follows from the fact that in the one-sided orthogonalization the relative support of subsequent elements in the basis are $2/(n+1), 3/(n+1),\dots, n/(n+1), 1$, where $n=2^N-1$.
Similarly, for the two-sided symmetric orthogonalization and the equal spaced knots case, the relative support is 
$$
\frac 1 4 2^{N}\left(1 -1/2^{N-1}\left(1+1/2^{N-1}\right)\right), 
$$
since if one applies one-sided orthoghonalization from both sides until the midpoint of the interval, then it 
yields two sets of the splines with the individual relative support sequence $2/(n+1),3/(n+1),\dots, (n-1)/2/(n+1)$ and one needs also to account for the final spline in the center having the relative support equal to one. 

Thus while the two-sided orthogonalization is asymptotically 50\% more efficient than the one-sided one, both are on the order of $n=2^N-1$ while the splinet's support is on the order of $\log(n+1)-1=N-1$. 
We conclude the splinets are much more efficient in the terms of the support size and only slightly worse than the non-orthogonalized B-splines which have the total relative support $2/(1+n^{-1})=2(1-2^{-N})$. 

The locality of the splinets can be even further improved if one allows for negligible errors in the orthonormalization. 
Namely, in \ref{alg:DySp0}, after few iterations of the main recursion in {\sc Step 4}, independently on $n$, the change in the form of splines is negligible and they become for all practical reasons orthogonal, see \ref{prop:appr}. 
Thus it would be natural to stop the iterations after some criterion of the accuracy is achieved. 
Thus, if the algorithm is stopped at some $l$, then the resulting splinet is obtained from $\left(\mathcal B_l, \mathcal{OB}_l\right)$ and the total support of the splines in $\mathcal B_l$ is the same as the ones in the top row in $\mathcal{OB}_l$, see also \ref{prop:totsup}.
Furthermore, since the bound for the norm of the difference between orthogonalized and not-fully-orthogonalized elements of the bases does not depend on $n$, the overall total support no longer will depend on $n$, i.e. becomes constant achieving the same (up to a multiplicative constant) asymptotic rate as the one in the original splines. 
We summarize our findings in \ref{tab:comp}. 
\begin{figure}[t!]
  \centering
  \includegraphics[width=0.46\textwidth, height=7cm]{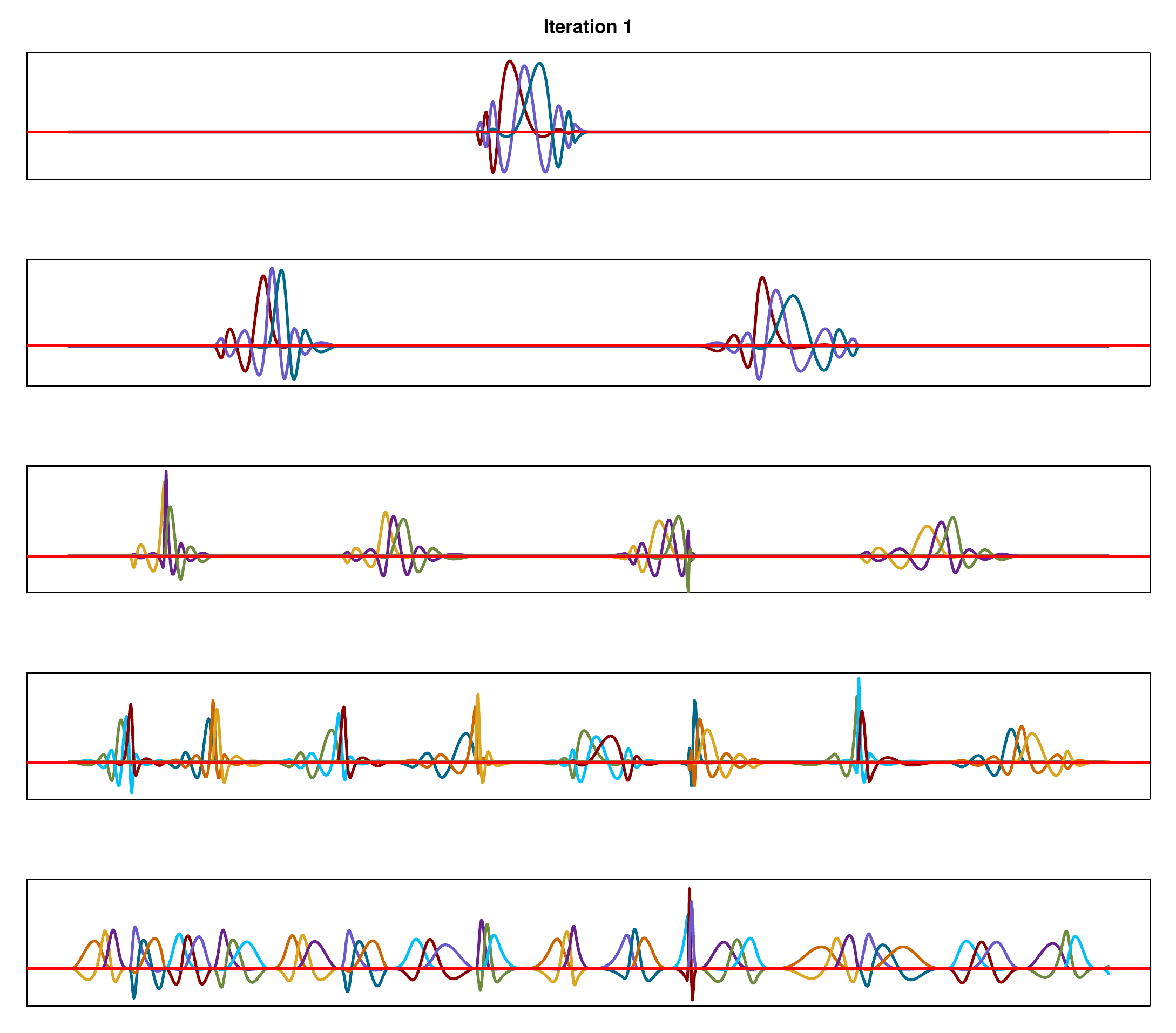} \includegraphics[width=0.46\textwidth,height=7cm]{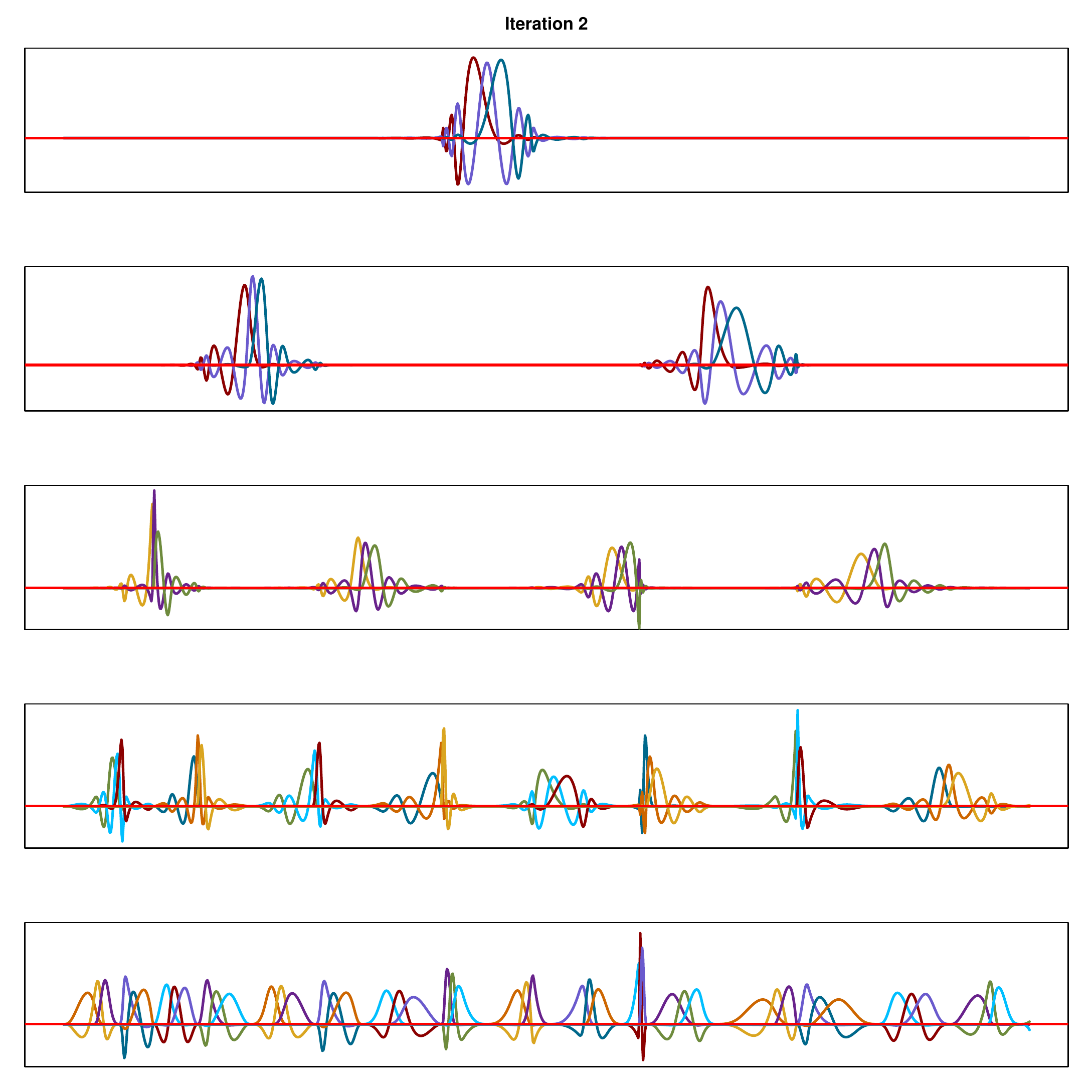}\\
  \includegraphics[width=0.46\textwidth, height=7cm]{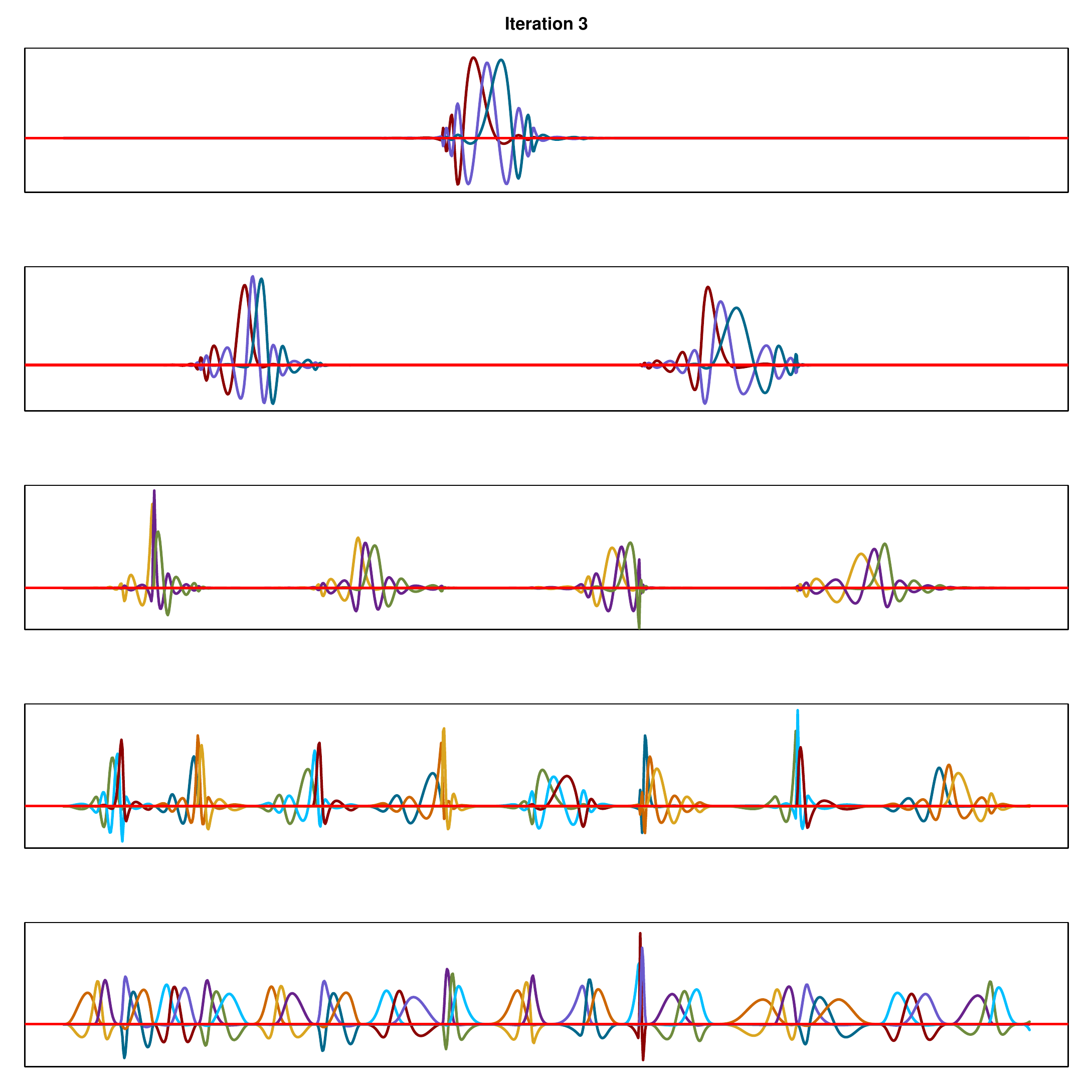} \includegraphics[width=0.46\textwidth, height=7cm]{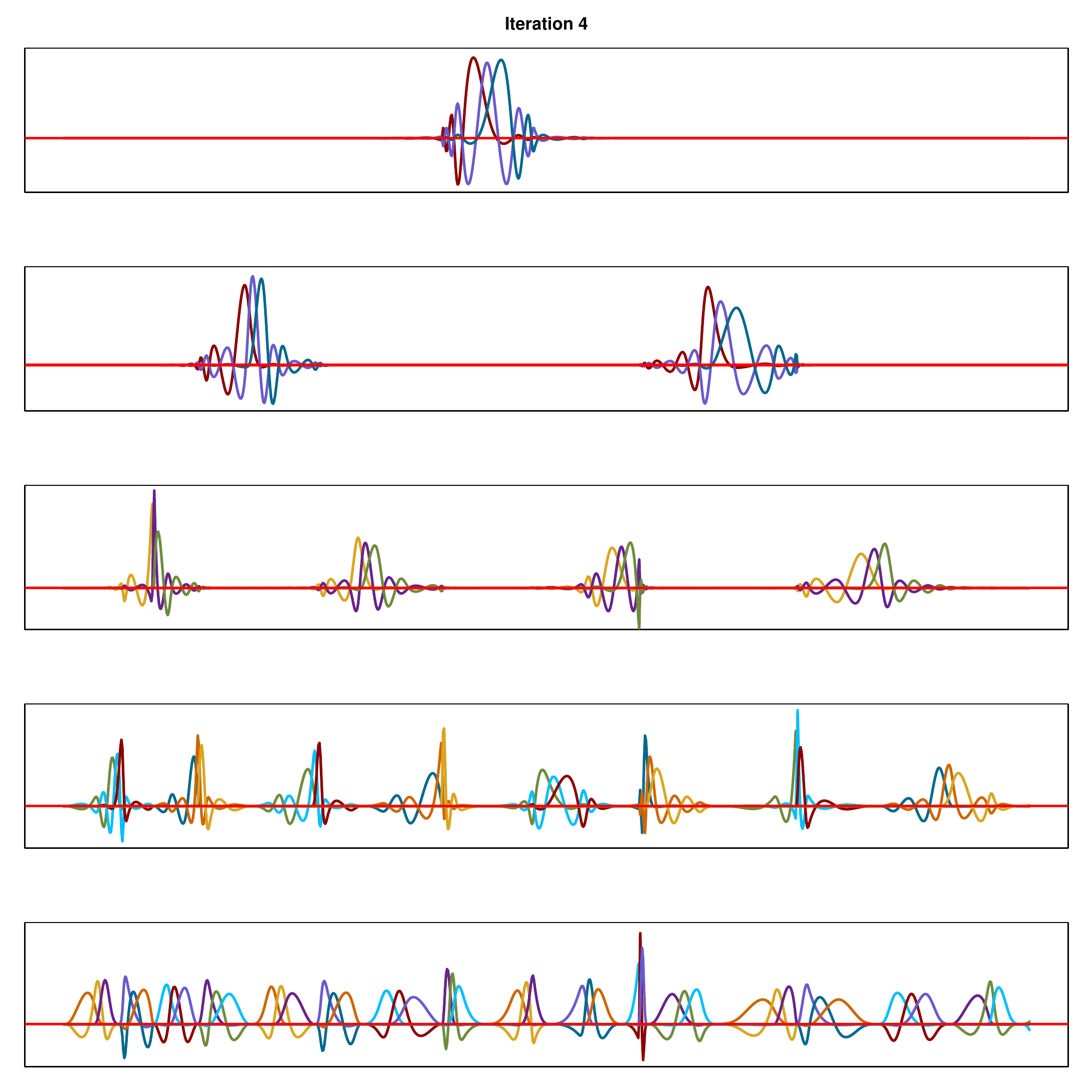}
   \caption{ {\it Top: }{\it (Left): } Splinet resulting from one iteration only of {\sc Step 4} in \ref{alg:DySp0}. {\it (Right):} 
Splinet resulting from two iterations only of {\sc Step 4} in \ref{alg:DySp0}. 
{\it (Bottom): } {\it (Left): } Splinet resulting from three iterations only of {\sc Step 4} in \ref{alg:DySp0}. {\it (Right):}   Splinet resulting from running  \ref{alg:DySp0}.  One can see the differences between the two figures at the bottom are negligible.
}
  \label{fig:illustration}
\end{figure}

%
%
%
%
The following results for the first order splinet provides mathematical foundation for the above claims and can be utilized to formulate  a stopping rule in \ref{alg:DySp0}. 

\begin{proposition}
\label{prop:appr}
Consider a dyadic splinet $\mathcal{B}$ of the first order, with $N$ levels on equally spaced knots and let  $\mathcal {OB}$ be the splinet build from $\mathcal B$, defined in  \ref{subsec:da}
Let $\mathcal{OB}_l=\{OB_{i,l}, i=1,\dots, 2^{N-l-1}\}$ for at a fixed level  $l=0,\dots, N-1$.
Moreover, let ${\mathcal{B}}_l=\{\tilde{B}_{j,r}, j=1,\dots, 2^{N-r-1}, r>l\}$,  be the $B$-splines at support levels above $l$ orthogonalized with respect to  $\mathcal{OB}_l$ as obtained in the $l$-th step of the algorithm and normalized.

Let $h_k$, $k\ge 0$ be defined through the following recurrence $h_0=1/4$, $h_{l+1}=-(2^{2/2}-2^{-1})h_l^2/(1-h_l^2)$.
Then the splines $OB_{i,l+1}$ and $\tilde{B}_{i,l+1}$ satisfy
\begin{equation*}
\|OB_{i,l+1}-\tilde{B}_{i,l+1}\|^2_2= \frac{4h_l^2}{1+\sqrt{1-2h_l^2}}
\le  \frac{8(4-\sqrt{14})}{a^2}\left( \frac a4\right)^{2^{l+1}} ,
\end{equation*}
where
$$
a= 8 \frac{2\sqrt{2}-1}{15}\approx 0.9751611 .
$$
More generally, for $r>l$, we have
\begin{equation*}
\|OB_{i,r}-\tilde{B}_{i,r}\|_2
\le
\frac{2\sqrt{2(4-\sqrt{14})}}{a}
\left( \frac a4\right)^{2^{l}}
\sum_{s=0}^{r-l-1} \left( \frac a4\right)^{2^{s}}
.
\end{equation*}
\end{proposition}
The proof of this proposition is given in \ref{sec:pfsal}. 
We note a very rapid decay of the error in using not fully orthogonalized basis. 
It leads to the following  upper bound for the total error of using the approximately orthogonalized basis. 

\begin{corollary}
\label{cor:totb}
Suppose that $y$ is an arbitrary spline of the first order over equidistant knots. 
For each $l=0,\dots, N-1$,  the partially orthogonalized linear basis made of orthonormalized splines in  $\mathcal{OB}_r$, $r\le l$ and the normalized splines $\tilde{ B}_{i,r}$, $i=1,\dots, 2^{N-r-1}$ from $\mathcal{ B}^r$ for $r > l$.
Then
\begin{align*}
\left \|
y-\sum_{r=0}^{l} \sum_{i=1}^{2^{N-r-1}} 
 \langle y, OB_{i,r}\rangle OB_{i,r} 
-\sum_{r=l+1}^{N-1} \sum_{i=1}^{2^{N-r-1}} 
\langle y, \tilde{B}_{i,r} \rangle \tilde{B}_{i,r} 
\right \|^2_2 \le 
K_{N,l} 
\left( \frac a4\right)^{2^{l}}
\|y\|_2,
\end{align*}
where 
$$
K_{N,l}=\frac{2^{N+1}\sqrt{2(4-\sqrt{14})}}{a} \sum_{r=l+1}^{N-1} \frac{\sum_{s=0}^{r-l-1} \left(  a/4\right)^{2^{s}}}{2^{r+1}} 
$$ 
with $a$ as in \ref{prop:appr}.
\end{corollary} 

\begin{remark}
From the previous proposition, one can note that the approximation is improving rapidly with an increase of $l$ due to the term $\left( a/4\right)^{2^{l}}$. 
Thus typically it is natural to stop the iterations after few steps, i.e if we stop after the 3rd iteration, then this term is $1.25e-05$ while after the 5th iteration it is $2.42e-20$.
However, the total error of using the approximately orthogonalized basis not only depends on the error due to each non-orthogonalized basis spline but also on the total number of elements of the basis, i.e. on $n$. 
Thus one should pre-determine the number of iterations through \ref{cor:totb} in order to achieve a desired uniform level of accuracy.
\end {remark}

\begin{table}[t!]
\begin{center}
\resizebox{\columnwidth}{!}{
\begin{tabular}{l  c c c c c} \toprule
\footnotesize \bf Spline basis type & \footnotesize \bf $B$-splines & \footnotesize \bf Gram-Schmidt &\footnotesize \bf symmetric ON &\footnotesize \bf splinet  & \footnotesize  \bf PO splinet  \\ \midrule
\footnotesize Relative support size & $2$ & $n/2$ & $n/4$ & $\log n/\log 2$ & \footnotesize const  \\
\footnotesize Orthogonalization &\footnotesize No & \footnotesize Yes & \footnotesize  Yes & \footnotesize Yes & \footnotesize Approx.  \\ \bottomrule
\vspace{1mm}
\end{tabular}
}
\end{center}
\caption{Comparison of the total support sizes for different spline bases. A symmetric ON basis is obtained through two-sided orthogonalization of \ref{subsec:symGS}, a PO splinet is obtained by stopping  {\sc Step 4} in \ref{alg:DySp0} at $n$ independent $l$, see also \ref{eq:stopsp}.}
\label{tab:comp}
\end{table}

The results discussed for the first order splinets can be generalized to an arbitrary order. 
Indeed, there are $N=(\log(n+1)-\log k)/\log 2$ levels and at each level, the total support of the  splinet is always equal to $k$ times the range of knots.  
If the orthogonalization recurrence in {\sc Step 4} of \ref{alg:DySp0} is stopped at some sufficiently large $l=0,\dots, N-1$, then the resulting nearly orthonormal basis is obtained as 
\begin{align}
\label{eq:stopsp}
\mathcal{OS}_l&=\left(\mathcal{B}_l , \mathcal{OB}_l\right),
\end{align}
where $ \mathcal{OB}_l$ is first $l$ rows of a complete orthogonalization of the $B$-splines as described in \ref{eq:re} carrying with itself the total support size $lk$ times the range of knots, while $\mathcal{B}_l$ is obtained from $\mathcal D({\mathcal B}_{l-1})$ given through \ref{eq:re} by normalization and having $N-l-1$ rows by applying the symmetrized orthonormalization $\mathcal G$ to each $k$-tuplet in ${\mathcal B}_l$. 
It is easy to notice that $\mathcal{B}_l$ have the total support not bigger than the top row of $ \mathcal{OB}_l$, i.e. $k$ times the range of knots. 
Thus we have the following result.
\begin{proposition}
\label{prop:totsup}
The splinet $\mathcal{OS}$ of order $k$ defined over a dyadic set of knots $\mathbf \xi=(\xi_0,\dots,\xi_{n+1})$, where $n=k 2^N-1$ for $N\ge 0$ has the relative size of the total support independent of the location of knots and equal to 
$$
k \frac{\log(n+1)-\log k}{\log 2}.
$$ 

Moreover, the partially orthonormalized splinet  $ \mathcal{OB}_l$ given in (\ref{eq:stopsp}) has the relative total support no bigger than
$
k(l+2),
$
thus if the stopping value of $l$ does not depend on $n$, then the relative support size, asymptotically in $n$, is constant. 
\end{proposition}

\begin{remark}
The results in \ref{prop:appr} for the first order splinets and equally spaced knots can be generalized to arbitrary order and not equally spaced knots.
However, the technicality of results will be quite increased and in the consequences not that much different from the first order and equally spaced case.  
We will numerically check in what follows that this is really the case. In the example we consider splines of order three see \ref {fig:illustration}, and \ref{tab:norm3} displays the norm of the error versus the number of iterations. The norm of the error measures the difference between each spline  in splinet (full iterations) and with its corresponding spline resulting from a specific number of iterations of Step 4 in \ref{alg:DySp0}, i.e. $\|\epsilon _{i,r}\|^2 =\| OB_{i,r} - \tilde{B}_{i,r}\|^2$.
We see that the third-order case produces even smaller error than the first-order case, which suggests that, in numerical implementations, one can use \ref{cor:totb} as a conservative stopping rule.  
One could investigate the approximation error for the general case  even further by using the two sided symmetric Gram-Schmidt procedure as formed in the proofs in \ref{sec:pfsal} but this does not seem to have practically important consequences. 
\end {remark}

\begin{table}[t!]
\label{tab:norm3}
\begin{center}
\begin{tabular}{l  c c c} 
\toprule
Number of itterations   & \bf  $\|\epsilon _{1,0}^j\|^2 $ & \bf  $\|\epsilon _{2,0}^j\|^2$ & \bf  $\|\epsilon _{3,0}^j\|^2$  \\ \midrule
$j=1$  & 2.90e-06 & 1.43e-05 & 6.60e-06 \\
$j=2$  & 1.71e-09 & 7.21e-09 & 3.50e-09 \\
$j=3$  & 1.78e-16 & 5.29e-16 & 1.94e-16 \\
\bottomrule \vspace{1mm}
\end{tabular}
\end{center}
\caption{Comparison between the norm of the error  in a splinet when the numbers of  iteration  of Step 4 in \ref{alg:DySp0} are one, two and three iterations. The norm of the error when having three iterations is by any count negligible.  }
\end{table}
\subsection{Computational efficiency}
One can measure computational efficiency by counting how many inner products one has to evaluate in the orthonormalization process.
For the classical GS orthogonalization of $n$ vectors, one has to evaluate $1+2+\dots + n-1=n(n-1)/2$ inner products. 
However in the case of the $B$-splines of the first order, if one goes for the one-sided orthogonalization which is based on the GS method, this number is reduced because each next spline has to be orthogonalized only with respect to the previous one since all other are already orthogonalized due to having disjoint support with the one currently orthogonalized. 
Thus the total number of the inner products that need to be evaluated is $n-1$.  
 
For a dyadic splinet of the first order, with $n=2^N-1$, the $B$-splines at the smallest support level are already orthogonalized. 
There are $2^{N-1}$ of them. 
Remaining $2^{N-1}-1$ splines over $N-1$ rows have to be orthogonalized, each one of them with respect to two splines from the bottom $N^{th}$ row as in the first run of the loop in {\sc Step 4} of \ref{alg:DySp0}.
Thus one has to evaluate $2^N-2$ inner products in the first run of the loop, then $2^{N-1}-2$ in the second and so on for each row in the dyadic structure until the top row is reached. 
Thus the total number of inner product evaluations is 
$$
\sum_{j=1}^{N-1} 2(2^k-1)=2^{N}-2 -2N.
$$
which is on the order of $n=2^N-1$, i.e. the same as in the GS procedure applied to the $B$-splines although it is better as it is reduced by $2\log(n+1)/\log 2$. 

This efficiency is preserved for any order of the splines as presented in the next result. The proof can be found in \ref{sec:pfsal}. 
\begin{proposition}
\label{prop:compeff}
Consider the dyadic structure case for the $B$-splines of order $k$. 
Then the one-sided orthogonalization requires evaluation of 
$$
J^1_n=nk-3 k^2/2 +k/2
$$ 
inner products, while the corresponding number for the splinet is
$$
J^2_n=\frac{5k-1}4 n -\frac{2k^2}{\log 2}\log(n+1) + \frac94 k-3k^2+2k^2\frac{\log k}{\log 2}.
$$
 \end{proposition} 

\newpage

\begin{center}
\vspace*{2cm} \noindent {\bf Supplementary Materials: Splinets -- orthogonalization of the $B$-splines}\\
\vspace{1cm} \noindent {\sc Xijia Liu$^\dagger$, Hiba Nassar$^\ddagger$, Krzysztof Podg\'{o}rski}$^\ddagger$\\

\end{center}

\section{Proofs and auxiliary results} \label{sec:pfsal}
In this section, we collect some proofs of the results across the paper and auxiliary results that are used in the proofs and some arguments in the main text. 
\begin{proof}[Proof of \ref{prop:dersp}]
To see the above, we notice that \ref{eq:indi} coincides with \ref{eq:recder} in the case of $i=0$ (the undefined term ${d^{-1}B_{l,k-1}^{\boldsymbol \xi}}/{dx^{-1}}$ can be neglected since it is multiplied by $0$, so one can define it, for example, equal to zero). 
For the first derivative, i.e. $i=1$, we have
\begin{multline*}
\frac{dB_{l,k}^{\boldsymbol \xi}}{dx}(x)
=
\frac{1}{\xi_{l+k+1}-\xi_l}B_{l,k-1}^{\boldsymbol \xi}(x)+\frac{ x- {\xi_{l}}
  }{
  {\xi_{l+k+1}}-{\xi_{l}} 
  } 
\frac{dB_{l,k-1}^{\boldsymbol \xi}}{dx}(x)
+\\
+
\frac{1}{\xi_{l+1}-\xi_{l+k+1}}B_{l+1,k-1}^{\boldsymbol \xi}(x)+\frac{ {\xi_{l+k+1}-x}
  }{
  {\xi_{l+k+1}}-{\xi_{l+1}} 
  } 
\frac{dB_{l+1,k-1}^{\boldsymbol \xi}}{dx}(x).
\end{multline*}
We note that if $k=1$, then ${dB_{l+1,k-1}^{\boldsymbol \xi}}/{dx}\equiv 0$.
Then a simple induction argument leads to  \ref{eq:recder}. 
\end{proof}
\begin{proof}[Proof of \ref{prop:equivalence}]
The first part is a rather obvious consequence of that the definition of $\widetilde{B}_{K+i,K}$, $i=0,\dots, n-K$ is not dependent on the location of  the external knots $\widetilde{\xi}_{0}\le \dots \le \widetilde{\xi}_{K-1}$ and $\widetilde{\xi}_{K+n+2}\le \dots = \widetilde{\xi}_{2K+n+1}$.
A simple proof of this can be obtained through mathematical induction on the order of spline and is omitted. 

The second part can be argued through mathematical induction with respect to the order of a spline as follows. 
There are $K$ splines at each of the endpoints that are affected by the passage with $h$ to zero, so it is enough to consider them.
By the symmetry argument, it is enough to consider the ones at the left-hand-side endpoint. 

\sloppy For them we will prove a stronger thesis by adding the condition that   $\sup_{h\in (0,\epsilon)} \widetilde{B}^h_{i,K}$ is bounded and for $x\in [-\epsilon+\xi_0,\xi_0]$ and also showing that 
$$
\lim_{h\rightarrow 0} \widetilde{B}^h_{i,K}(x)=0,~~x\in (\xi_0-\epsilon,\xi_0).
$$
We note that these conditions guarantee the convergence in $\| \cdot \|$. 

Let $K=1$. Then the result is obvious since the values inside $(\xi_0,\xi_{n+1})$ do not depend on the value of $h$, while the portion of the norm of the first (and the last spline) corresponding to the intervals $(\xi_0-h,\xi_0)$ (and $(\xi_{n+1},\xi_{n+1}+h)$) converge to zero with $h\rightarrow 0 $. 

Suppose that the result is valid for  all $j<K$, for a certain $K>1$.  
It is thus sufficient to show that for  $i=0,\dots, K-1$:
\begin{align}
\label{eq:P1}
\lim_{h\rightarrow 0} \widetilde{B}^h_{i,K}(x)&=\widetilde{B}_{i,K}(x),~~x\in (\xi_0,\xi_1),\\
\label{eq:P2}
\lim_{h\rightarrow 0} \widetilde{B}^h_{i,K}(x)&=0,~~x\in (\xi_0-\epsilon,\xi_0),\\
\label{eq:P3}
 \exists{ M>0}: \sup_{h\in (0,\epsilon)} |\widetilde{B}^h_{i,K}(x)| &<M, ~~x\in (\xi_0-\epsilon,\xi_1).
\end{align}
Using \ref{eq:indi} we obtain for $i=1,\dots, K-1$:
\begin{equation*}
\widetilde{B}^h_{i,K}(x)
=
 \frac{ x +h-\xi_0
  }{
  {\xi_{i}}+h-\xi_0
  } 
 \widetilde{B}^h_{i,K-1}(x)+
  \frac{{\xi_{i+1}}-x}{
 {\xi_{i+1}}-{\xi_{1}} 
 } 
 \widetilde{B}^h_{i+1,K-1}(x).
\end{equation*}
One can apply the induction assumption about $\widetilde{B}^h_{i,K-1}$ to derive all the properties as $h$ converges to zero due to existence of the proper limit for $( x +h-\xi_0)/(\xi_{1}+h-\xi_0) \rightarrow ( x -\xi_0)/(\xi_{1}-\xi_0)$. 

It remains to consider $\widetilde{B}^h_{0,K}(x)$ for which we have 
\begin{equation*}
\widetilde{B}^h_{0,K}(x)
=
 \frac{ x +Kh-\xi_0
  }{
  Kh
  } 
 \widetilde{B}^h_{0,K-1}(x)+
  \frac{{\xi_{1}}-x}{
 {\xi_{1}}+(K-1)h-{\xi_{0}} 
 } 
 \widetilde{B}^h_{1,K-1}(x).
\end{equation*}
This case needs some additional argument since the first term has the factor $(x +Kh-\xi_0)/Kh$, which is unbounded. 
The property \ref{eq:P1} for this splines follows from the fact that the first term is zero on $(\xi_0,\xi_1)$. 
The second property is obvious due to continuity of splines and since for $x<\xi_0$, eventually, with $h$ such that $\xi_0-Kh>h$ and the support of  $\widetilde{B}^h_{0,K}$ does not contain $x$. 

To argue for \ref{eq:P3} let us note that $ \widetilde{B}^h_{0,K-1}(x)$ is spread over equidistant knots and since the values of $B$-splines over equidistant knots are independent of a scale transformation of the grid. Thus if one consider $ \widehat{B}^h_{l,K}(x)$ that spreads over  equidistant knots of the form $\xi_0-(K-i)h$, $i=0,\dots,2K+n+1$, we have
\begin{equation*}
\widehat{B}^h_{0,K}(x)
=
 \frac{ x +Kh-\xi_0
  }{
  Kh
  } 
 \widetilde{B}^h_{0,K-1}(x)+
  \frac{{\xi_{1}}-x}{
 {\xi_{1}}+(K-1)h-{\xi_{0}} 
 } 
 \widehat{B}^h_{1,K-1}(x).
\end{equation*}
where the left-hand-side values are independent of $h$  due to scale invariance and thus bounded uniformly with respect to $h$. 
Since the terms on the right hand side are symmetric and positive it implies that each of them has to be uniformly bounded, which yields \ref{eq:P3}.  
\end{proof}

\begin{proof}[Proof of \ref{prop:reflect}]
The symmetry is obvious and the orthogonality follows from
\begin{align*}
2\langle \tilde x, \tilde y \rangle &= 1- \frac{\langle x+y,x-y\rangle}{\|x+y\|\|x-y\|} + \frac{\langle x+y,x-y\rangle}{\|x+y\|\|x-y\|} - 1=0.
\end{align*}
The proof of normalization is straightforward 
\begin{align*}
\| \tilde x\|^2 &= \frac{1}{2} \left( \frac{1}{\|x+y\|^2} +\frac{1}{\|x-y\|^2} + \frac{2}{\|x+y\|\|x-y\|}\right) \|  x\|^2 \\
 & \quad + \frac{1}{2} \left( \frac{1}{\|x+y\|^2} +\frac{1}{\|x-y\|^2} - \frac{2}{\|x+y\|\|x-y\|}\right) \|  y\|^2 \\
 & \quad + \left( \frac{1}{\|x+y\|^2} -\frac{1}{\|x-y\|^2} \right) \langle x, y \rangle \\
 &= \frac{1+ \langle x, y \rangle}{\|x+y\|^2} + \frac{1- \langle x, y \rangle}{\|x-y\|^2} =  \frac{1}{2} + \frac{1}{2}  = 1. 
\end{align*}
An analogous proof holds for $ \| \tilde y\|^2=1$. 
\end{proof}

\begin{proof}[Proof of \ref{prop:gssym}]
Consider first $(y_1, y_{2k})$, then it is clear that these two vectors are orthogonal to each other as they are obtained by \ref{prop:reflect} from $(x_1^L, x_{2k}^R)$ which are a linear combination of  $(x_1, x_{2k})$ and thus this proves that $(y_1, y_{2k})$ span $(x_1, x_{2k})$.
Clearly from the same result, $Sy_1=y_{2k}$ if the same holds for $(x_1, x_{2k})$.
We proceed with the proof using the mathematical induction. 

Let us assume that for $i<k$ all the claims about  $\{ y_j,  y_{n-j+1}, j \le i\}$ are true. It follows from the Gram-Schmidt method that vectors in the pair $(x_{i+1} ^L, x_{2k-i} ^R)$ are orthogonal to vectors in $\{ x_j,  x_{n-j+1}, j \le i\}$, and thus this also holds for $(y_{i+1},y_{2k - i})$. 
Since by the induction assumption we know that $\{ x_j,  x_{n-j+1}, j\le i\}$ is spanned by $\{ y_j,  y_{n-j+1}, j \le i\}$ we conclude that $y_{i+1}$ and $y_{2k - i}$ are orthogonal to $\{ y_j,  y_{n-j+1}, j \le  i\}$. 
They are also orthogonal to each other and normalized because of what \ref{prop:reflect} guarantees. 
It follows also that $y_{i+1}$ and $y_{2k - i}$ are spanning $x_{i+1}$ and $x_{2k - i}$ and thus also $\{ x_j,  x_{n-j+1}, j\le i+1\}$ is spanned by $\{ y_j,  y_{n-j+1}, j \le i+1\}$.

Moreover, if $Sx_j=x_{2k-j+1}$, $j\le 2k$, then note first that $Sx_{i+1}^L=x_{2k-i}^R)$. 
Indeed, $x_{i+1}^L$ is orthogonalized with respect to $\{ x_j,  x_{n-j+1}, j \le i\}=\{ Sx_{n-j+1},  Sx_j, j \le i\}$ and thus with respect to  $\{ y_j,  y_{n-j+1}, j \le i\}=\{ Sy_{n-j+1},  Sy_j, j \le i\}$ by the induction assumption. 
Thus $Sx_{i+1}^L$ is orthonormalization of $x_{n-i}=Sx_{i+1}$, with respect to $\{ x_j,  x_{n-j+1}, j \le i\}$, since $\langle S x, S y\rangle = \langle x, y\rangle$ for each $x$ and $y$. 
Since the Gram-Schmidt orthonormalization uniquely defines $x_{2k-i}^R$ through these conditions we need to have $Sx_{i+1}^L=x_{2k-i}^R$.
Consequently, by \ref{prop:reflect}, it must be $Sy_{i+1}=y_{2k-i}$ and this concludes the proof.
\end{proof}

\begin{proof}[Proof of \ref{cor:c2}]
All except $Sy_{k+1}=y_{k+1}$ follows from \ref{prop:gssym}. 
We can assume without loss of generality that $(y_i)_{i=1,i\ne k+1}^{2k+1}$ is normalized and let $P$ be the projection to this space of vectors. 
It is enough to show that for any $x$ such that $Sx=x$, we have
$$
SPx=PSx=Px.
$$
This follows from
\begin{align*}
SPx&=\sum_{i=1}^k\left(\langle y_i,x \rangle Sy_{i}+\langle y_{2k-i+2},x \rangle Sy_{2k-i+2}\right)\\
&=\sum_{i=1}^k\left(\langle y_i,Sx \rangle y_{2k-i+2}+\langle y_{2k-i+2},Sx \rangle y_{i}\right)\\
&=\sum_{i=1}^k\left(\langle S^*y_i,x \rangle y_{2k-i+2}+\langle S^*y_{2k-i+2},x \rangle y_{i}\right)\\
&=\sum_{i=1}^k\left(y_{2k-i+2},x \rangle y_{2k-i+2}+\langle y_i,x \rangle y_{i}\right) = Px
\end{align*}
\end{proof}

Although it is not the case for the spline systems considered in this work, in some situation one may deal with vectors that are not symmetric in the above sense. The following results can be used to obtain symmetric version of vectors to be orthogonalized if the original ones are not. 
\begin{lemma}
\label{lem:l2}
Let $x_0$ and $y_0$ be arbitrary linearly independent vectors and $T$ is a symmetry operator. Then one of the following holds
\begin{itemize}
\item[(i)] $y_0=Ty_0$ and $x_0=Tx_0$, i.e. the vectors in the pair $(x_0,y_0)$ are symmetric,
\item[(ii)] $x=x_0+Ty_0$ and $y=Tx_0+y_0$ are symmetric to each other and linearly independent, 
\item[(iii)] $x=x_0-Ty_0$ and $y=Tx_0-y_0$ are symmetric to each other and linearly independent.
\end{itemize}
\end{lemma}

\begin{proof}
We first note that the symmetries in {\it (ii)} and {\it (iii)} are always satisfied, since 
\begin{align*}
T(x_0+Ty_0)&=Tx_0+T^2y_0=Tx_0+y_0,\\
T(x_0-Ty_0)&=Tx_0-T^2y_0=Tx_0-y_0.
\end{align*}
Now, assume that both the pairs $(x_0+Ty_0, Tx_0+y_0)$ and $(x_0-Ty_0, Tx_0-y_0)$ are made of linearly dependent vectors. 

Let first assume that none of these vectors is equal to zero, then for some non zero $a$ and $b$ we have $x_0+Ty_0=a (Tx_0 +y_0)$ and $x_0-Ty_0=b (Tx_0 -y_0)$. 
By applying $T$ to both of these equalities, we obtain  $Tx_0+y_0=a (x_0 +Ty_0)$ and $Tx_0-y_0=b (x_0 -Ty_0)$, which implies that $a=b=1$ yielding $T(x_0-y_0)=x_0-y_0$ and  $T(x_0+y_0)=x_0 +y_0$. 
Therefore, both $x_0$ and $y_0$ are symmetric, i.e. $Tx_0=x_0$ and $Ty_0=y_0$ and  {\it (i)} is satisfied. 

Next we notice that one of $x_0+Ty_0$ and $x_0-Ty_0$ must be non-zero or, otherwise $x_0=0$, which contradicts assumptions. 
Then if $x_0+Ty_0=0$ but $x_0-Ty_0\ne 0$, in {(iii)} we have $x=2x_0$ and $y=-2y_0$ which are linearly independent by the assumption. 
On the other hand, if   $x_0+Ty_0\ne 0$ but $x_0-Ty_0= 0$, then $x$ and $y$ in {\it (ii)} are linearly independent. 
\end{proof}
Combining the two above lemmas guarantees that, given a symmetry operator $S$, one can obtain orthogonalized  pair of vector satisfying symmetry property starting from an arbitrary pair of linearly independent vectors, which is formally stated in the next result. 
\begin{corollary}
\label{cor:c1}
Assume a symmetry operator $T$ on a linear space. 
Let $x$ and $y$ be arbitrary two linearly independent vectors. There exists a pair $(\tilde x,\tilde y)$ of linear combinations of $x$, $y$, $Tx$, and $Ty$, such that $\tilde x$ and $\tilde y$  are  orthogonal to each other and  $\{\tilde x,\tilde y\}=\{T\tilde x,T\tilde y\}$.
\end{corollary}
\begin{proof}
If we start with vectors that follow {\it (i)} of \ref{lem:l2}, then any orthogonalization of the two vectors leads to $(\tilde x,\tilde y)$. In all other cases the result follows easily by first applying \ref{lem:l2} {\it (ii)} or {\it (iii)} and then \ref{prop:reflect}. 
\end{proof}
\begin{proof}[Proof of \ref{prop:augmentGrammMatrix}]
Let us call the vectors corresponding to the first $n_{U}$ and last $n_{D}$ columns of the augmented Gram matrix $\widetilde{H}$ the superfluous vectors and the rest as the regular vectors. Note that the augmented Gram matrix $\widetilde{H}$ is a block diagonal matrix. The blocks corresponding to the part of superfluous vectors are identity matrices. Given this structure, all the superfluous vectors are orthogonal to each other and to the regular vectors. 
Let us consider $i \in \{n_U+1,\dots, n_D\}$, then we want to show that the column vector $\widetilde{\mathbf P}_{\cdot i}$ has zeros for the coordinates below $n_U+1$ and above $n_D$. 

Let us consider $\widetilde{P}_{j i}$, for some $j\le n_U$ or $j> n_D$.
Since all vectors are merged into the $k$-tuplets let us consider first those in $\tilde{\mathbf A}_{\cdot,\mathcal J}$ for the lowest support level $l=0$ in \ref{alg:genalg}. 
There are two possible situations. First, the index $j$ could belong to a $k$-tuplet consisting only of regular vectors. 
However, in this case, $P_{ji}$ cannot be non-zero since all superfluous vectors are already orthogonal to the regular vectors.  
Second, the index $j$ corresponds to a vector in a $k$-tuplet that contains both types of vectors.
Given \ref{prop:reflect}, \ref{alg:iso}, \ref{prop:gssym} and \ref{cor:c2}, one can notice that the symmetric GS orthogonalization leaves the superfluous vectors unchanged and thus again $P_{ji}$ must be zero.

Let us now consider $i$ that correspond to a column in  $\tilde{\mathbf A}_{\cdot,\mathcal J}$ at the level $l>0$. 
If $j$ is corresponding to a superfluous vector located in some lower layer than the $l$th one, then $P_{ji}$ has to be zero since the recurrent procedure in \ref{alg:genalg} makes vectors at the given layer orthogonal to the ones below. 
Moreover, if $i$  indexes  a vector that belongs to a $k$-tuple made of only the regular vectors, then these vectors would still be orthogonal to all superfluous vectors implying $P_{ji}=0$.  
Thus the only case not considered is when $i$ indexes a vector in the $k$-tuple that is a mixture of superfluous and transformed regular vectors. 
It is clear that the transformed regular vectors remain orthogonal to all superfluous vectors and the symmetric GS orthogonalization does not change it leading to $P_{ji}=0$.  
\end{proof}

We turn to a number of results that are used to assess the approximation of the orthonormal basis by stopping the recurrence represented in \ref{alg:DySp0} at a certain value $l$. 
Let us start with a simple lemma about the inner product between two symmetric first order splines over equally spaced set of knots. 
\begin{lemma}
\label{lem:inn}
Let $l\ge 0$ and $\mathbf g=(g_0,\dots, g_{l+1})$ represent values of a first order spline $G$ at its knots $(0,1,\dots, l, l+1)$. 
Its symmetric reflection $\tilde G$ is the first order spline (over the same set of knots) represented by $\tilde g=(g_{l+1},\dots, g_0)$. 
Then 
\begin{align*}
\|G\|^2_2&=\frac 1 3
 \left \|(\tilde{\mathbf Z}^1+\tilde{\mathbf Z}_1)\mathbf g\right \|^2 \\
\langle G, \tilde G \rangle &=
 \frac{1}{6}\left( 
 4~ \mathbf g^\top \tilde{\mathbf I}_0\, \mathbf g +  \mathbf g^\top \tilde{\mathbf Z}^1 \mathbf g  +  \mathbf g^\top  \tilde{\mathbf Z}_1\, \mathbf g  
 \right),
\end{align*}
where $\tilde{\mathbf I}^0$ is the anti-diagonal matrix with zero at the top-right position of the anit-diagonal and ones on the rest, $\tilde{\mathbf Z}^1$ is the  matrix the matrix that has ones above the anti-diagonal and zero otherwise while  $\tilde{\mathbf Z}_1$ is its anti-diagonal transpose. 
\end{lemma}
\begin{proof}
Since the inner product of the two linear functions $y=(d-a)x+d$ and $y=(b-c)x+b$ over $[0,1]$ is $(db+ac)/3+(dc+ab)/6$ thus we obtain
\begin{align*}
\langle G\rangle^2_2&= \frac{
\sum_{j=1}^{l+1}\left(g_j^2+ g_{j-1}^2\right)
}{
3}+
\frac{\sum_{j=1}^{l+1}g_jg_{j-1}
}{3}\\
&=\frac 1 3 \sum_{j=1}^{l+1} (g_j+g_{j-1})^2,
\\
\langle G, \tilde G \rangle &= \frac{\sum_{j=1}^{l+1}\left(g_jg_{l+1-j}+ g_{j-1}g_{l+2-j}\right)}{3}+\frac{\sum_{j=1}^{l+1}\left(g_jg_{l+2-j}+ g_{j-1}g_{l+1-j}\right)}{6}\\
 &= \frac{\sum_{j=1}^{l+1}g_jg_{l+1-j}+ \sum_{j=0}^{l}g_{j}g_{l+1-j}}{3}+\frac{\sum_{j=1}^{l+1}g_jg_{l+2-j}+ \sum_{j=0}^{l}g_{j}g_{l-j}}{6}
 \\
 &= \frac{2 \sum_{j=0}^{l}g_{j}g_{l+1-j}}{3}+\frac{\sum_{j=1}^{l+1}g_jg_{l+2-j}+ \sum_{j=0}^{l}g_{j}g_{l-j}}{6},
\end{align*}
which is equivalent to the matrix formulation in the lemma. 
\end{proof}
Define a sequence of vectors $\mathbf g^l=(g^l_0,\dots, g^l_{2^l})$, $l=0,1,\dots$  through the following recurrent formula
\begin{align*}
\mathbf g^0&=(\sqrt{3/2},0),\\
\mathbf g^{l+1}&=\frac{ (g^l_0,g^l_1,\dots,g_{2^l}^l,0,\dots,0)
-
\frac{\sqrt{2}}2\langle G_l, \tilde G_l \rangle (g_{2^l}^l, \dots,g_1^l, g_0^l,g^l_1,\dots,g_{2^l}^l) 
}{
\sqrt{1  -\langle G_l, \tilde G_l \rangle ^2}},
\end{align*}
where $\langle G_l, \tilde G_l \rangle$ is the inner product of the splines given by $\mathbf g^l$ and $\tilde{\mathbf g}^l$ as described in \ref{lem:inn}.
Let $S(G_l)$ be a concatenation of $\tilde{G}_l$ and $G_l$ symmetrically around zero and thus defined over the equally spaced knots $(-2^l,\dots,-1,0,1,\dots, 2^l)$ having the values $(g_{2^l}^l, \dots,g_1^l, g_0^l,g^l_1,\dots,g_{2^l}^l)$ at these knots. 
Similarly let $ I_0(G_{l})$ be the extension of  $G$ over the same grid by padding with zero over the knots on the right hand side of $G$. 
Then the above equation can written compactly as
\begin{equation}
\label{eq:gplus}
G_{l+1}=\frac{
I_0(G_{l})-
\frac{\sqrt{2}}2\langle G_l, \tilde G_l \rangle S(G_l)}{\sqrt{1-\langle G_l, \tilde G_l \rangle^2)}}.
\end{equation}
\begin{lemma}
\label{lem:bound}
For $l=0,1,2,\dots$, for $G_l$ satisfying \ref{eq:gplus} and $\|S(G_0)\|=1$, then we have $\|S(G_{l})\|=1$ for all $l\ge 0$ and
$$
\langle G_{l+1},\tilde{G}_{l+1}\rangle = -\frac{2\sqrt{2}-1}2 \frac{\langle G_l, \tilde G_l \rangle^2}{1-\langle G_l, \tilde G_l \rangle^2}.
$$
If $G_0$ corresponds to $\mathbf g^0=(\sqrt{3/2},0)$, then for $l\ge 1$:
\begin{equation}
\label{ineq:bound}
\left |\langle G_{l},\tilde{G}_{l}\rangle \right| 
\le 
\frac{1}{a} 
\left(\frac a4\right)^{2^{l}},
\end{equation}
where
$$
a= 8 \frac{2\sqrt{2}-1}{15}\approx 0.9751611 .
$$
\end{lemma}
\begin{proof}
We note from the previous lemma that $\|S(G_0)\|^2= 2\|G_0\|^2=2\times \frac 1 3 \frac 3 2=1$.
Assume that $\|S(G_k)\|^2=1$ for $k\le l$. 
Then we note that 
\begin{align*}
\| I_0(G_l)\|^2_2&=\| G_l\|^2_2=\|S(G_k)\|^2/2=1/2,\\
\langle I_0(G_l), S(G_l) \rangle &=\langle G_l, \tilde G_l  \rangle 
\end{align*}
so that $\frac{\sqrt{2}}2\langle G_l, \tilde G_l \rangle S(G_l)$ is the orthogonal projection of  $I_0(G_l)$ into  $S(G_l)$ and the orthogonal decomposition of $I_0(G_l)$  takes the form 
$$
I_0(G_{l})=\frac{\sqrt{2}}2\langle G_l, \tilde G_l \rangle S(G_l) 
+ 
I_0(G_{l})-
\frac{\sqrt{2}}2\langle G_l, \tilde G_l \rangle S(G_l)
$$
so that 
$$
\left \|I_0(G_{l})-
\frac{\sqrt{2}}2\langle G_l, \tilde G_l \rangle S(G_l)
\right\|^2_2
=
\frac 1 2\left( 1 - \langle G_l, \tilde G_l \rangle^2 \right).
$$
This implies
\begin{align*}
\|S(G_{l+1})\|^2_2&=2\|G_{l+1}\|_2^2=2\times \frac 12=1.
\end{align*}
Using \ref{eq:gplus} we get 
\begin{align*}
\langle G_{l+1},\tilde{G}_{l+1}\rangle & =
\frac{
\langle I_0(G_{l}), \widetilde{ I_0(G_{l})}\rangle-
\sqrt{2}
\langle G_l, \tilde G_l \rangle
 \langle S(G_l), \widetilde{ I_0(G_{l})}\rangle
 + 
 \frac 1 2\langle G_l, \tilde G_l \rangle^2 \|S(G_l)\|^2
}{
1-\langle G_l, \tilde G_l \rangle^2
}\\
& =
\frac{
-\sqrt{2}
\langle G_l, \tilde G_l \rangle^2 + \frac 1 2\langle G_l, \tilde G_l \rangle^2 
}{
1-\langle G_l, \tilde G_l \rangle^2
}
 =(-\sqrt{2}+1/2)
\frac{
\langle G_l, \tilde G_l \rangle^2
}{
1-\langle G_l, \tilde G_l \rangle^2
}.
\end{align*}
Finally, we see that in the case $\mathbf g^0=(\sqrt{3/2},0)$, $\langle G_0,\tilde{G}_0\rangle=1/4$ and   the sequence $|\langle G_l,\tilde{G}_l\rangle|$ is decreasing since first 
$$
|\langle G_1,\tilde{G}_1\rangle|=\frac{2\sqrt{2}-1}2 \frac{1}{15}<1/4
$$
and than if $|\langle G_j,\tilde{G}_j\rangle|\le |\langle G_{j+1},\tilde{G}_{j+1}\rangle|$ for $j<l$, then
\begin{align*}
|\langle G_{l+1},\tilde{G}_{l+1}\rangle|&=\frac{2\sqrt{2}-1}2 \frac{\langle G_l, \tilde G_l \rangle^2}{1-\langle G_l, \tilde G_l \rangle^2}\\
&\le \frac{2\sqrt{2}-1}2 \frac{\langle G_{l-1}, \tilde G_{l-1} \rangle^2}{1-\langle G_{l-1}, \tilde G_{l-1} \rangle^2}\\
&=|\langle G_{l},\tilde{G}_{l}\rangle|
\end{align*}
 since $x/(1-x)$ is an increasing function.
 Thus for all $l\ge 0$:
 \begin{align}\nonumber
|\langle G_{l+1},\tilde{G}_{l+1}\rangle|&=\frac{2\sqrt{2}-1}2 \frac{\langle G_l, \tilde G_l \rangle^2}{1-\langle G_l, \tilde G_l \rangle^2}\\
\nonumber
&\le \frac{2\sqrt{2}-1}2 \frac{\langle G_l, \tilde G_l \rangle^2}{1-\langle G_0, \tilde G_0 \rangle^2}\\
\label{eq:recb}
&\le 8 \frac{2\sqrt{2}-1}{15}  \langle G_l, \tilde G_l \rangle^2.
\end{align}
The bound in \ref{ineq:bound} for $l=1$ takes form
$$
\left |\langle G_{1},\tilde{G}_{1}\rangle \right|  \le \frac{a}{16}
$$
which holds because 
$$
\langle G_{1},\tilde{G}_{1}\rangle= -\frac{2\sqrt{2}-1}2 \frac{\langle G_0, \tilde G_0 \rangle^2}{1-\langle G_0, \tilde G_0 \rangle^2}=-\frac{2\sqrt{2}-1}2 \frac{1}{15}=-\frac{a}{16}.
$$
Let us assume that \ref{ineq:bound} holds for $l<k$. 
Then, by \ref{eq:recb},
\begin{align*}
|\langle G_{k},\tilde{G}_{k}\rangle|&\le a  \langle G_{k-1}, \tilde G_{k-1} \rangle^2\le a \left(\frac 1 a \left(\frac a 4\right)^{2^{k-1}} \right)^2=\frac 1 a \left(\frac a 4\right)^{2^k},
\end{align*}
which proves the final step of the result.
\end{proof}
\begin{lemma}
\label{lem:iden}
Consider the first order OB-spline 
$$\mathcal{OB}_l=\{OB_{i,l}, i=1,\dots, 2^{N-l-1}\}$$ 
at a fixed level  $l=0,\dots, N$ defined in  \ref{subsec:da} and over equally spaced knots 
$\xi_i=i2^{-N}$, $i=0,\dots,2^N$. 
Moreover, let ${\mathcal{B}}_l=\{\tilde{B}_{j,r}, j=1,\dots, 2^{N-r-1}, r>l\}$,  be the $B$-splines at support levels above $l$ orthogonalized with respect to  $\mathcal{OB}_l$ obtained at the $l$-th step of the algorithm and also additionally normalized.

Then 
$$
OB_{i,l}(t)=2^{l}S(G_l)\left(
2^{2l}
\left(
t-2^{l-N}(2i-1)
\right)
\right),
$$
and 
 $$
\tilde{B}_{j,r}(t)
=
{2^{l}}S(G_l)\left(
2^{2l}
\left(
t-2^{r-N}(2j-1)
\right)
\right).
$$
Moreover, 
\begin{align*}
\langle OB_{i,l},\tilde{B}_{j,r}\rangle= \begin{cases}
\langle G_l,\tilde{G}_l\rangle;& |2^r(2j-1)-2^l(2i-1)|=2^l, \\
0;& \mbox{\it otherwise.}
\end{cases}
\end{align*}
 \end{lemma}
\begin{proof}
It is clear that $OB_{i,l}(t)$ is  defined non-zero  only over $[\xi_{i,l}^L, \xi_{i,l}^R]$ with the center point $\xi_{i,l}^C$ and the structure as discussed in \ref{eq:knots1} and \ref{eq:recst}. 
Consider the change of variables
$$
u=2^{2l}
\left(
t-2^{l-N}(2i-1)
\right)
$$
in
\begin{align*}
\int_{\xi_{i,l}^L}^{\xi_{i,l}^R} OB_{i,l}^2(t)~dt = \int_{-2^l}^{2^l} S(G_l)^2(u)~du=\|S(G_l)\|^2_2=1.
\end{align*}
The formula for the inner product follows by the same change of variables. 
\end{proof}
\begin{proof}[Proof of \ref{prop:appr}]
Let us note that 
\begin{equation*}
\|OB_{i,l+1}-\tilde{B}_{i,l+1}\|^2_2
=2(1-\langle OB_{i,l+1}, \tilde{B}_{i,l+1}\rangle).
\end{equation*}
Further, the orthogonal decomposition in \ref{eq:bigD} yields
\begin{align*}
\tilde{B}_{i,l+1}&=\left(\tilde{B}_{i,l+1}-
\langle \tilde{B}_{r_i,l+1}, OB_{r_i,l}\rangle  OB_{r_i,l}
-\langle 
\tilde{B}_{i,l+1}^j, OB_{r_i+1,l}\rangle  OB_{r_i+1,l}
\right)\\
& +\langle \tilde{B}_{i,l+1}, OB_{r_i,l}\rangle  OB_{r_i,l} +\langle \tilde{B}_{i,l+1}, OB_{r_i+1,l}\rangle  OB_{r_i+1,l},
\end{align*}
where $r_i=2(2i-1)$ and thus
\begin{multline*}
\left\|\tilde{B}_{i,l+1}-\langle \tilde{B}_{i,l+1}, OB_{r_i,l}\rangle  OB_{r_i,l}
-
\langle 
\tilde{B}_{i,l+1}, OB_{r_i+1,l}\rangle  OB_{r_i+1,l}
\right\|^2_2=\\
=1
 -2\langle \tilde{B}_{i,l+1}, OB_{r_i,l}\rangle^2.
\end{multline*}
Thus
\begin{align*}
OB_{i,l+1}
=
\frac{ \tilde{B}_{i,l+1}-
\langle \tilde{B}_{i,l+1}, OB_{r_i,l}\rangle  OB_{r_i,l}
-\langle 
\tilde{B}_{i,l+1}, OB_{r_i+1,l}\rangle  OB_{r_i+1,l}}{
\sqrt{
1
 -2\langle \tilde{B}_{i,l+1}, OB_{r_i,l}\rangle^2
 }
 }
\end{align*}
which leads to
\begin{align*}
\langle OB_{i,l+1},\tilde{B}_{i,l+1} \rangle
&=
\frac{1-
\langle \tilde{B}_{i,l+1}, OB_{r_i,l}\rangle^2  
-\langle 
\tilde{B}_{i,l+1}, OB_{r_i+1,l}\rangle^2}{
\sqrt{
1
 -2\langle \tilde{B}_{i,l+1}, OB_{r_i,l}\rangle^2
 }
 }\\
 &=
\sqrt{
1
 -2\langle \tilde{B}_{i,l+1}, OB_{r_i,l}\rangle^2
 }.
\end{align*}
Using \ref{lem:iden} and \ref{lem:bound} we obtain
\begin{align*}
\langle OB_{i,l+1},\tilde{B}_{i,l+1} \rangle
&=
\sqrt{
1
 -2\langle \tilde{B}_{i,l+1}^j, OB_{r_i,l}^j\rangle^2
 }
 =
\sqrt{
1
 -2\langle G_l,\tilde{G}_l\rangle^2
 }
.
\end{align*}
Consequently
\begin{equation*}
\|OB_{i,l+1}-\tilde{B}_{i,l+1}\|^2_2
=2
\left(1-\sqrt{1-2\langle G_l,\tilde{G}_l\rangle^2}\right)=\frac{4\langle G_l,\tilde{G}_l\rangle^2}{1+\sqrt{1-2\langle G_l,\tilde{G}_l\rangle^2}}.
\end{equation*}
Then the upper bound follows from \ref{ineq:bound}. 

For the case of arbitrary $r>l$, we have by the triangle inequality
\begin{align*}
\|OB_{i,r}-\tilde{B}_{i,r}\|_2&\le \|OB_{i,r}-\tilde{B}^{r-1}_{i,r}\|_2+\|\tilde{B}^{r-1}_{i,r}-\tilde{B}^{r-2}_{i,r}\|_2+ \dots +\\
&+\|\tilde{B}_{i,r}^{l+2}-\tilde{B}_{i,r}^{l+1}\|_2+\|B_{i,r}^{l+1}-\tilde{B}_{i,r}\|_2,
\end{align*}
where $\tilde{B}_{i,r}^{l+s}$, $s=1,\dots r-l-1$ denote the $B$-splines that have the same shape as the orthogonormalized splines $OB_{i,l+s}$ while centered at the same knot as $OB_{i,r}$ and $\tilde{B}_{i,r}$.
By applying the first part of the result but to the level $r$ each term has the corresponding bound formulated for the level $l+s$ instead for $l$ yielding  
\begin{align*}
\|OB_{i,r}-\tilde{B}_{i,r}\|_2&\le  \frac{2\sqrt{2(4-\sqrt{14})}}{a}\left(\left( \frac a4\right)^{2^{l}}+\left( \frac a4\right)^{2^{l+1}}+\dots +\left( \frac a4\right)^{2^{r-1}}\right).
\end{align*}
\end{proof}
\begin{proof}[Proof of \ref{cor:totb}]
We note that 
\begin{multline*}
\left \|
y-\sum_{r=0}^{l} \sum_{i=1}^{2^{N-r-1}} 
 \langle y, OB_{i,r}\rangle OB_{i,r} 
-\sum_{j=l+1}^{N-1} \sum_{i=1}^{2^{N-j-1}} 
\langle y, \tilde{B}_{i,j} \rangle \tilde{B}_{i,j} 
\right \|_2=
\\
=\left \|
\sum_{r=l+1}^{N-1} \sum_{i=1}^{2^{N-r-1}} 
\left (
 \langle y, OB_{i,r}\rangle OB_{i,r} 
-
\langle y, \tilde{B}_{i,r} \rangle \tilde{B}_{i,r} \right )
\right \|_2 \le
\\
\le 
\sum_{r=l+1}^{N-1} \sum_{i=1}^{2^{N-r-1}} 
\left(
 \left | \langle y, OB_{i,r}\rangle \right |
 \left \| OB_{i,r} 
-\tilde{B}_{i,r} \right \|_2
 +
\left |
 \langle y, OB_{i,r}-\tilde{B}_{i,r} \rangle\right |
 \right)
 \le 
 \\
 \le 2 \|y\|_2
\sum_{r=l+1}^{N-1} \sum_{i=1}^{2^{N-r-1}} 
 \left \| OB_{i,r} 
-\tilde{B}_{i,r} \right \|_2
\le
 \\
 \le \frac{2^{N+1}\sqrt{2(4-\sqrt{14})}}{a} 
\left( \frac a4\right)^{2^{l}}
\sum_{r=l+1}^{N-1} \frac{\sum_{s=0}^{r-l-1} \left( \frac a4\right)^{2^{s}}}{2^{r+1}} 
\|y\|_2.
\end{multline*}
where in the last line we used the bound from \ref{prop:appr}. 
\end{proof}
\begin{proof}[Proof of \ref{prop:compeff}]
We consider dyadic case for which $n=k2^{N}-1$ and the number of $B$-splines to be orthogonalized is $n-k+1$. 
For the Gram-Schmidt one sided orthogonalization, say left-to-right, the first $k+1$ splines requires $0,1,\dots, k$ inner products, respectively, with respect to the terms that are on the left-hand-side of each.
The remaining $n-2k$ requires the constant number $k$ of inner products due to disjointness of the supports of $B$-splines that are separated by at least $k$ $B$-splines. The total number of the inner product is thus $(n-2k)k +1+\dots +k =nk-2k^2+k(k+1)/2=nk-3 k^2/2 +k/2$.

For a splinet, orthogonalization of each $k$-tuplet requires $k(k-1)/2$ inner products. Thus the bottom row in the dyadic structure with $N$ rows (containing $2^{N-1}$ $k$-tuplets) requires $2^{N-1}k(k-1)/2$ inner products in the orthogonalization procedure. There are $2^{N-1}-1$ $k$-tuplets above the bottom row that needs to be also orthogonalized in the first run of the recurrence in {\sc Step 4} of \ref{alg:DySp0}. Each of them needs to be orthogonalized only with respect to two $k$-tuplets at the bottom row. Thus each of these not orthogonalized $k$-tuplets requires $2k^2$ inner products for the total $2k^2(2^{N-1}-1)$. 
Consequently the total number of the inner products that is required at this step of the recurrence is
$$
2^{N-1}k(k-1)/2+2k^2(2^{N-1}-1)=2^{N-2}\left(5k^2-k\right)-2k^2.
$$
In each run of the loop the above count for the dyadic structure of the size reduced by one and thus the total count is
\begin{align*}
\sum_{j=1}^{N-1} \left(2^{N-j-1}\left(5k^2-k\right)-2k^2\right)&=\frac{5k^2-k}2  \sum_{j=1}^{N-1}2^j-2(N-1)k^2\\
&=\frac{5k^2-k}2 \left(2^{N-1}-2\right)-2(N-1)k^2\\
&=\frac{5k-1}2 k2^{N-1}-\left({3k^2-k}+2Nk^2\right).
\end{align*}
Thus the stated rate follows from the relations $k2^N=n+1$ and $N=(\log(n+1)-\log k)/2$. 
\end{proof}
\section{The dyadic algorithm in the Hilbert space setting}
\label{app:dyad}
Here, we elaborate the details of \ref{alg:genalg} that is presented in \ref{subsec:genapp}.
The only missing part is full explanation of the output $(\bar{\mathbf A}, \bar{\mathbf H})$ from $\mathcal D(\tilde{\mathbf H}, \tilde N)$. 
Since the part corresponding to the orthonormalization of the `lowest' level in the dyadic structure, i.e.  the columns in $\bar{\mathbf A}_{\cdot,\mathcal J}$ and submatrix $ \bar{\mathbf H}_{\mathcal J, \mathcal J}$ have been already defined, it remains to define the rest of the entries of $\bar{\mathbf A}$ and $ \bar{\mathbf H}$. 

For this let us define a vector of indices $\mathcal I$ as all these in $(1,\dots, \tilde d)$ that are not in $\mathcal J$. 
The undefined entries of  $\bar{\mathbf A}$ are  the columns in $\bar{\mathbf A}_{\cdot,\mathcal I}$. 
It will be convenient to discuss this the terms of a given set linearly independent vectors $\tilde{\mathcal X}=\left\{\tilde x_r, r=1,\dots, \tilde d\right\}$ with a band Gram matrix $\tilde{\mathbf H}$.
Let $\tilde{\mathbf x}_{i,\tilde N-1}=\left\{\tilde x_{(2i-1)k-k+1},\dots, \tilde x_{(2i-1)k}\right\}$, $i=1,\dots, 2^{\tilde N-1}$. 
The columns  $\bar{\mathbf A}_{\cdot,\mathcal J}$ represent the orthonormal vectors 
$$\mathbf y_{i,\tilde N-1}=\mathcal G(\tilde{\mathbf x}_{i,\tilde N-1}),$$
 $i=1,\dots, 2^{\tilde N-1}$ in the basis given by $\tilde{\mathcal X}$. 
Consider  $\mathbf B_{i,\tilde N-1}$  that is defined as the output from \ref{alg:isgso} with the input $\mathbf I_k$ and $\tilde{\mathbf H}_{i,\tilde N-1}$, which is the Gram matrix for $\tilde{\mathbf x}_{i,\tilde N-1}$, i.e. 
$$
\tilde{\mathbf H}_{i,\tilde N-1}=\tilde{\mathbf H}_{(2i-2)k,(2i-2)k}^{k,k},
$$ 
where $\tilde{\mathbf H}_{r,s}^{t.u}=[\tilde h_{r+i,s+j}]_{i=1,j=1}^{t,u}$.
Then   $\mathbf B_{i,N-1}$ represents  $\mathbf y_{i,\tilde N-1}$ in the basis given through $\tilde{\mathbf x}_{i,\tilde N-1}$, i.e. for $i=1,\dots, 2^{\tilde N-1}$, $ m=1,\dots,k$:
$$
 y^m_{i,\tilde N-1}=\sum_{j=1}^k  B^{j,m}_{i,\tilde N-1} \tilde x^j_{i,N-1}=\sum_{j=1}^k  B^{j,m}_{i,\tilde N-1} \tilde x_{(2i-2)k+j}.
$$

Next we orthogonalize the vectors in $\{\tilde x_i,i=1,\dots, d\}\setminus \bigcup_{i=1}^{2^{\tilde N-1}}\tilde{\mathbf x}_{i,\tilde N-1}$ with respect to the space spanned on the `lowest' level vectors $\bigcup_{i=1}^{2^{\tilde N-1}}\mathbf y_{i,\tilde N-1}$.
Namely, for $i=1,\dots,2^{\tilde N-1}-1$, $l=1,\dots,k$:
$$
\bar{x}_{(2i-1)k+l}\stackrel{def}{=}\tilde {x}_{(2i-1)k+l}- \left(\mathbf P_{i}+\mathbf P_{i+1}\right)\tilde {x}_{(2i-1)k+l},
$$ 
where  $\mathbf P_{i}$ is projection to the space spanned by the $k$-tuplet $\mathbf y_{i,\tilde N-1}$.
We observe that 
\begin{align*}
\mathbf P_{i} \tilde {x}_{(2i-1)k+l}&=\sum_{m=1}^k \langle y_{i,\tilde N-1}^m,\tilde{x}_{(2i-1)k+l}\rangle y_{i,\tilde N-1}^m\\
&=\sum_{m=1}^k
\left(\sum_{j=1}^k  B^{j,m}_{i,\tilde N-1} \tilde h_{(2i-2)k+j,(2i-1)k+l}\right)
 y_{i,\tilde N-1}^m
 \\
&=\sum_{m=1}^k
\left(  \mathbf B_{i,\tilde N-1}^\top \tilde{\mathbf H}_{(2i-2)k,(2i-1)k}^{k,k}\right)_{m,l}
\sum_{j=1}^k  B^{j,m}_{i,\tilde N-1} \tilde x_{(2i-2)k+j}
 \\
&=\sum_{j=1}^k  \sum_{m=1}^k
\left(  \mathbf B_{i,\tilde N-1}^\top \tilde{\mathbf H}_{(2i-2)k,(2i-1)k}^{k,k}\right)_{m,l}
B^{j,m}_{i,\tilde N-1} \tilde x_{(2i-2)k+j}
 \\
&=\sum_{j=1}^k 
\left( \mathbf B_{i,\tilde N-1}\mathbf B_{i,\tilde N-1}^\top \tilde{\mathbf H}_{(2i-2)k,(2i-1)k}^{k,k} \right)_{j,l}
\tilde x_{(2i-2)k+j}.
\end{align*}
Similarily,
\begin{align*}
\mathbf P_{i+1} \tilde {x}_{(2i-1)k+l}&=\sum_{m=1}^k \langle y_{i+1,\tilde N-1}^m,\tilde{x}_{(2i-1)k+l}\rangle y_{i+1,\tilde N-1}^m\\
&=\sum_{m=1}^k
\left(\sum_{j=1}^k  B^{j,m}_{i+1,\tilde N-1} h_{2ik+j,(2i-1)k+l}\right)
\sum_{j=1}^k  B^{j,m}_{i+1,\tilde N-1} \tilde  x_{2ik+j}
\\
&=\sum_{j=1}^k 
\left( \mathbf B_{i+1,\tilde N-1}\mathbf B_{i+1,\tilde N-1}^\top \tilde{ \mathbf H}_{2ik,(2i-1)k}^{k,k} \right)_{j,l}
\tilde x_{2ik+j}.
\end{align*}
Thus 
\begin{align*}
\bar{x}_{(2i-1)k+l}=&-\sum_{j=1}^k 
\left( \mathbf B_{i,\tilde N-1}\mathbf B_{i,\tilde N-1}^\top \tilde{\mathbf H}_{(2i-2)k,(2i-1)k}^{k,k} \right)_{j,l}
\tilde x_{(2i-2)k+j}\\
&+\tilde {x}_{(2i-1)k+l} \\
&- \sum_{j=1}^k 
\left( \mathbf B_{i+1,\tilde N-1}\mathbf B_{i+1,\tilde N-1}^\top \tilde{ \mathbf H}_{2ik,(2i-1)k}^{k,k} \right)_{j,l}
\tilde x_{2ik+j}.
\end{align*}
In the matrix representation of $\bar{x}_{(2i-1)k+l}$'s in the basis $\tilde{\mathcal X}$ we define columns in $\bar{\mathbf A}_{\cdot ,(2i-1)k}^{\cdot, k}$ through
\begin{align*}
\bar{\mathbf A}_{(2i-2)k,(2i-1)k}^{k,k}&= -\mathbf B_{i,\tilde N-1}\mathbf B_{i,\tilde N-1}^\top \tilde{\mathbf H}_{(2i-2)k,(2i-1)k}^{k,k},\\
\bar{\mathbf A}_{2ik-1,(2i-1)k}^{,k}&=\mathbf I_k,\\
\bar{\mathbf A}_{2ik,(2i-1)k}^{k,k}&=-\mathbf B_{i+1,\tilde N-1}\mathbf B_{i+1,\tilde N-1}^\top \tilde{\mathbf H}_{2ik,(2i-1)k}^{k,k},
\end{align*}
and zero everywhere else. This defines  $\bar{\mathbf A}_{\cdot, \mathcal I}$ and thus completes the definition of $\bar{\mathbf A}$. 

It remains to define $\bar{\mathbf H}_{\mathcal I, \mathcal I}$, which is done through 
$$
\bar{\mathbf H}_{\mathcal I, \mathcal I}=\left(\bar{\mathbf A}_{\cdot, \mathcal I}\right)^\top \tilde{\mathbf H} \bar{\mathbf A}_{\cdot, \mathcal I}.
$$

\section{Numerical implementation}
\label{sec:ni}
Gain in the efficiencies  by using our methods can be fully utilized only if a proper approach to the spline calculus is taken. 
In particular, the efficiency obtained due to the small supports for the $O$-splines in a splinet. If two splines are having in common only small portion of their supports, then the inner product between them needs to be evaluated only over this common support. 
Thus for splines with small support  evaluations become computationally less demanding. 
To utilize all these features, we have implemented in our computational implementation a spline object that contains information about the range of a spline support and we have  utilized this information in the computational procedures, see \ref{subsec:redsup}. 
The details of our implementation is the subject of the remaining part of this work.

The splines are functional objects that, for a given set of knots and a given order,  form a finite dimensional functional space. 
They can be represented in a variety ways. In \cite{Qin} a general matrix representation was proposed that allows for efficient numerical processing of the operations on the splines. 
This was utilized in Zhou et al. (2012) \cite{Zhou} and Reed \cite{Redd} to represent and orthogonalize $B$-splines that was implemented in the R-package \href{https://CRAN.R-project.org/package=orthogonalsplinebasis}{\it Orthogonal B-Spline Basis Functions}. 
In our approach we propose to represent a spline in a different way.
Namely, we focus on the values of the derivatives at knots and the support of a spline.  
The goal is to achieve better numerical stability as well as to utilize the discussed efficiency of base splines having support only on  small portion of the considered domain.
Our approach constitutes the basis for spline treatment in the package {\tt splinets} that accompanies this work.
In the discussion below we present mathematical foundations of this approach from the perspective of the numerical implementation.

\subsection{Fundamental isomorphic relation}
In our numerical representation of splines, we use the fundamental fact that for a given order, say $k$, and a vector of knot points $\boldsymbol \xi=\left(\xi_0,\dots, \xi_{n+1}\right)$, the splines are uniquely defined by  the values of the first $0,\dots, k$ derivatives at the knots.
Here, by a natural convention that we often use, the $0$ derivative is the function itself. 
The values of derivatives at the knots allow for the Taylor expansions at the knots but they cannot be taken arbitrarily due to the smoothness at the knots. 
Since our computational implementation of the spline algebra fundamentally depends on the relation between the matrix of derivatives values and the splines the restrictions for the derivative matrix needs to be addressed.

For any spline function $S$ of order $k$ over the $\boldsymbol \xi$ we define an object
$$
\mathcal S_0 (S)=\left\{k, \boldsymbol \xi, \mathbf s_0,\mathbf s_1, \dots, \mathbf s_k\right\},
$$
where $\mathbf s_j=(s_{0j},\dots, s_{n+1j})$ is an $n+2$-dimensional  vector (column) of values of the $j^{\rm th}$-derivative of $S$ at the knots given in $\left(\xi_0,\dots, \xi_{n+1}\right)$, $j=0,\dots, k$. 
These columns are kept in a $(n+2)\times (k+1)$ matrix $\mathbf S$. 
\begin{equation}
\label{eq:spmat}
\mathbf S\stackrel{def}{=}\left [ \mathbf s_0 \mathbf s_1 \dots  \mathbf s_k \right]
\end{equation} 
Since the derivative of the $k^{\rm th}$ order is not continuous at the knots and constant between knots, one needs some convention how to keep its values in $\mathbf S$.
There are different ways of doing this and here we consider two alternatives, one sided and symmetric.
In the one sided one consider the RHS (LHS) limits for the $k$-th derivative at the knots. 
In  the symmetric approach one  consider the RHS limits for the LHS half of the knots and the LHS for the RHS half of the knots. 
The second approach leads to more natural treatment of the splines with the zero boundary condition and whenever this version of the matrix of values and derivatives is used the superscript  $s$ is us used in the notation.
More specifically, $\mathbf S^{s}$ has the same first $k$ columns as $\mathbf S$.
However the last column of $\mathbf S^{s}$  is obtained from the last column of  $\mathbf S$ by shifting the zero from the last position to the middle position in the odd $n$ case,   while in the even $n$ case duplicating the value at the position with index $i={k~n/2}$ in the next position and shifting all the remaining values for $i>n/2$ by one and thus eliminating the zero at the last position.

\subsubsection{One sided approach}
The one sided approach will be illustrated by the LHS-to-RHS case and the other version can be treated in an analogous way. 
The value of the $k$th derivative at a knot is consider as the right hand side limit except for the last knot $\xi_{n+1}$ where it is assumed to be equal to zero, as there are no values on the right hand side of $\xi_{n+1}$.
In general,  $\mathbf S$ lies in the $(n+2)(k+1)$ dimensional linear space of $(n+2)\times (k+1)$ matrices.
Howevcer, the matrices corresponding to legitimate splines occupy only a proper subspace that correspond to the $(n+1)(k+1)-kn=k+n+1$ dimensional space of splines, see also \ref{subsec:knbc}. 
This restricted subspace can be expressed by relations that entries of $\mathbf S$ need to satisfy.
For obtaining them explicitly,  we note that the spline in interval $(\xi_i,\xi_{i+1}]$, $i=0,\dots, n$, is given through its Taylor expansion
$$
S(t)=\sum_{l=0}^k \frac{(t-\xi_{i})^l}{l!}s_{i l}. 
$$
Similarly, its $r^{\rm th}$, $r=1,\dots,k$ derivative is given through
$$
S^{(r)}(t)=\sum_{l=0}^{k-r} \frac{(t-\xi_{i})^l}{l!}s_{i r+l}. 
$$
The relations that restrict admissible matrices come from the smoothness conditions that require that all the derivatives up to the order $k-1$ have to be equal at the internal knots.
Consequently, we have additional $k (n+1) $ relations
$$
S^{(r)}(\xi_{i+1})=s_{i+1,r}, ~ i=0,\dots, n, ~ r=0,\dots,k-1,
$$
which  translate to expressions for the entries of $\mathbf S$:
\begin{equation}
\label{eq:addk-1n}
\begin{split}
s_{i+1k-1}&=s_{ik-1}+(\xi_{i+1}-\xi_{i})s_{ik},\\
s_{i+1k-2}&=s_{ik-2}+(\xi_{i+1}-\xi_{i})s_{ik-1}+\frac{(\xi_{i+1}-\xi_{i})^2}{2}s_{ik},\\
\hspace{3mm} &\hspace{3cm} \vdots\\
s_{i+1~0}&=s_{i0}+ (\xi_{i+1}-\xi_{i})s_{i1}+\dots + \frac{(\xi_{i+1}-\xi_{i})^k}{k!}s_{ik},
\end{split}~~~~ i=0,\dots, n.
\end{equation}
Thus from the matrix space dimension $nk+2k+n+2$ by the virtue of the $(n+1)k$ linear equation of \ref{eq:addk-1n} we reduce the dimension to $n+k+2$. 
The final restriction comes from the fact that we always assume $s_{n+1k}=0$ so the dimension of admissible $\mathbf S$ is $n+k+1$, as expected. 

The importance of the derived relations for numerical implementation of the spline objects is two-fold. 
Firstly, they can be used to define splines through specifying their derivatives values.
Secondly, in intense computational applications where computations are performed on the entries of the matrix $\mathbf S$, often due to numerical inaccuracies the matrix entries cease to satisfy \ref{eq:addk-1n}, in which the case some corrections of computational results need to be addressed and the discussed equations can be utilized for the purpose. 

In principle, if one wants to obtain a spline by setting the derivatives at knots, among $nk+2k+n+2$ possible values of derivatives at the knots one can only choose freely $n+k+1$. 
For example, if one sets the values of all $k+1$ derivatives at one of the knots, say, the first one, then the values of the $k$ derivatives at the next knot to it are directly determined and only the right hand side $k$th derivative can be arbitrarily chosen at this knot. 
The same applies to every other knot and thus one needs to set values of the first (or any other) row and the last column which leads to $n+k+1$ values (the last entry in the last column is always zero)  to be set and all the remaining values can be obtained by solving for the remaining entries using \ref{eq:addk-1n}. 

In the practical context of approximating a function, it can be natural to set the values of a spline at all knots (the values of the $0$-order derivative), then there remain only $k-1$ possible values to choose among the rest of derivatives at the knots.
Thus the $n+2$ values of the first column in $\mathbf S$ are given. 
Since the order $k$ of a spline is usually a small number comparing to the number of knots, one could choose these $k-1$ values simply by setting some values of  the first derivatives at selected knots (in the column $\mathbf s_1$).  
All the remaining entries of matrix $\mathbf S$ can be then obtained by solving the above equations for the yet unknown entries. 

To utilize matrix computations, it is convenient to represent the \ref{eq:addk-1n} in a matrix algebra format
\begin{equation}
\label{eq:new}
\boldsymbol \Pi\mathbf S \mathbf P= 
\sum_{i=0}^{n} \mathbf E_i \mathbf S \mathbf R \mathbf A_{\xi_{i+1}-\xi_i}
\end{equation} 
where $\boldsymbol \Pi$ is a $(n+1)\times (n+2)$-matrix, while $\mathbf P$ and $\mathbf R$ are   $(k+1)\times k$ matrices given by

\begin{align}
\label{eq:threem}
\boldsymbol \Pi&= 
\begin{bmatrix} 
-1 &       1 &           0 &  \dots    & 0 \\ 
0  &     -1  &          1  &\ddots    &\vdots  \\
\vdots & \vdots  & \ddots & \ddots & 0 \\ 
0  &    0 &    \dots  &    -1        & 1 
\end{bmatrix},\hspace{3mm}
\mathbf P= 
\begin{bmatrix} 
    1 &           0 &  \dots    & 0 \\ 
      0  &          1  & \ddots   & \vdots  \\
 \vdots  & \ddots & \ddots &0  \\ 
   0 &    \dots  &    0        & 1  \\
   0 & \dots &0 & 0
\end{bmatrix},
\hspace{3mm}
\mathbf R= 
\begin{bmatrix} 
    0&           0 &  \dots    & 0 \\ 
      1  &          0  & \ddots   & \vdots  \\
 \vdots  & \ddots & \ddots &0  \\ 
   0 &    \dots  &    1       & 0  \\
   0 & \dots &0 & 1
\end{bmatrix}.
\end{align}
Moreover, $\mathbf E_i$'s, $i=0,\dots,n$,  are  $(n+1)\times (n+2)$ matrices and, for a number $\alpha$,  $\mathbf A_\alpha$ is a $ k\times k$ is lower triangular Toeplitz matrix, that are given through 
\begin{equation}
\label{eq:A}
\mathbf E_i=
\begin{bmatrix} 
\mathbf 0 \\ 
\vdots  \\ 
\mathbf e_1^T \makebox[0in]{~~~\hspace{20mm}$ \leftarrow i+1$}  \\
\vdots  \\ 
\mathbf 0  
\end{bmatrix}~~~~\hspace{1.2cm},~~~~
\mathbf A_{\alpha}=\begin{bmatrix} 
\alpha  & 0 &0 &  \cdots &0 \\
\frac{\alpha^2}2 & \alpha & 0 &  \cdots &0 \\
\frac{\alpha^3}{3!} & \frac{\alpha^2}2 &\alpha &  \ddots & 0 \\
\vdots & \vdots  & \ddots & \ddots &\vdots \\
\frac{\alpha^k}{k!} & \frac{\alpha^{k-1}}{(k-1)!} & \cdots & \frac{\alpha^2}2 & \alpha \\
\end{bmatrix},
\end{equation}
where $\mathbf e_1$ is   canonical vector  $(1,0,\dots,0)$ treated as a  $(n+2)\times 1$ column.

\subsubsection{Two sided approach}
To preserve the symmetry with respect to the RHS and LHS endpoints, one can alternatively assume that  the value is the RHS $k$th derivative at the knots $\xi_i$, $i \le n/2$ and  the LHS derivative at the knots $\xi_i$, $i \ge n/2 +1$. 
If $n$ is odd, $n=2l+1$ for some integer $l$,  the undefined yet value of $s_{l+1k}$ is assumed to be zero and the LHS and RHS values of the $k$-derivatives at $\xi_{l+1k}$ coincide with $s_{lk}$ and $s_{l+2k}$, respectively. 
If $n$ is even, $n=2l$, then $s_{lk}=s_{l+1k}$. 
If both the knots given in $\boldsymbol \xi$ and the order $k$ are fixed, one can identify a spline $S$ with $\mathbf S$
and thus we write $\mathcal S_0(S)=\mathbf S^s$.
In what follows we write the entries of $\mathbf S^s$ as $s_{ij}$ as in $\mathbf S$ but we assume that the last column is symmetrically modified as previously described.
Next we discuss some additional relations that the entries of  $\mathbf S^s$ have to satisfy for this matrix to correspond to a spline. 

Again, although $\mathbf S^s$ lies in the $(n+2)(k+1)$ dimensional linear space of matrices, the matrices corresponding to legitimate splines occupy only a proper subspace that correspond to the $(n+2)(k+1)-kn-k-1=k+n+1$ dimensional space of splines. 
This restricted subspace can be expressed by linear relations that the entries of $\mathbf S$ need to satisfy.
For obtaining them explicitly,  we note that the spline in interval $(\xi_i,\xi_{i+1}]$, $i=0,\dots, n$, is given through its Taylor expansions
$$
S(t)=\sum_{j=0}^{k-1} s_{i j} \frac{(t-\xi_{i})^j}{j!}+s_{i+\delta_i\,k}\frac{(t-\xi_{i})^k}{k!}=\sum_{j=0}^{k-1} s_{i+1 l} \frac{(t-\xi_{i+1})^j}{j!}+s_{i+\delta_i\,k}\frac{(t-\xi_{i+1})^k}{k!},
$$
where $\delta_i=\mathbb I_{(n/2,n+1]}(i)$ (here $\mathbb I_A$ is the indicator function of a set $A$). 
Similarly, its $r^{\rm th}$ derivative, $r=1,\dots,k$, is given through
\begin{align*}
S^{(r)}(t)&=\sum_{j=0}^{k-r-1} s_{i\, r+l} \frac{(t-\xi_{i})^j}{j!} +s_{i+\delta_i\,k}\frac{(t-\xi_{i})^{k-r}}{(k-r)!}
\\& =
\sum_{j=0}^{k-r-1} s_{i+1r+ l} \frac{(t-\xi_{i+1})^j}{j!}+s_{i+\delta_i\,k}\frac{(t-\xi_{i+1})^{k-r}}{(k-r)!}. 
\end{align*}

In the spirit of the symmetry of the two sided approach we divide the knots into the left and  introduce the following notation for the right half ones
\begin{equation}
\label{eq:RHSk}
\left(\xi_0^{R},\dots, \xi_{m+1}^{R}\right)=\left(\xi_{n+1},\dots, \xi_{n-m}\right),
\end{equation}
where $m=[n/2]$. We observe that for $n$ even we have $\xi_{m}=\xi^{R}_{m+1}$ and $\xi_{m+1}=\xi^{R}_{m}$, or otherwise  $\xi_{m+1}=\xi^{R}_{m+1}$.
Let us also define the RHS portion of the matrix $\mathbf S^s$ as 
$$
s_{ij}^{R}=s_{n+1-i\,j}, ~~ i=0,\dots, m+1, ~~ j=0,\dots, k.
$$
The restrictive relations following from the above Taylor expansions can be split into  the LHS and the RHS knots, for  $i=0,\dots, m, ~~r=0,\dots, k-1$ as follows
\begin{equation*}
\label{eq:addk-1nL}
\begin{split}
s_{i+1r}=\sum_{j=0}^{k-r} \frac{(\xi_{i+1}-\xi_{i})^{j}}{j!}s_{i~j+r},& ~~~s^R_{i+1r}=\sum_{j=0}^{k-r} \frac{(\xi^R_{i+1}-\xi^R_{i})^{j}}{j!}s^R_{i~j+r}.
\end{split}
\end{equation*}
We observe that for even $n$ the relations for the knots $\xi_{m+1}$ and $\xi^R_{m+1}$ are equivalent twice due to the overlap  $\xi_{m}=\xi^{R}_{m+1}$ and $\xi_{m+1}=\xi^{R}_{m}$. 
Consequently, the number of the above equations for even number $n=2m$ is equal to $2(m+1)k-k= nk + k$ yielding the dimension of the matrices satisfying them equal to $(n+2)(k+1)-1-nk - k=n+k+1$ as required (in this case we assumed that $s_{m k}=s_{m+1 k}$). 
Similarly for the odd number $n=2(m-1)$, the total count of equations  is again $2(m+1)k= nk +k$ also yielding the correct dimension of the matrices.  

In the previous matrix notation this can be equivalently written as 
\begin{equation}
\label{eq:LR}
\boldsymbol \Pi\mathbf S^L \mathbf P= 
\sum_{i=0}^{m} \mathbf E_i \mathbf S^L \mathbf R \mathbf A_{\xi_{i+1}-\xi_i}, ~~ \boldsymbol \Pi\mathbf S^R \mathbf P= 
\sum_{i=0}^{m} \mathbf E_i \mathbf S^R \mathbf R \mathbf A_{\xi^R_{i+1}-\xi^R_i},
\end{equation} 
where $\mathbf S^L$ is made of the first $m+1$ rows of $\mathbf S^s$, $\mathbf S^R=[s_{ij}^R]_{i=0,j=0}^{m+1,k}$, except the bottom-right term for each of the matrices is set to zero, $\mathbf P$,  $\mathbf R$ are $(k+1)\times k$ matrices defined earlier in \ref{eq:threem}, $\mathbf A_{\alpha}$ is the $k\times k$ matrix defined in \ref{eq:A},  while  $\boldsymbol \Pi$ and $\mathbf E_i$ are $(m+1)\times (m+2)$ matrices defined as before only with $n$ replaced by $m$. 

\subsubsection{Splines with the boundary conditions}
Symmetric treatment of the $k$th derivative is particularly useful when one approach to the splines by setting boundary conditions. 
If one notices that at the end points $\xi_0$ and $\xi_{n+1}$ there are in total $2k$ values if we exclude the $k$ derivatives at these points, then by setting these values one can only choose $n+1-k$ values from the remaining values in the matrix $\mathbf S^s$. 
For example, one can arbitrarily choose $n+1-k$ out of $n$ values of the $k$th derivative at the internal points and the rest of the matrix $\mathbf S$ is uniquely determined. 
Alternatively, one can choose $n+1-k$ out of $n$ values of the function at the internal points. 
It is natural to express these additional $k-1$ restrictions by splitting the knots to $m-k+1$ LHS and RHS knots and the central knots ($2k$ of them for the even $n$ case and $2k+1$ for odd $n$), the matrix $\mathbf S^s$ into corresponding three parts $\mathbf S^L$, $\mathbf S^R$ and $\mathbf S^C$.

More precisely, let $m=[n/2]-k$ and define the RHS knots as in \ref{eq:RHSk}.
The RHS and LHS matrices $\mathbf S^R$ and $\mathbf S^L$ are defined as before for this new value of $m$ (the lower right corner values are set to zero).
Assuming that at each of the two initial knots we set the derivatives up to the $k-1$st order, there are $m+1$ additional values than can be chosen, for each of the two matrices (for example values of the function at the internal knots and the last knot or, alternatively,  the values of the $k$th derivatives at the internal knots and the first knot). The central knots are defined as the set of two sets of knots defined around the central knots as follows
$$
\left(\xi_0^{CL},\dots, \xi_{k}^{CL}\right)=\left(\xi_{m+1},\dots, \xi_{m+k+1}\right),
~~~ 
\left(\xi_0^{CR},\dots, \xi_{k}^{CR}\right)=\left(\xi_{n-m},\dots, \xi_{n-m-k}\right).
$$ 
We note that for even $n$ the last two central-left knots coincide with the last two central-right. 
The central knots have the corresponding central matrices of the derivatives defined by the entries $s_{ij}^{CL}$, $s_{ij}^{CR}$, $i=0,\dots, k$, $j=0,\dots, k$. 
The complete set of the restrictions that define admissible spline matrices is comprised of \ref{eq:LR} for the newly defined  $\mathbf S^R$ and $\mathbf S^L$  and an analog for the central matrices 
\begin{equation}
\label{eq:CLR}
\widetilde{\boldsymbol \Pi} \mathbf S^{CL} {\mathbf P}
= 
\sum_{i=0}^{k-1} 
\widetilde{\mathbf E_i} \mathbf S^{CL} {\mathbf R} {\mathbf A}_{\xi^{CL}_{i+1}-\xi^{CL}_i}, ~~ 
\widetilde{\boldsymbol \Pi} \mathbf S^{CR} {\mathbf P}
= 
\sum_{i=0}^{k-1} \widetilde{\mathbf E_i} \mathbf S^{CR} {\mathbf R} {\mathbf A}_{\xi^{CR}_{i+1}-\xi^{CR}_i},
\end{equation} 
where $\widetilde{\boldsymbol \Pi}$ and  $\widetilde{\mathbf E_i}$'s are $k \times (k+1)$ matrices  defined as before in \ref{eq:threem} and \ref{eq:A} but with $n$ replaced by $k-1$.

As mentioned earlier our preference is to consider the space of the $k$th order splines for which the derivatives of the order smaller than $k$  are set to zero at the two endpoints.
Therefore, each of the two matrices $\mathbf S^L$ and $\mathbf S^R$ are set by specifying $m+1$ of their entries. 
Then the $k$th derivatives up to the $k-1$st order at each $\xi_0^{CL}$ and $\xi_0^{CR}$ that appear also in $\mathbf S^L$ and $\mathbf S^R$ are set as well. 
Only the $k$th derivative at each these two points is not set by the choice of $\mathbf S^L$ and $\mathbf S^R$.
Thus for each of these two matrices we have $(k+1)^2-k-1=k^2+k$ entries (extra unit is extracted since the lower-right corner is set to zero) further  restricted by \ref{eq:CLR}. 
We observe there are $k^2$ equations in each of the two matrix equations in \ref{eq:CLR}. 
This restriction leaves $2k$ free parameters that are further restricted by $k$ conditions because of the duplication of relations at the central knot(s).
We conclude that one can choose $k$ parameters for the central matrix $\mathbf S^{C}$, for example one can set the derivatives of the order smaller than $k$ at (one of) the central knot(s). 

\begin{example}
In this example we use one-sided notation for the spline matrices of values and derivatives although we deal here with splines with zero boundary conditions. 
Thus the last row of the corresponding matrix is always equal to zero. 
Let consider the zero$^{\rm th}$ and first order $B$-splines with the boundary conditions  $B_{0,l}^{\boldsymbol \xi}$,  $B_{1,r}^{\boldsymbol \xi}$, where $l=0,\dots, n$, $r=0,\dots, n-1$. 
Then the corresponding $(n+2)\times 1$ and $(n+2)\times 2$ matrices are
\begin{equation*}
\mathbf S^{(0,l)}=\begin{bmatrix} 0 \\ \vdots \\ 0 \\ 1 \makebox[0in]{~~~\hspace{14mm}$\leftarrow l+1$} \\ 0 \\ \vdots\\  0
\end{bmatrix} \hspace{1.2cm} ,~~l\le n, \hspace{3mm} 
\mathbf S^{(1,r)}=\begin{bmatrix} 
0 & 0  \\
 \vdots & \vdots \\
  0 &  \frac{1}{\xi_{r+1}-\xi_{r}}\\ 
  1 & \frac{-1}{\xi_{r+2}-\xi_{r+1}} \makebox[0in]{~~~\hspace{28mm}$\leftarrow r+2$, $r < n$.} \\ 
  0  &  0\\
   \vdots & 
  \vdots \\  0 & 0
\end{bmatrix}
\end{equation*} 
\end{example}

Using the recurrent relation \ref{eq:recder} between the derivatives of $B$-splines, one can generalize the recurrence between matrix representation of the $B$-splines seen in the above example to the arbitrary order of splines. 
Using one sided representation of the splines, we have 
\begin{multline}
\label{eq:matrec}
\mathbf S_{\cdot j}^{(k,l)}=\frac{1}{\xi_{l+k}-\xi_l}\left({j}\cdot \mathbf S_{\cdot j-1}^{(k-1,l)} + \boldsymbol \Lambda_l  \mathbf S_{\cdot j}^{(k-1,l)} \right)
+\\
+
\frac{1}{\xi_{l+1}-\xi_{l+k+1}}\left({j}\cdot \mathbf S_{\cdot j-1}^{(k-1,l+1)} + \boldsymbol \Lambda_{l+k+1}  \mathbf S_{\cdot j}^{(k-1,l+1)} \right),
\end{multline} 
where $l=0,\dots, n-k$, $j=0,\dots , k$ and the diagonal $(n+1)\times (n+1)$ matrices $\boldsymbol \Lambda_l$'s have $(\xi_0-\xi_l , \dots , \xi_n-\xi_l)$ on the diagonal.
Here we assume that if $j=k$, then $\mathbf S_{\cdot j}^{(k-1,l)}$ is a column made of zeros as the $k^{\rm th}$ derivatives of the $(k-1)^{\rm th}$ order spline is always zero.

We know that the operation between the splines and the corresponding matrix preserves the linear transformation, i.e. linear combinations of splines correspond to linear combinations of corresponding matrices. These and other basic properties are summarized in the result, the proof of which is obvious. 
\begin{proposition} Let 
$S$ and $\tilde S$ be splines given by the spline objects
\begin{align*}
\mathcal S(S)&=\left\{k, \boldsymbol \xi, \mathbf s_0,\mathbf s_1, \dots, \mathbf s_{k}\right\},~~
\mathcal S(\tilde S)=\left\{\tilde k, \boldsymbol \xi, \tilde {\mathbf s}_0,\tilde {\mathbf s}_1, \dots, \tilde {\mathbf s}_{\tilde k}\right\},
\end{align*}
$\tilde k \le k$. 
Then the spline properties are expressed in the terms of the spline objects as follows
 \begin{itemize}
 \item[i)] Linearity:
  $$
 \mathcal S(\alpha S+\tilde \alpha \tilde S)=\left\{k, \boldsymbol \xi, \alpha \mathbf s_0+ \tilde \alpha \tilde {\mathbf s}_0,\alpha \mathbf s_1+\tilde \alpha \tilde{\mathbf s}_1, \dots,\alpha \mathbf s_{\tilde k}+ \tilde \alpha \tilde{\mathbf s}_{\tilde k}, \alpha \mathbf s_{\tilde k+1},\dots, \alpha \mathbf s_{k}\right\}
 $$
 \item[ii)] Differentiation:
 $$
 \mathcal S(S')=\left\{k-1, \boldsymbol \xi, \mathbf s_1, \dots, \mathbf s_{k}\right\},
 $$
 \item[iii)] Multiplication:
 $$
\mathcal S(S\cdot \tilde S)=\left\{k+\tilde k, \boldsymbol \xi, \mathbf s^p_0,\dots, \mathbf s^p_{k+\tilde k}\right\},
 $$
 where for $u=0,\dots, k+\tilde k$, we have
 \begin{align*}
 \mathbf s_u^p= \left(\sum_{j=0}^u \binom{u}{j} \mathbf s_j \cdot \tilde{\mathbf s}_{u-j} \right)
 \end{align*}
 \end{itemize}
\end{proposition}
\begin{remark}
The formula for the product of two splines involves evaluation of many convolutions. In practical, implementation it is faster to utilize smaller dimension of the matrix representation of a spline and utilize it in computations. Namely, for a product of two splines the values at the knots are directly obtained from taking coordinate-wise product of the first columns in the matrices $\mathbf S_1$ and $\mathbf S_2$. Then one can evaluate one row (for example the first one).  using the convolution the corresponding rows in  $\mathbf S_1$ and $\mathbf S_2$ as in the above result. All the remaining terms can be obtained by solving simple linear relations given in \ref{eq:addk-1n}. 
\end{remark}

Finally,  we consider how the topology induced by the inner product in the space of splines can be expressed in the terms of the coefficients of the matrices. 
\begin{proposition}
\label{prop:topo}
Let $\tilde{\mathcal S}$ be the $ n + k + 1$ dimensional space  of $(n+2)\times (k+1)$ matrices $\mathbf S$ as in \ref{eq:spmat} satisfying \ref{eq:addk-1n} (or, equivalently,  \ref{eq:new}). 
For $\mathbf S,\tilde{\mathbf S} \in \tilde{\mathcal S}$ in this space let us define the inner product 
$$
 \langle \mathbf S, \tilde{\mathbf S}  \rangle 
 =
\begin{pmatrix} 1 & \frac 12 & \dots & \frac 1{2k+1} \end{pmatrix} 
\sum_{i=0}^n  
(\mathbf F_{i\cdot} \cdot \mathbf S_{i\cdot})*({\mathbf F_{i\cdot} \cdot \tilde{\mathbf S}}_{i\cdot}),
$$
where $\mathbf F$ is a matrix with entries $f_{ij}=(\xi_{i+1}-\xi_{i})^{j+1/2}/j!$, $j=0,\dots , k$, $i=0,\dots, n$. Here we use the following notations and conventions:  for two  $r\times 1$ vectors $\mathbf v$ and $\mathbf w$,  their convolution is  a  $(2r-1)\times 1$ vector defined by 
$$
\mathbf v * \mathbf w=
\left(\sum_{m=(p-r+1)\vee 1}^{p\wedge r} v_{p-m+1}w_{m}
\right)_{p=1}^{2r-1}
,
$$ 
 while $\mathbf v \cdot \mathbf w$ is coordinate-wise multiplication of vectors. Moreover, for a matrix $\mathbf X$,  its  $i$th row is denoted by $\mathbf X_{i\cdot}$. 
 
Then $\tilde{\mathcal S}$ equipped with this inner product is isomorphic with the space of splines of the $k$th order spanned over the knots $\xi_0,\dots, \xi_{n+1}$ equipped with the standard inner product of the square integrable functions.  
\end{proposition}

\begin{proof}
For a given set of knots $\boldsymbol \xi$ let $S$ and $\tilde S$ be the (unique) splines such that $\mathcal S_0 (S)= \mathbf S$ and $\mathcal S_0 (\tilde S)=\tilde{\mathbf S}$.
Then 
\begin{align*}
\langle S, \tilde{S}\rangle 
&= 
\sum_{i=0}^n \int_{\xi_{i}}^{\xi{i+1}} S(t) \tilde S(t)~ dt\\
&=
\sum_{i=0}^n \int_{\xi_{i}}^{\xi{i+1}} 
\sum_{j=0}^k s_{ij} \frac{(t-\xi_i)^j}{j!}
\sum_{j=0}^k \tilde s_{ij} \frac{(t-\xi_i)^j}{j!}
~ dt
\\
&=
\sum_{i=0}^n  
\sum_{j,r=0}^k \frac{s_{ij}\tilde s_{ir}}{j!r!} \int_{\xi_{i}}^{\xi{i+1}} (t-\xi_i)^{j+r} ~dt
\\
&=
\sum_{i=0}^n  
\sum_{j,r=0}^k \frac{s_{ij}\tilde s_{ir}}{j!r!} \frac{(\xi_{i+1}-\xi_i)^{j+r+1}}{j+r+1}
\\
&=
\sum_{i=0}^n  
\sum_{l=0}^{2k}\frac{(\xi_{i+1}-\xi_i)^{l+1}}{l+1}\sum_{m=(l-k)\vee 0}^{l\wedge k} \frac{s_{il-m}\tilde s_{im}}{(l-m)!m!}
\\
&=
\sum_{i=0}^n  
\sum_{l=0}^{2k}
\frac{1}{l+1}
\sum_{m=(l-k)\vee 0}^{l\wedge k} 
\frac{
 s_{il-m}(\xi_{i+1}-\xi_i)^{l-m+1/2}
}{
(l-m)!
}
\frac{
\tilde s_{im}(\xi_{i+1}-\xi_i)^{m+1/2}
}{
m!
},
\end{align*}
which show the isometry property of the mapping $S_0$. 
\end{proof}

\subsection{Reduced support} \label{subsec:redsup}
For our bases of splines, frequently the support is contained only in some in between knots intervals and outside the support the values of the derivatives are zero. 
To utilize this in efficient computations, we modify the notation and representation of a spline
$$
\mathcal S(S)=\left\{k, \boldsymbol \xi, i, m, \mathbf s_0,\mathbf s_1, \dots, \mathbf s_k\right\},
$$
where $[\xi_i, \xi_{i+m+1}]$ is the support of $S$ and  $\mathbf s_0,\mathbf s_1, \dots, \mathbf s_k$ are $m+2$ dimensional vectors of values of the $j$-derivative of $S$ at the knots given in $\left(\xi_i,\dots, \xi_{i+m+1}\right)$, $j=0,\dots, k$. 
Let us introduce an $(m+2)\times (k+1)$ matrix $\mathbf S=[\mathbf s_0 \mathbf s_1  \dots \mathbf s_k]$.
We consider shorter notation if the usually fixed $\boldsymbol \xi$ and $k$ that is implicitly given in $\mathbf S$ are both removed from the notation
$$
\mathcal S(S)=\left\{i, m,\mathbf S \right\}.
$$

The actual form of the spline $S$ within each interval, i.e. polynomials of the order of the spline, can be easily obtained by Taylor's expansions at either of the endpoints. 
Namely, if $x\in [\xi_{i+r},\xi_{i+r+1}]$, $r=0,\dots, m$, then 
$$
S(x)=\sum_{l=0}^k  \frac{(x-\xi_{i+r})^l}{l!} s_{rl} =\sum_{l=0}^k  \frac{(x-\xi_{i+r+1})^l}{l!} s_{r+1l}.
$$

We note the following properties of this representation of splines that parallel the ones obtained in the previous secton.
\begin{proposition} Let 
$S$ and $\tilde S$ be splines given by the spline objects
\begin{align*}
\mathcal S(S)&=\left\{i, m, \left[\mathbf s_0 \mathbf s_1 \dots \mathbf s_{k}\right] \right\},~~
\mathcal S(\tilde S)=\left\{ \tilde i, \tilde m, \left[ \tilde {\mathbf s}_0\tilde {\mathbf s}_1 \dots \tilde {\mathbf s}_{k}\right]\right\}.
\end{align*}
Then the spline properties are expressed in the terms of the spline objects as follows
 \begin{itemize}
 \item[i)] Linearity:
  $$
 \mathcal S(\alpha S+\tilde \alpha \tilde S)=\{\bar i,\bar m,\bar {\mathbf S}\},
 $$
 where $\bar i= i \wedge \tilde i$, $\bar m=\left( (m +i) \vee (\tilde m+\tilde i)\right) -\bar i$ and a $(\bar m +2)\times (k+1)$ matrix of the derivatives at knots is given by 
 $$
\bar{\mathbf S}=\mathbf S_0 + \tilde{\mathbf S}_0,
 $$
 where  $(\bar m +2)\times (k+1)$ matrices $\mathbf S_0$ and  $\tilde{\mathbf S}$ are extended from $\mathbf S$ and $\tilde{\mathbf S}$ by having rows of zeros at these indexes $i \in [\bar i,\bar i+\bar m+1]$ where the derivatives of $S$ or $\tilde{ S}$, respectively, are equal to zero at $\xi_i$. 
 \item[ii)] Differentiation:
 $$
 \mathcal S(S')=\left\{i,m,\left[ \mathbf s_1 \dots \mathbf s_{k}\right]\right\},
 $$
 \item[iii)] Multiplication:
 $$
\mathcal S(S\cdot \tilde S)=\left\{\bar i,\bar m,  \mathbf s^p_0,\dots, \mathbf s^p_{2k}\right\},
 $$
 where $\bar i= i\vee \tilde i$ and $\bar m = \left( (m +i) \wedge (\tilde m +\tilde i) \right) -\bar i$, while for $u=0,\dots, 2k$, we have
 \begin{align*}
 \mathbf s_u^p= \left(\sum_{j=0}^u \binom{u}{i} \mathbf s_j^- \cdot \tilde{\mathbf s}^-_{u-j} \right),
 \end{align*}
 where $\mathbf s_j^-$'s and $\tilde{\mathbf s}_j^-$'s are $\bar m+2$ column vectors obtained from $\mathbf s_j$'s and $\tilde{\mathbf s}_j$' by removing these entries for which either the corresponding entry in $\mathbf s_j$ or in $\tilde{\mathbf s}_j$ is equal to zero. 
 \item[iv)] Inner product: \\
For  the standard inner product of the square integrable functions, the notation of \ref{prop:topo} and of the previous item, we have
$$
 \langle S, \tilde{S}  \rangle 
 =
\begin{pmatrix} 1 & \frac 12 & \dots & \frac 1{2k+1} \end{pmatrix} 
\sum_{i=0}^{\bar m}  
(\mathbf F_{\bar i + i\cdot} \cdot \mathbf S^-_{i\cdot})*({\mathbf F_{\bar i + i\cdot} \cdot \tilde{\mathbf S}}^-_{i\cdot}),
$$
with $(\bar m+ 2) \times (k+1)$ matrices   $\mathbf S^-=\left[ \mathbf s_0^- \mathbf s_1^- \dots \mathbf s_k^-\right]$ and  $\tilde{\mathbf S}^-=\left[ \tilde{\mathbf s}_0^- \tilde{\mathbf s}_1^- \dots \tilde{\mathbf s}_k^-\right]$.
 \end{itemize}
\end{proposition}


\begin{thebibliography}{10}

\bibitem{BogdanL}
{\sc M.~Bogdan and T.~Ledwina}, {\em Testing uniformity}, Statistics, 28
  (1996), pp.~131--157.

\bibitem{cho2005class}
{\sc O.~Cho and M.~J. Lai}, {\em A class of compactly supported orthonormal
  b-spline wavelets}, Splines and Wavelets,  (2005), pp.~123--151.

\bibitem{Boor1978APG}
{\sc C.~de~Boor}, {\em A practical guide to splines}, in Applied Mathematical
  Sciences, 1978.

\bibitem{goodman2003class}
{\sc T.~N. Goodman}, {\em A class of orthogonal refinable functions and
  wavelets}, Constructive approximation, 19 (2003), pp.~525--540.

\bibitem{lowdin}
{\sc P.-O. L{\"o}wdin}, {\em Quantum theory of cohesive properties of solids},
  Advances in Physics, 5 (1956), pp.~1--171.

\bibitem{mason1993orthogonal}
{\sc J.~Mason, G.~Rodriguez, and S.~Seatzu}, {\em Orthogonal splines based on
  b-splines -- with applications to least squares, smoothing and regularisation
  problems}, Numerical Algorithms, 5 (1993), pp.~25--40.

\bibitem{Mason}
{\sc J.~Mason, G.~Rodriguez, and S.~Seatzu}, {\em Orthogonal splines based on
  $b$-splines- with applications to least squares, smoothing and regularisation
  problems}, Numerical Algorithms, 5 (1993), pp.~25--40.

\bibitem{nguyen2015construction}
{\sc T.~Nguyen}, {\em Construction of spline type orthogonal scaling functions
  and wavelets},  (2015).

\bibitem{Qin}
{\sc K.~Qin}, {\em General matrix representations for $b$-splines}, Vis.
  Comput., 16 (2000), pp.~177--186.

\bibitem{Redd}
{\sc A.~Redd}, {\em A comment on the orthogonalization of $b$-spline basis
  functions and their derivatives}, Stat. Comput, 22 (2012), pp.~251--257.

\bibitem{schumaker2007spline}
{\sc L.~Schumaker}, {\em Spline functions: basic theory}, Cambridge University
  Press, 2007.

\bibitem{Zhou}
{\sc L.~Zhou, J.~Huang, and R.~Carroll}, {\em Joint modeling of paired sparse
  functional data using principal components.}, Biometrika, 95 (2008),
  pp.~601--619.


\end{thebibliography}
\end{document}